\documentclass[sn-mathphys-num]{sn-jnl}

\usepackage{graphicx}%
\usepackage{amsmath,amssymb,amsfonts}%
\usepackage{amsthm}%
\usepackage{xcolor}%
\usepackage{algorithm}%
\usepackage[all]{xy}

\newcommand\tP{{\mathcal P}}
\newcommand\tQ{{\mathcal Q}}
\newcommand{\tX}{{\mathcal X}}

\newcommand\C{{\mathcal C}}
\newcommand\F{{\mathcal F}}
\newcommand\G{{\mathcal G}}
\newcommand\tS{{\mathcal S}}
\newcommand{\Oo}{\mathcal O}

\newcommand{\Ff}{\mathcal F}
\newcommand{\Ee}{\mathcal E}
\newcommand{\Ll}{\mathcal L}

\newcommand{\Mmm}{\mathcal M}

\newcommand{\N}{\mathcal N}
\newcommand{\D}{\mathcal D}

\newcommand{\Rr}{\mathrm R}
\newcommand{\subsetq}{\subset}
\newcommand{\Spec}{\mathrm{Spec}}
\newcommand{\Tor}{\mathrm{Tor}}
\newcommand{\RR}{\mathrm R}

\newcommand{\rep}{{\mathrm{rep}}}
\newcommand{\Char}{{\mathrm{Char}}}
\newcommand{\ind}{{\mathrm{ind}}}
\newcommand{\hgt}{{\mathrm {ht}}}
\newcommand{\res}{{\mathrm{res}}}
\newcommand{\triv}{{\mathrm{triv}}}
\newcommand{\Rind}{{\mathrm{Rind}}}
\newcommand{\hull}{{\mathrm{hull}}}
\newcommand{\id}{{\mathrm{id}}}

\newcommand{\pre}{{\mathrm{pre}}}
\newcommand{\Ext}{{\mathrm{Ext}}}
\newcommand{\Rep}{{\mathrm{Rep}}}

\newcommand{\fr}{{\mathrm {fr}}}
\newcommand*{\Dd}{\mathop{\mathrm D\kern0pt}\nolimits}

\newcommand*{\Hom}{\mathop{\mathrm {Hom}}\nolimits}
\newcommand*{\RHom}{\mathop{\mathrm {RHom}}\nolimits}

\newcommand*{\Coh}{\mathop{\sf {Coh}}\nolimits}
\newcommand{\QCoh}{\mathsf{QCoh}}
\newcommand{\coker}{{\mathsf{coker}}}
\newcommand{\cone}{{\mathsf{cone}}}
\newcommand{\kernel}{{\mathsf{ker}}}

\newcommand{\A}{{\sf A}}

\newcommand{\sk}{{\sf k}}
\newcommand{\sK}{{\sf K}}
\newcommand{\sX}{{\sf X}}
\newcommand{\sH}{{\sf H}}

\newcommand{\onto}{{\twoheadrightarrow}}
\newcommand{\Gm}{\mathbb G}
\newcommand{\Bm}{\mathbb B}
\newcommand{\Tm}{\mathbb T}
\newcommand{\Pm}{\mathbb P}

\newcommand{\Pp}{\mathbb P}

\newcommand{\Z}{{\mathbb Z}}
\newcommand{\cC}{{\mathbb C}}

\newcommand{\kmod}{\sk\textrm{-}\mathrm{mod}}
\newcommand{\kMod}{\sk\textrm{-}\mathrm{Mod}}
\newcommand*{\vectk}{\mathop{\sf {vect}\textrm-\sk}\nolimits}

\newcommand{\Cc}{\mathbf C}
\newcommand{\bg}{\mathbf G}
\newcommand{\fb}{\mathbf B}
\newcommand{\bt}{\mathbf T}
\newcommand{\bp}{\mathbf P}
\newcommand{\bh}{\mathbf H}
\newcommand{\bu}{\mathbf U}
\newcommand{\bl}{\mathbf L}

\newcommand{\fX}{{ X}}

\def\text#1{\mbox{#1}}

\newtheorem{theorem}{Theorem}[section]
\newtheorem{corollary}[theorem]{Corollary}
\newtheorem{lemma}[theorem]{Lemma}
\newtheorem{proposition}[theorem]{Proposition}
\newtheorem{conjecture}[theorem]{Conjecture}
\theoremstyle{definition}
\newtheorem{definition}[theorem]{Definition}%
\newtheorem{example}[theorem]{Example}
\newtheorem{remark}[theorem]{Remark}
\newtheorem{aside}[theorem]{Aside}

\newtheorem{notation}[theorem]{Notation}
\long\def\comment#1{}

\begin{document}
\setlength{\baselineskip}{12.5pt}\topsep=.5\baselineskip

\title[Highest weight category structures and full exceptional collections]{Highest weight category structures on rep($\mathbf B$) and full exceptional collections on generalized flag varieties over $\mathbb Z$}

\author*[1]{\fnm{Alexander} \sur{Samokhin}}\email{alexander.samokhin@math.uni-bielefeld.de}

\author[2]{\fnm{Wilberd} \spfx{van der}\sur{Kallen}}\email{W.vanderKallen@uu.nl}


\affil[1]{\orgdiv{Fakult\"at f\"ur Mathematik}, \orgname{Universit\"at Bielefeld}, \orgaddress{
\postcode{33501}
\city{Bielefeld},  \country{Germany}}}

\affil[2]{\orgdiv{Mathematical Institute}, \orgname{Utrecht University}, \orgaddress{\street{P.O.Box 80.010}, \postcode{3508 TA} \city{Utrecht},  \country{The Netherlands}}}

\abstract{\unboldmath Given a split simply connected semi-simple algebraic group scheme $\Gm$ over $\Z$, a split maximal
torus $\Tm$ in $\Gm$
and a parabolic subgroup scheme $\Pm\supset \Tm$, this paper constructs
semi-orthogonal decompositions of the bounded derived category $\Dd ^b(\rep( \Pm))$ of noetherian representations of $\Pm$ with each semi-orthogonal component being equivalent to the bounded derived category $\Dd ^b(\rep( \Gm))$ of noetherian representations of $\Gm$. 
The semi-orthogonal components of those decompositions are stable under the monoidal action of $\Dd ^b(\rep( \Gm))$ on $\Dd ^b(\rep( \Pm))$. The decompositions depend on an arbitrarily chosen total order on the Weyl group that refines the Bruhat order.
The  semi-orthogonal decompositions are also compatible with the Bruhat order on cosets of the Weyl group of $\Pm$ in the Weyl group of $\Gm$. 
Their construction builds upon the foundational results on $\Bm$-modules from the works of Mathieu, Polo, and van der Kallen, and upon properties of the Steinberg basis of the $ \Tm$-equivariant $\sK$-theory of $ \Gm/\Bm$. As a corollary, we obtain full exceptional collections in the bounded derived category of coherent sheaves on generalized flag schemes $\Gm/\Pm$ over $\mathbb Z$.}




\maketitle

\tableofcontents

\section{\unboldmath Introduction}

\subsection{\unboldmath The context}
Let  $\Gm$ be a split simply connected  algebraic group scheme over $\Z$,  with terminology as in \cite[II 1.1, II 1.6]{Jan}. Let $\Tm\subset \Gm$ be a split maximal torus and let $\Bm\supset \Tm$ be a Borel subgroup scheme. 
Associated to it is the flag scheme $\Gm/\Bm$ over $\mathbb Z$. More generally, let $\Pm\subset \Gm$ be a parabolic subgroup scheme containing $\Bm$. Generalized flag varietes for the group scheme $\Gm$ are the quotient schemes
$\Gm/\Pm$.

The goal of this paper is to study $\Dd ^b(\rep (\Bm))$,
the bounded derived category of noetherian representations of $\Bm$,  in terms of the monoidal action of 
$\Dd ^b(\rep (\Gm))$ on $\Dd ^b(\rep (\Bm))$ (and, more generally, to study 
$\Dd ^b(\rep (\Pm))$ in terms of the monoidal action of 
$\Dd ^b(\rep (\Gm))$ on $\Dd ^b(\rep (\Pm))$).
We will use the terminology $\Gm$-linear to indicate $\Dd ^b(\rep (\Gm))$-linearity (Subsection \ref{subsec:G-linear-sod}).
Our main results, which are Theorems \ref{th:semi-orthogonal} and \ref{th:Psemi-orthogonal}, construct $\Gm$-linear semi-orthogonal decompositions
of $\Dd ^b(\rep (\Bm))$ (resp., $\Gm$-linear semi-orthogonal decompositions of $\Dd ^b(\rep (\Pm))$) that can be considered as categorifications of the classical results 
\cite{BGG} and \cite{Demres}. As a corollary to those theorems, we obtain in Theorem \ref{th:coh-semi-orthogonal} (resp., in Theorem \ref{th:Pcoh-semi-orthogonal}) full exceptional collections in the bounded derived category of coherent sheaves on $\Gm/\Bm$ over $\mathbb Z$ (resp., in the bounded derived category of coherent sheaves on $\Gm/\Pm$). 

We believe that the categorical decompositions of $\Dd ^b(\rep (\Bm))$ will have further applications in geometric representation theory (see Section \ref{sec:disc} below). Chronologically, we arrived at Theorems \ref{th:semi-orthogonal} and \ref{th:Psemi-orthogonal} starting off with the question whether full exceptional collections on generalized flag varietes can be obtained using representation theory of a Borel subgroup scheme $\Bm$. We now explain this path in a greater detail.

The study of exceptional collections on generalized flag varieties has a long history; more recently, the paper \cite{KuzPol} set out an approach that influenced many works on the subject. For a very recent and comprehensive survey of the literature that followed {\it loc.cit.}, we refer the reader to \cite{Fon} and the references therein. In this non-technical part of the introduction we emphasise some of the features that make the approach of the present paper different from the previous constructions.

Given a smooth proper Noetherian scheme over a field $\sk$, a full exceptional collection in the bounded derived category of coherent sheaves $\Dd ^b(\Coh(X))$ is, informally speaking, a way to break up $\Dd ^b(\Coh(X))$ into elementary pieces, each equivalent to $\Dd ^b(\vectk)$ which are glued to each other in a non-trivial way. 
Upon decategorification, a full exceptional collection gives rise to a basis of the Grothendieck group ${\sK}_0(X)$.
 If one is interested only in the $\sK$-theoretic information, for many schemes of interest it is easy to produce a basis of ${\sK}_0(X)$: for instance, if a scheme $X$ has a ${\mathbb G}_m$-action with isolated fixed points then, by Bia\l{}ynicki-Birula's theorem, $X$ has a stratification into locally closed subschemes such that the structure sheaves of their closures are a basis of ${\sK}_0(X)$. 
 Flag schemes $\Gm/\Bm$ furnish a classical example of such a stratification: in this case, the closures of algebraic cells of the  Bia\l{}ynicki-Birula's stratification are Schubert schemes via the Bruhat decomposition. 

Given a scheme $X$ as above and assuming it has a natural geometric basis of the group ${\sK}_0(X)$, one can ask whether such a basis can be lifted to a full exceptional collection in $\Dd ^b(\Coh(X))$. 
Speaking less loosely, a full exceptional collection must satisfy two conditions: it should generate the category $\Dd ^b(\Coh(X))$ in a suitable sense and some cohomological vanishing should hold. 
Generation doesn't pose a problem: for schemes as above, \emph{i.e.}\ having an algebraic stratification into locally closed subschemes, which are isomorphic to affine spaces, the direct sum of structure sheaves of their closures is a classical generator of $\Dd ^b(\Coh(X))$. 
But there is no chance that for such a set of generators the cohomological vanishing (known as ``semiorthogonality'' condition) would hold. 
On the other hand, there is more to flag schemes that one would want to take into account: there is a partial order on Schubert schemes (the Bruhat order) and the order on the sought-for full exceptional collection is expected to be compatible in some sense with the Bruhat order on Schubert schemes. 
Thus, what one is after is finding a correct ``lift''  to the derived category of a  $\sK$-theoretic basis that 
involves the Bruhat order. 
In the present paper, we provide such a lift. 
As explained below, there are two major inputs into our approach: one is representation-theoretic and comes from the highest weight category structures on the category of rational representations of the Borel subgroup scheme $\Bm$, and the other one is combinatorial and comes from the  Steinberg basis, \cite{St}, which is a distinguished basis of the ($\mathbb \Tm$-equivariant) $\sK$-group of
the flag scheme $\Gm/\Bm$, where $\Tm$ is a split maximal torus in $\Bm$.

{There is a basic way to relate algebraic representations of $\Bm$ to $\Gm$-equivariant vector bundles on $\Gm/\Bm$ via the ``associated sheaf'' construction.}
Note that the latter construction was the starting point in \cite{KuzPol}.

Now let $\sk$ be a field.
Let $\fb$ (resp., $\bg$) denote the group schemes over $\sk$ obtained  
by base change from $\Bm$ (resp., $\Gm$) along ${\rm Spec}(\sk)\rightarrow {\rm Spec}(\mathbb Z)$.
The category $\rep(\bg)$ of rational modules over the group scheme $\bg$ that are finite-dimensional over $\sk$ has a highest weight category structure, \cite{CPSHW}, \cite {Donkin1}, and it is this  structure that is expected to manifest itself through conjectural full exceptional collections. 
In this respect, a model example of such semi-orthogonal decompositions/full exceptional collections is given by a theorem of Efimov's \cite{Ef} that solves the problem for the Grassmannians ${\sf Gr}_{k,n}$
 in the best possible way. 
Precisely, one has (cf. also \cite{BLVdB} for a different approach over a field $\sk$):

\begin{theorem}[{\cite[Theorem 1.8]{Ef}}]\label{th:Ef}
Let  ${\sf Gr}_{k,n}$ be the Grassmannian of $k$-dimensional vector subspaces of $n$-dimensional space, defined over $\mathbb Z$. 
There exists a tilting vector bundle on  ${\sf Gr}_{k,n}$ such that its endomorphism algebra has two natural structures of a split quasi--hereditary algebra over $\mathbb Z$.
\end{theorem}

The above theorem \ref{th:Ef} and, more generally, results of \cite{Ef} and \cite{BLVdB} rely considerably on representation theory of the full general
linear group $GL_n$ and on the highest weight category structure on $\rep(GL_n)$.
Remarkably, through the works \cite{Math}, \cite{Polo3}, \cite{WvdK} and \cite{WvdKTata}, highest weight category structures have also been known to exist on the category $\rep( \fb)$ for the Borel subgroup $ \fb\subset  \bg$ when $\sk$ is a field. Here $\rep( \fb)$ stands for the category of rational $\fb$-modules that are finite-dimensional over $\sk$.
It seems, however, that the category $\rep( \fb)$ as a highest weight category has not yet received the full attention that it undoubtedly deserves. 
The following quote from Stephen Donkin, \cite[Remark 4.7]{Donsur} has served as a particular impetus for the first named author: ``In fact ${\mathsf k}[{\fb}]$ is a quasi--hereditary coalgebra. 
This follows from van der Kallen's paper \cite{WvdK}. 
This is a deep and sophisticated work which generalizes the above (restriction gives a full embedding of the category of $ \bg$-modules into the category of $\fb$-modules) and whose consequences, to the best of my knowledge, have so far not been investigated or exploited.''

In the present paper, we demonstrate the force of highest weight category structures on the category $\rep(\fb)$  and, in accordance with Donkin's suggestion, show some of its consequences.
Let $\sk$ be a field or $\Z$. 
If $\sk=\Z$ this just means that $\bg=\Gm$, $\fb=\Bm$, $\bp=\Pm$.
Our two main theorems are:

\begin{theorem}[cf.\ Theorem \ref{th:semi-orthogonal}]\label{th:1}
The category $\D=\Dd^b (\rep ({\fb}))$ has a $ \bg$-linear semi-orthogonal decomposition 
\begin{equation}
\D= \langle {\sX} _v\rangle _{v\in W}
\end{equation}
with respect to   a  total order $\prec$ on the Weyl group $W$,  arbitrarily chosen so that it refines the Bruhat order. Each subcategory $\sf X _v$ is equivalent to $\Dd ^b(\rep( \bg))$.
\end{theorem}

\begin{remark}
It may happen that a subcategory ${\sf X} _v$  gets replaced with a different subcategory when one changes the choice of the total order $\prec$. This happens for instance in types $\fb_3$, $\Cc_3$, $\mathbf A_4 $, cf.\ Subsection \ref{subsec:var}.
\end{remark}

As it will become clear, the case of Borel subgroups is the most important. With all the setup developed for proving Theorem \ref{th:1}, we are in a position to prove the parabolic version of it. 
Let $\bp$ be a parabolic subgroup containing $\fb$. 
Let $W_\bp$ be the parabolic Weyl group corresponding to $\bp$, and let $W^\bp$ be the set of minimal coset representatives of $W/W_\bp$. Let $\prec _{\bp}$ denote the restriction to $W^\bp$ of the chosen total order $\prec$ on $W$ from Theorem \ref{th:1} above. Then:

\begin{theorem}[cf.\ Theorem \ref{th:Psemi-orthogonal}]\label{th:2}
The category $\D=\Dd^b (\rep ({\bp}))$ has a $ \bg$-linear semi-orthogonal decomposition 
\begin{equation}
\D= \langle {\hat{\sX}} _v\rangle _{v\in W^\bp}
\end{equation}
with respect to the order $\prec _{\bp}$ on  $W^\bp$. Each subcategory ${\hat{\sX}} _v$ is equivalent to $\Dd ^b(\rep( \bg))$.
\end{theorem}

Theorems \ref{th:1} and \ref{th:2} have some immediate applications.

\begin{theorem}[Theorem \ref{th:coh-semi-orthogonal}]\label{th:3}
Let $\prec$ be the same total order on $W$ as in Theorem \ref{th:1}, 
and let $\D=\Dd^b (\Coh( \bg/\fb))$. Let $v,w$ denote elements of  $W$. Then there are objects $\tX_v\in \D$ such that 
\begin{enumerate}
\item $\Hom_\D(\tX_v,\tX_v[i])=\begin{cases}\sk &\text{if $i=0$}\\
0&\text{else.}\end{cases}$.

\

\

\item If $w\succ  v$ then $\Hom_\D(\tX_v,\tX_w[i])=0$ for all $i$.\\
\item The triangulated hull of $\{\tX_v\mid v \in W\}$ is $\D$.
\end{enumerate}
In other words, the collection of objects $(\tX_v)_{v \in W}$ is a full exceptional collection in $ \D$.
\end{theorem}

\

The next Theorem \ref{th:4} is in the same position with respect to the previous Theorem \ref{th:3} as Theorem \ref{th:2} is with respect to  Theorem \ref{th:1}: 

\

\begin{theorem}[Theorem \ref{th:Pcoh-semi-orthogonal}]\label{th:4}
Let $\prec _{\bp}$ be the same total order on $W^{\bp}$ as in Theorem \ref{th:2}, and let $\D=\Dd^b (\Coh( \bg/\bp))$. Let $v,w\in W^\bp$. Then there are objects $\hat \tX_v\in \D$ such that 
\begin{enumerate}
\item $\Hom_\D(\hat\tX_v,\hat\tX_v[i])=\begin{cases}\sk &\text{if $i=0$},\\
0&\text{else.}\end{cases}$\\

\item If $w\succ _{\bp}v$ then $\Hom_\D(\hat\tX_v,\hat\tX_w[i])=0$ for all $i$.\\
\item The triangulated hull of $\{\hat\tX_v\mid v \in W^\bp\}$ is $\D$.
\end{enumerate}
In other words, the collection of objects $(\hat\tX_v)_{v \in W^\bp}$ is a full exceptional collection in $ \D$.
\end{theorem}

In Sections \ref{sec:FEConGmodB} and \ref{sec:caseP} we explain how Theorem \ref{th:3} (resp., Theorem \ref{th:4}) follows from Theorem \ref{th:1} (resp., from Theorem \ref{th:2}). Still, our line of thought worked backwards and we conclude this section with a few preliminary remarks explaining our logic. By its very design, the highest weight category structure on an abelian category gives rise to two distinguished collections of objects (standard and costandard objects) that satisfy the $\Ext$-vanishing close to the one that is required for exceptional collections. At the outset, the examples of full exceptional collections in rank-two cases indicated that standard modules in $\rep({\fb})$ with respect to one of the highest weight category structures on $\rep({\fb})$ could be relevant to exceptional objects in the coherent category $\Dd ^b(\Coh({\bg}/{\fb}))$. More specifically, those examples suggested that one would start with an appropriate finite collection of  standard modules in $\rep({\fb})$, would convert them into a collection of equivariant vector bundles on $\bg/\fb$, and would try to ensure the cohomological vanishing conditions for exceptional collections in the category $\Dd ^b(\Coh({\bg}/{\fb}))$. The main challenges are in understanding what is the right collection of standard modules in $\rep({\fb})$ to build the sought-for exceptional collection in the category $\Dd ^b(\Coh({\bg}/{\fb}))$, and in relating cohomology vanishing for $\fb$-modules to sheaf cohomology vanishing on ${\bg}/{\fb}$.
It turns out that the Steinberg basis \cite{St} is the foundation for the right choice of a collection of standard modules in $\rep({\fb})$; no less importantly, it becomes clear that one must also use {\it costandard} modules in $\rep({\fb})$ to actually produce exceptional objects in $\Dd ^b(\Coh({\bg}/{\fb}))$.  With these two main inputs combined together, the rest is cohomology vanishing theorems for $\fb$-modules that essentially rely on the foundational results from \cite{Math}, \cite{Polo3}, \cite{WvdK}. Eventually, to show all the needed cohomological vanishing, one works in the category $\Dd ^b(\rep({\fb}))$ considered as a module category over the base category $\Dd ^b(\rep({\bg}))$, cf. once again \cite[Remark 4.7]{Donsur}. 
In that setting, formulations of Theorems \ref{th:1} and \ref{th:2} appear naturally (see Section \ref{subsec:outline} below for more on that).

\comment{It would therefore be reasonable to start off with one of such collections in $\rep(\fb)$, convert them into a collection of equivariant vector bundles, and try to ensure the cohomological vanishing conditions in the category $\Dd ^b(\Coh({\bg}/{\fb}))$.
This doesn't work on the nose, for instance because full exceptional collections in $ \Dd^b (\rep ({\fb}))$ are infinite and on $\Dd^b (\Coh( \bg/\bp))$
one seeks a finite collection.
But it does work $\bg$-linearly with the input provided by the combinatorics of the Steinberg basis, \cite{St}.}

Coming back to the problem of lifts of $\sf K$-theoretic classes to the derived category level, we have already mentioned the classical works \cite{BGG} and \cite{Demres} that compute the singular cohomology and, respectively, the $\sf K$-theory of flag varieties $\bg/\fb$ in terms of natural operators called nowadays BGG-Demazure operators. These operators, in different guises, will appear throughout the paper: one of the highest weight category structures on $\rep( \fb)$ is defined in terms of those; see Section \ref{sec:rappelB-mod}. 
We now proceed to the second part of the introduction in which we outline the main steps of the argument.

\subsection{\unboldmath Outline of the proof}\label{subsec:outline}

Evidence for the principal construction of this paper comes from special full exceptional collections on flag varieties of rank-two groups over ${\mathbb Z}[\frac{1}{6}]$, the proofs of which were sketched in \cite{Sam}. We provided the complete details of the rank-two case in Section 14 of  
\href{https://doi.org/10.48550/arXiv.2407.13653}{arXiv:2407.13653v3}.
In a paper in preparation we will explain how the construction in Section 14 of \href{https://doi.org/10.48550/arXiv.2407.13653}{arXiv:2407.13653v3} fits the general construction given below.
For the purposes of introduction, assume for simplicity that $\Pm=\Bm$ and that
we are working over a field $\sk$  whose characteristic  is not too small.
Thus, $\fb$ and $ \bg$ denote the base change of the group schemes $\Bm$ and $\Gm$ along the morphism ${\rm Spec}(\sk)\rightarrow {\rm Spec}(\mathbb Z)$, \cite[Section 2.6]{Donkin2}. 
Given a finite-dimensional $\fb$-module $M$, let $\Ll(M)$ denote the associated (locally free) coherent sheaf on $ \bg/ \fb$ as defined in  \cite[I 5.8]{Jan} (see also Subsection \ref{subsec:sheafification_functor}).
Its fibre at the point $\fb/\fb$ is $M$.

The full exceptional collections from Section 14 of  
\href{https://doi.org/10.48550/arXiv.2407.13653}{arXiv:2407.13653v3}
 are given as certain $\bg$-equivariant locally free sheaves $\tX_v$. They are of the form $\Ll(X_v)$ where $X_v$ is the fibre at the point $\fb/\fb$.
 That the terms of those collections are $\bg$-equivariant vector bundles is in accordance with a general fact on exceptional objects on smooth projective varieties acted upon by a linear algebraic group, \cite[Lemma 2.2]{Polishchuk}. 
 
 More importantly, the collections in question are obtained with the help of Demazure functors. That points  to the fact that the highest weight category structures on $\rep(\fb)$ can be relevant: as was already mentioned, one of those structures is defined using Demazure operators, see Section \ref{subsec:excord}. 

With those rank-two examples in hand, it has become apparent that the $\fb$-module theory developed in \cite{WvdK} would be vital in generalizing the approach to arbitrary rank. 
One way to summarize the results of {\it loc.cit.} is to say that, in arbitrary rank, the abelian category $\rep( \fb)$ has two highest weight category structures,
dual to each other in a suitable sense.
Recall that simple modules in $\rep(\fb)$ are parametrized by  the weight lattice 
$X(\bt)$ of a maximal torus $\bt$ in $\fb$. 
Now the costandard objects with respect to one structure, which is defined by the {\it excellent} order on $X(\bt)$, (Definition \ref{def:excord}) are denoted $P(\lambda)$ for $\lambda \in X(\bt)$ and are related to {\it excellent filtrations} on $\fb$-modules, see Subsection \ref{subsec:filtrations}.  
The costandard objects with respect to another highest weight category structure, which is defined by the {\it antipodal excellent} order on  $X(\bt)$, (Definition \ref{def:antipodal}) are denoted $Q(\lambda)$ for $\lambda \in X(\bt)$ and are related to {\it relative 
Schubert filtrations} on $\fb$-modules, see Subsection \ref{subsec:filtrations}. It is these modules that will play a pivotal role in the construction of full exceptional collections on generalized flag schemes. 
We refer the reader to \cite{WvdK} for the highest weight category approach to $\rep( \fb)$; in \cite{WvdKTata} those structures are somewhat hidden in the background. The lectures \cite{WvdKTata} are devoted to the breakthrough paper \cite{MathG} by Mathieu; that breakthrough will be used in the guise of Corollary \ref{cor:PQnabla}.

In the next two paragraphs, we assume again that the group $\bg$ is of rank two.  A detailed inspection of the modules $ X_v$ led to the two key observations that made it possible to see the pattern. The first one is that in 
each type in rank two the $\fb$-modules $X_v$ have relative Schubert filtrations by modules $Q(\lambda)$, while the $\sk$-linear dual 
$\fb$-modules $X_v^*$ have excellent filtrations by modules $P(\lambda)$. 
The second observation is the appearance of 
the Steinberg weights $e_v$ in filtrations  on the modules $X_v$.
For a weight $\lambda\in X(\bt)$ we denote by $\sk_\lambda$ the one-dimensional $\fb$-module of weight $\lambda$. By \cite{St}, the Steinberg weights $e_v$ yield a basis $\{[\sk_{e_v}]\}_{v\in W}$ (the Steinberg basis) of the representation ring $R( \fb)=\sK_0(\rep( \fb))$ as a module over the representation ring $R( \bg)=\sK_0(\rep( \bg))$, see Section \ref{subsec:Steinberg}. 

Remarkably, if $v\in W$, then $Q(e_v)$ is on top in the relative Schubert filtration of $X_v$, and if $Q(\lambda)$ is another layer of the filtration, then the weight $\lambda$ of the $\fb$-socle of $Q(\lambda)$ is always a Steinberg weight $e_w$ with $w> v$ in the Bruhat order on $W$. 
On the other hand, the  modules $P(-\lambda)^*$
that occur in the other filtration of $X_v$ are of the form $\nabla_\mu\otimes P(-e_w)^*$ where $\nabla_\mu$ is a dual Weyl module for a  dominant weight $\mu\in X(\bt)$ and $w\leq v$ in the Bruhat order. As the corresponding coherent sheaves $\Ll(X_v), v\in W$ on $\bg/\fb$ form a full exceptional collection in rank two, this imposes a semiorthogonality condition on the sheaves $\Ll(P(-e_w)^*), w\in W$ and $\Ll(Q(e_v)), v\in W$ with respect to the Bruhat order. Furthermore, the $\fb$-socle of the module $Q(e_v)$ is ${\sk }_{e_v}$, while the same ${\sk }_{e_v}$ is also the $\fb$-head of the module $P(-e_v)^*$. These facts in rank two will serve as the basis for producing the sought-for semi-orthogonal decompositions for flag varieties in arbitrary rank; it starts with Theorems \ref{th:indSteinberg_vanishing} and \ref{th:indPQ} (at this point, the reader is also invited to keep consulting Section \ref{sec:Leitfaden} for an easier navigation between the key statements across the text). It is also at this stage that the interplay among three orders - one on the Weyl group (the Bruhat order) and the other two on $X(\bt)$ (the excellent and antipodal excellent orders) becomes crucial in our considerations.

It should be noted that an earlier paper \cite{AAGZ} unveiled the role played  by the Steinberg basis in the setting of exceptional collections. At the same time, it showed its limitations already for the group ${\bf G}_2$ if the Steinberg basis is taken literally. As our paper explains, the crucial missing ingredient was that the Steinberg weights should have been seen through the lens of the category $\rep(\fb)$ -- that is, as weights of socles (resp., heads) of costandard (resp. standard) objects for the highest weight category structure on $\rep(\fb)$ given by the antipodal excellent order. 
The remaining parts of the costandard modules $Q(e_v), v\in W$ (resp., of   the standard modules $P(-e_v)^*, v\in W$) that lie above the socle ${\sk}_{e_v}$ (resp., below the head ${\sk}_{e_v}$)   are responsible for the eventual semiorthogonality properties.

Now, as we have the Steinberg basis $\{[\sk_{e_v}]\}_{v\in W}$, and the 
$\fb$-modules $P(-e_v)^*$ and $Q(e_v)$ with $v\in W$ at our disposal for  $\bg$ of arbitrary rank, let us take a closer look at the semiorthogonality conditions on $\Ll (P(-e_w)^*)$ and $\Ll (Q(e_v))$ and explain how we arrive at those. Precisely, this is about Theorem  \ref{th:indPQ}. 
The preparatory technical work for this theorem is done in Section \ref{sec:triangular} and proceeds as follows. Let $\fb^+$ be the opposite Borel subgroup to $\fb$. The $\bt$-equivariant $\sf K$-theory of a flag variety $\bg/\fb^+$ for a Borel subgroup $\fb^+$
has two natural bases: one consisting of the classes $\{[\Oo _{X_w}],w\in W\}$ of the closures of $\fb^+$-orbits on $\bg/\fb^+$  (the ``Schubert basis''), and another consisting of the classes 
$\{[\Oo_{\fX^w}(-\partial \fX^w)]\}_{w\in W}$ of the $\fb ^{-}$-orbits on $\bg/\fb^+$ for the opposite Borel subgroup $\fb ^{-}$ to $\fb^+$ (the ``opposite Bruhat cell basis''). The characters of the modules $Q(e_v)$ and $P(-e_w)$ can be computed with Demazure operators, and with computer-assisted computations in rank three we have found a close relationship in the $ \bt$-equivariant $\sK$-group $\sK_{ \bt}( \bg/ \fb)$ between the classes $[\Ll(P(-e_w))], w\in W$ and the Schubert basis $\{[\Oo _{X_w}],w\in W\}$. Let $R(\bt)$ denote the representation ring $\sK_0(\rep( \bt))$. Then the $R(\bt)$-module ${\sf K}_{\bt}(\bg/\fb)$  is equipped with a natural $R(\bt)$-valued pairing, and a result from \cite{GrahamKumar}, which is attributed to Knutson in {\it loc.cit.}, asserts that the Schubert basis and the opposite Bruhat cell bases are orthogonal to each other with respect to that pairing. Using this very important input, in Section \ref{sec:triangular} we prove Theorem \ref{th:triang}, which states that the transition matrix between the basis $\{[\Ll(P(-e_w))]\}_{w\in W}$ and the Schubert basis of $\sK_{ \bt}( \bg/ \fb)$ is always invertible and triangular, up to a permutation of rows and columns. Similarly, the transition matrix from the basis $\{[\Ll(Q(e_v))]\}_{v\in W}$ to the opposite Bruhat cell basis is also invertible and triangular, up to a permutation of rows and columns. Finally, much of the cohomology vanishing required by the statement of Theorem \ref{th:indPQ} is assured by the results explained in \cite{WvdKTata}, e.g.\ Theorem \ref{th:highInd}, which is based on \cite{MathG}. These allow us to reduce the proof of Theorem \ref{th:indPQ} to  Euler characteristic computations, and this is precisely what Section \ref{sec:triangular} does with the help of all the input explained above in this paragraph. The triangularity results of Theorems \ref{th:indSteinberg_vanishing} and \ref{th:indPQ} turn out to be the key to the desired $\Ext$-vanishing required for exceptional collections. Arguably,  Theorems \ref{th:indSteinberg_vanishing} and \ref{th:indPQ} are the most important cohomological statements of our work.

Section \ref{sec:triangular}, discussed in the previous paragraph, is concerned with
 the structure of the {$R(\bt)$-module ${\sf K}_\bt(\bg/\fb)$}. Similarly, the representation ring $R( \fb)=\sK_0(\rep( \fb))$ is a module over the representation ring $R(\bg)=\sK_0(\rep( \bg))$; in both cases, the $R(\bt)$-{module structure} for the former module (resp., the $R(\bg)$-{module structure} for the latter) is essential for our purposes. More generally, the category $\rep( \fb)$ has a {module structure} over the base abelian category $\rep( \bg)$; that {module structure} is given by the restriction functor $\res_{\fb}^{ \bg}$,  cf. the aforementioned quote from Donkin. Passing to derived categories, the category $\Dd ^b(\rep( \fb))$ receives a {module  structure} over the base category $\Dd ^b(\rep( \bg))$ (since that {module structure} comes from a monoidal action at the level of abelian categories, it doesn't require further compatibilities). In this framework, thanks to Theorems \ref{th:indSteinberg_vanishing} and \ref{th:indPQ}, the basic examples of Section 14 of  
 \href{https://doi.org/10.48550/arXiv.2407.13653}{arXiv:2407.13653v3}
 can be reformulated as saying that for a group $\bg$ of rank two and a Borel subgroup $\fb\subset \bg$, there are $\bg$-linear semi-orthogonal decompositions of the category $\Dd ^b(\rep( \fb))$  with each semi-orthogonal component being equivalent to $\Dd ^b(\rep( \bg))$. 

Given a base scheme $S$, there are robust notions of $S$-linear triangulated categories and of $S$-linear semi-orthogonal decompositions of a given $S$-linear triangulated category, \cite{Kuzbasech}. The notions of a $\bg$-linear triangulated category and of a $\bg$-linear semi-orthogonal decomposition are in complete parallel and could be put on an equal footing with those from {\it loc.cit.} if we used the language of (quotient) stacks that we deliberately avoid. For that reason, we recall in Sections \ref{subsec:G-linear-triang_cat} and \ref{subsec:G-linear-sod} all the necessary definitions and statements concerning $\bg$-linear triangulated categories in the classical language of equivariant (quasi)-coherent sheaves.

Having set up the necessary framework, we proceed to constructing $ \bg$-linear semi-orthogonal decompositions of $\Dd ^b(\rep( \fb))$ in arbitrary rank. For that we need to choose a total order $\prec$ on the Weyl group $W$ that refines the Bruhat order. For any element $p\in W$ of the Weyl group we will first construct in Subsection \ref{subsec:cut_at_p} a $\bg$-linear semi-orthogonal decomposition of $\Dd ^b(\rep( \fb))$ (called ``cut at $p$'') into two  subcategories 
\begin{equation}\label{eq:p-cutSOD}
\langle\;\hull(\{\nabla_\lambda\otimes Q(e_v)\}_{v\succ p,\lambda\in X(\bt)_+}),\;\hull(\{\nabla_\lambda\otimes P(-e_v)^*\}_{v\preceq p,\lambda\in X(\bt)_+})\;\rangle.
\end{equation}

Here $\hull(-)$ denotes the triangulated envelope of a given set of objects, and $\nabla_\lambda$ are dual Weyl modules. Strikingly, the $\fb$-modules $P(-e_v)^*, v\in W$ and $Q(e_v), v\in W$ will serve for both the semiorthogonality condition in (\ref{eq:p-cutSOD}), which will follow from Theorem \ref{th:indPQ} and Corollary \ref{cor:left orthogonal}, and for the generating property. Precisely, in Theorem \ref{th:derived_generation} 
we prove that the triangulated envelope of the two triangulated hulls in (\ref{eq:p-cutSOD}) is the whole $\Dd ^b(\rep( \fb))$. That proof, which starts in Subsection \ref{subsec:D^brepB-generation}, can be considered as a categorical upgrade of 
\cite[Theorem 2]{Ana} which is essentially the same statement as in Theorem \ref{th:derived_generation}, but at the $\sf K$-theory level. The reader may prefer to consult \cite[Theorem 2 and Section 4]{Ana} as a starting point before proceeding to Section \ref{subsec:D^brepB-generation}. As a byproduct of the results of Section \ref{sec:B-generation}, we obtain an alternative proof and categorification of a theorem of Steinberg \cite{St}, see Theorem \ref{th:Pgeneration}, which is similar to \cite[Theorem 2]{Ana} of Ananyevskiy.

Semi-orthogonal decompositions (\ref{eq:p-cutSOD}) are the core part of our paper as they allow us to construct objects of $\Dd ^b(\rep( \fb))$ that will eventually give rise to full exceptional collections on $\bg/\fb$. More precisely, for an element $p\in W$ we use (\ref{eq:p-cutSOD}) to define in Subsection \ref{subsec:X_p-Y_p} two objects $X_p$ and $Y_p$ of $\Dd ^b(\rep( \fb))$. Specifically, we define the object $X_p\in \Dd ^b(\rep( \fb))$ to be the image of $P(-e_p)^*$ under the left adjoint of the inclusion of the triangulated hull of $\{\nabla_\lambda\otimes Q(e_v)\}_{v\succeq  p,\lambda\in X(\bt)_+}$ into $\Dd ^b(\rep( \fb))$. Similarly, we define $Y_p$ to be the image of $Q(e_p)$ under the right adjoint of the inclusion of the triangulated hull of $\{\nabla_\lambda\otimes P(-e_v)^*\}_{v\preceq p,\lambda\in X(\bt)_+}$ 
into $\Dd ^b(\rep( \fb))$. The reason for defining the objects $X_p$ and $Y_p$ in this manner was prompted by the insight coming from the examples in rank two: the exceptional objects $\tX_v$ on $\bg/\fb$ for rank-two groups lie inside the intersection  of the hull of $\{\Ll(P(-e_v)^*)\}_{v\preceq p}$ with the hull of $\{\Ll(Q(e_v))\}_{v\succeq p}$. 

The final step of this core part is showing that for any $p\in W$ the objects $X_p$ and $Y_p$ are isomorphic. This is done in Section  \ref{sec:X_p&Y_p-isom}. To this end, we need a refinement of our categorification of Steinberg's theorem, which is Theorem \ref{th:cone generation}. This refinement, implicit in \cite{Ana}, shows that the Steinberg weights serve as a curious meeting ground of the Bruhat order and the antipodal excellent order $<_a$ on the weight lattice $X(\bt)$  (compare also Remark \ref{rem:Steinberg from anti}). At this stage we learn from Proposition \ref{prop:X_p-G-exceptional-coll} that the objects $X_p$ that we have constructed are $\fb$-exceptional, \emph{i.e.} $\RHom _{\Dd^b(\rep(\fb))}(X_p,X_p)=\sk$; that is, they are exceptional in $\Dd ^b(\rep( \fb))$ considered as a $\sf k$-linear category.

But now, thanks to Theorem \ref{th:indPQ} and to the very construction of objects $X_p, p\in W$, there is much more to cohomological properties of the $X_p$'s than being $\fb$-exceptional: these objects turn out to be exceptional in $\Dd ^b(\rep( \fb))$ considered as a $\bg$-linear category, see Theorem \ref{th:X_p-G-exceptional-coll} for the precise statement. Now, by Proposition \ref{prop:G-linear_exceptional}
each object $X_p$, $p\in W$ gives rise to a $ \bg$-linear functor 
$\Phi _p\colon \Dd^b(\rep ( \bg))\rightarrow \Dd^b(\rep(\fb))$. 
Cohomological vanishing statements from Section \ref{sec:b-cohom_vanish} combined with Theorem \ref{th:X_p-G-exceptional-coll} prove that each $ \bg$-linear functor $\Phi _p:\Dd^b(\rep ( \bg))\rightarrow \Dd^b(\rep(\fb))$ is full and faithful. 
Results of Subsections \ref{subsec:G-linear-triang_cat} and \ref{subsec:G-linear-sod} then allow us to establish that the collection of full triangulated subcategories ${\rm Im}(\Phi _p)\subset \Dd^b(\rep(\fb))$, $p\in W$, forms a $ \bg$-linear semi-orthogonal decomposition of the 
category $\D= \Dd^b(\rep(\fb))$ with respect to the {chosen} total order $\prec$ on $W$. 
This is the statement of Theorem \ref{th:semi-orthogonal}.

Now the proof of Theorem \ref{th:coh-semi-orthogonal} goes as follows.
We define objects $\tX_p\in \Dd^b(\Coh ( \bg/ \fb))$ as the images of $X_p$, $p\in W$ from Section \ref{sec:X_p&Y_p} under the map $\Dd^b(\rep(\fb))\to \Dd^b(\Coh ( \bg/ \fb))$ induced by the ``associated sheaf'' functor $\Ll:\rep(\fb)\to \Coh ( \bg/ \fb)$.
Theorem \ref{th:semi-orthogonal} implies that the collection of objects $\tX_p, p\in W$ is an exceptional collection in $\Dd^b(\Coh ( \bg/ \fb))$, while the generating property of $\tX_p, p\in W$ is assured by the results of Section \ref{sec:B-generation}. Remark \ref{rem:basechangeG-to-point} draws a parallel of Theorem \ref{th:coh-semi-orthogonal} with a base change type of statement for semi-orthogonal decompositions, \cite{Kuzbasech}.

Section \ref{sec:caseP} treats the case of a parabolic subgroup $\bp\supset \fb$. Its main statements, which are Theorems \ref{th:Psemi-orthogonal} and \ref{th:Pcoh-semi-orthogonal}, follow essentially the same path that has been set out in Theorems \ref{th:semi-orthogonal} and \ref{th:coh-semi-orthogonal}. Recall that $W^\bp$ is the set of minimal coset representatives of $W/W_\bp$, where $W_\bp$ is the Weyl group of $\bp$. For the reasons that are explained below in this paragraph, it is natural to expect that the objects ${\hat{\sX}} _v, {v\in W^\bp}$ that give rise to semi-orthogonal decompositions of $\Dd ^b(\rep(\bp))$ as a $\bg$-linear category (resp., the objects $\hat \tX_v\in \Dd ^b(\Coh(\bg/\bp)), v\in W^\bp$ giving full exceptional collections in $\Dd ^b(\Coh(\bg/\bp))$) are contained among the objects ${\sX} _v, {v\in W}$ of Theorem \ref{th:semi-orthogonal} (resp., among the objects $\tX_v, v\in W$ of $\Dd ^b(\Coh(\bg/\fb))$ of Theorem \ref{th:coh-semi-orthogonal}). One has therefore to recognize those objects among ${\sX} _v, {v\in W}$ (resp., among $\tX_v, v\in W$) that are obtained by the restriction functor $\res_{\fb}^{\bp}:\Dd ^b(\rep(\bp))\rightarrow \Dd ^b(\rep(\fb))$ (resp., by the pullback $\pi _{\bp}^{\ast}:\Dd ^b(\Coh(\bg/\bp))\rightarrow \Dd ^b(\Coh(\bg/\fb))$ along the projection $\pi _{\bp}:\bg/ \fb\rightarrow  \bg/ \bp$). The fundamental fact that both functors $\res_{\fb}^{\bp}$ and $\pi _{\bp}^{\ast}$ are $t$-exact and fully faithful on the respective derived categories makes it possible to recognize the sought-for exceptional objects on $\bg/ \bp$ by applying the induction functor $\Rind _{\fb}^{\bp}$ to 
appropriate objects ${\sX} _v, {v\in W}$ (resp., the pushforward $R{\pi _{\bp}}_{\ast}$ to $\tX_v, v\in W$). It turns out the Steinberg weights $e_v, v\in W^\bp$ for a given parabolic $\bp$ behave nicely with respect to the induction functor $\Rind _{\fb}^{\bp}$, suggesting a natural parabolic analogue of the key $\fb$-modules from Section \ref{sec:rappelB-mod}. The cohomological properties of those parabolic analogues are given by Theorem \ref{th:PindPQ}, a parabolic counterpart of Theorem \ref{th:indPQ}. That allows to further apply the arguments of Sections \ref{sec:B-generation} and \ref{sec:X_p&Y_p} in the parabolic case and obtain Theorems \ref{th:Psemi-orthogonal} and \ref{th:Pcoh-semi-orthogonal}.

\subsection{\unboldmath Discussion}\label{sec:disc}

The construction of the objects $\tX_p$, $p\in W$ from Theorem \ref{th:coh-semi-orthogonal} (resp., $\hat \tX_v, v\in W^{\bp}$ from Theorem \ref{th:Pcoh-semi-orthogonal}) a priori produces complexes of coherent sheaves on $\bg/\fb$ (resp., on  $\bg/\bp$). At the moment, we don't yet know whether their cohomological amplitude is of zero length, i.e. whether the objects $\tX_p$, $p\in W$ (resp., $\hat \tX_v, v\in W^{\bp}$) are vector bundles, possibly up to a shift in the derived category (cf. also \cite[Conjecture 4.1]{KuzPol}).

Evidence for the fact that the cohomological amplitude of $\tX_p$, $p\in W$ (resp., of $\hat \tX_v, v\in W^{\bp}$) is of zero length is provided by the rank-two cases, see Section 14 of  \href{https://doi.org/10.48550/arXiv.2407.13653}{arXiv:2407.13653v3}. In Section \ref{sec:dualFEC_conj}, based on further evidence coming from low rank cases in ranks up to three, we state some conjectural cohomological statements that shall eventually give the description of dual exceptional collections to $\tX_v, v\in W$ in $\Dd ^b(\Coh(\bg/\fb))$. If true, that description must provide further evidence for the cohomological amplitude of the objects $\tX_v, v\in W$ being minimal.

Presumably, Theorems \ref{th:semi-orthogonal} and \ref{th:Psemi-orthogonal} have further extensions to larger settings than the one of 
semisimple algebraic groups of our paper. Much of the representation-theoretic input that we used in proving Theorems \ref{th:semi-orthogonal} and \ref{th:Pcoh-semi-orthogonal} also exists in the quantum setting, \cite{APW}. This suggests natural generalizations
of the said theorems to the quantum case. The representation theory of parahoric subalgebras of affine Kac-Moody Lie algebras studied in \cite{FKM} suggests another extension.

A very interesting question is to understand a relation of our results to baric structures/staggered $t$-structures on the derived categories of equivariant coherent sheaves from \cite{Achar} and \cite{AchTr}. 
Further, full exceptional collections on flag varieties appear through the computations related to the dual Steinberg basis of ${\sf K}_0^{ \bg}( \bg/ \fb)$, \cite{Daw}. 
These, in turn, are related to Lusztig's asymptotic affine Hecke algebra and to Lusztig's canonical basis in the $ \bt$-equivariant $\sf K$-theory of Springer fibres. 
Ultimately, the problem is to relate the objects $X_p$, $p\in W$ of the present paper to Lusztig's canonical basis. 

\subsection*{Acknowledgements}
A.S. is grateful to Henning Krause for his unwavering support of this work and for his and Charles Vial's joint invitation to the BIREP group of  Bielefeld University for a long-term visit in 2023-2025. 
The research of A.S. was supported by the Deutsche Forschungsgemeinschaft (SFB-TRR 358/1 2023 - 491392403). He also acknowledges the hospitality of the Mathematical Institute of the University of Utrecht during his visits in May 2024, in March and May 2025, and in May 2026.

W.v.d.K. thanks BIREP for the opportunity to visit Bielefeld in September 2023.
We thank the referee for a very thorough reading and for numerous suggestions for improvement.

\subsection{\unboldmath Leitfaden}\label{sec:Leitfaden}
{\footnotesize
\vspace{2cm}
\xymatrix @R=2pc @C=-2.5pc{*+[F]\txt{\bf Corollary \ref{cor:PQnabla} \\ 
($\fb$-cohomology vanishing \\ for $P(\lambda)\otimes Q(\mu)\otimes \nabla_\nu$)} \ar[d]  \ar[drr]
 & &  *+[F]\txt{\rm {\bf Theorem \ref{th:triang}} (Triangularity)} \ar[d]   &  \\
 *+[F]\txt{\bf Theorem \ref{th:highInd}\\
($\RR^i\ind_{\fb}^ \bg (P(\lambda)\otimes Q(\mu))=0$ \\ for $i>0$)} \ar[drr]  
& &  *+[F]\txt{{\bf Theorem \ref{th:indSteinberg_vanishing} (\ref{th:indPQ})} \\ 
(Computing $\ind_{\fb}^ \bg (P(-e_v)\otimes Q(e_w))$ \\ 
 for $w\succeq v$)}  \ar[d]  &  \\
&  *+[F]\txt{\bf Theorem \ref{th:generation} \\
(Generation)} \ar[r] \ar@{.>}[ddl]^{Refined \ version \ of \ generation} &  *+[F]\txt{{\bf Section \ref{subsec:cut_at_p}} (``Cut at $p\in W$'')} \ar[dd]  \\
& & \\
 *+[F]\txt{\bf Section \ref{sec:X_p&Y_p-isom}\\ 
(Vanishing of the cone of $X_p\rightarrow Y_p$)
} \ar[ddr]        
&  & *+[F]\txt{\bf Section \ref{subsec:X_p-Y_p} \\
(Construction of $X_p$'s and $Y_p$'s)} \ar[ddl] \ar[ll] &  \\
\\
&   *+[F]\txt{\bf Section \ref{sec:SODofB-modoverG-mod} \\ 
(Constructing $\bf G$-linear semi-orthogonal \\\ decompositions of $\Dd ^b(\rep(\fb)$) \\ with the help of objects $X_v,v\in W$)} \ar[dd]  & \\
\\
&   *+[F]\txt{\bf Sections \ref{sec:FEConGmodB} and \ref{sec:caseP} \\ 
(Full exceptional collections \\ on ${\mathbb G/\mathbb P}$)}  &
 }
}

\newpage

\section{\unboldmath Generalities}\label{subsec:Notation}

\subsection{{Simply connected semi-simple  algebraic $\Z$-groups}}\label{subsec:reductive}
We fix a  split simply connected semi-simple algebraic $\Z$-group $\Gm$, with terminology as in \cite[II 1.1, II 1.6]{Jan}. Its geometric fibres are simply connected
semi-simple algebraic groups.
We choose a split maximal torus $\Tm$
in $\Gm$, a Borel subgroup scheme $\Bm$ containing $\Tm$ and a parabolic 
 subgroup scheme $\Pm$ containing $\Bm$.
 We have the flag scheme ${\Gm}/{\Bm}$ and the generalized flag scheme ${\Gm}/{  \Pp}$. 
Then ${\Gm/\Pp}\rightarrow {\rm Spec}({\Z})$ is flat.

Given an affine algebraic group $\bh$, flat over a noetherian base ring $\sk$, an $\bh$-module is a $\sk[\bh]$-comodule (a rational representation). 
The abelian category of $\bh$-modules is denoted $\Rep(\bh)$ and $\rep(\bh)$  denotes the full abelian subcategory of representations that are finitely generated over the base ring $\sk$ (cf. \cite[Section 1]{AndChev}).
We will mostly follow the notation and terminology of \cite{Jan}.
Thus $\Hom_\bh(M,N)$ will be the $\sk$-module of homomorphisms between 
$\bh$-modules $M$ and $N$ and $\Ext^i_\bh(M,N)$ is defined similarly.
Further $H^i(\bh,M)$ means $\Ext_\bh^i(\sk,M)$ with $\bh$ acting trivially on $\sk$.
If $\sk'$ is a commutative $\sk$-algebra, then $\bh_{\sk'}$ denotes the algebraic $\sk'$-group obtained by base change \cite[I 1.10]{Jan} along $\Spec(\sk')\to\Spec(\sk)$.

\begin{lemma}(\cite[I 2.7(4), I 4.2(1), I 4.4]{Jan})\label{lem:dualmodule}
    Let $M$, $N\in \Rep(\bh)$ with $M$
    finitely generated and projective over $\sk$. Then $M^*=\Hom_\sk(M,\sk)$ is an $\bh$-module and 
    $$\Ext^n_\bh(M,N)=H^n(\bh,M^*\otimes N).$$
\end{lemma}

\subsection{Universal coefficient Theorem}
\begin{theorem}[Universal coefficient Theorem]\label{th:universal coefficients}
Let $\sk$ be a Dedekind ring and let $\bh$
be an affine algebraic $\sk$-group, flat over $\sk$. Let $\sk'$ be a commutative $\sk$-algebra and $n\geq0$.
\begin{enumerate}
\item There is for any $\bh$-module $M$, flat over $\sk$, an exact sequence
    $$0\to H^n(\bh,M)\otimes_\sk \sk'\to H^n(\bh_{\sk'},M\otimes_\sk\sk')\to \Tor_1^\sk(H^{n+1}(\bh,M),\sk')\to 0.$$
\item 
Let  $Y$ be a closed reduced
$\Bm_\sk$-invariant subscheme, flat over $\sk$, of the generalized flag variety $\Gm_\sk/\Pp_\sk$. For any  $\Pp_\sk$-module $N$, finitely generated and flat over $\sk$, we have an exact 
sequence 
$$0\to H^n(Y,\Ll(N))\otimes_\sk \sk'\to H^n(Y_{\sk'},\Ll(N\otimes_\sk\sk'))
        \to \Tor_1^\sk(H^{n+1}(Y,\Ll(N)),\sk')\to 0,$$
where $\Ll(N)$ is the restriction to $Y$ of the vector bundle associated to $N$ 
(Subsection \ref{subsec:sheafification_functor}) and $\Ll(N\otimes_\sk\sk')$ is defined similarly.
 \end{enumerate}
\end{theorem}
\begin{proof}
        The first part is \cite[Proposition I 4.18]{Jan}.
        For the second part we choose an affine open cover of $\Gm/\Pp$ and use it to compute the cohomology groups as \v Cech cohomology groups (\cite[Theorem 4.5]{Hartshorne}). The  \v Cech complexes are flat over $\sk$ and $\sk'$ respectively, so we 
        may argue in the same manner as in the proof of the first part.
    \end{proof}

\begin{remark} 
In part 2 of Theorem \ref{th:universal coefficients} the sheaf $\Ll(N)$ is coherent so that
    $H^n(Y,\Ll(N))$ is a finitely generated $\sk$-module by \cite[Thm. III 5.2]{Hartshorne}.
    When applying part 1 of the same theorem we will always have that
    $H^n(\bh,M)$ is a finitely generated $\sk$-module by \cite[Chapter B]{Jan}.
\end{remark}

\subsection{Notation}

Our base ring $\sk$ is usually a field or $\Z$. We will discuss the case $\sk=\Z$ separately, or leave the necessary modifications to the reader.
For instance, when $M$ is a finitely generated free module over $\sk=\Z$, then $\dim_\sk(M)$ means its rank. When we say that 
a module $M$ has a weight $\nu$ with multiplicity one, then we mean that the weight space $M_\nu$ is free of rank one.

In our references the field is often assumed algebraically closed, but this assumption is irrelevant in our context because of
the straightforward behaviour of the various cohomology groups under field extensions of $\sk$.

Let  $\bg$  be  obtained from ${\mathbb G}$ by base change along ${\rm Spec}(\sk)\rightarrow {\rm Spec}(\mathbb Z)$.
When $\sk=\Z$ this just means that $\bg=\Gm$,  $\fb=\mathbb B$, $\bp=\mathbb P$, $\bt=\mathbb T$.

\subsection{{Associated sheaf functor $\Ll$}}\label{subsec:sheafification_functor}

Let $\bf H$ be a flat $\sf k$-group scheme acting freely (from the right) on a flat $\sf k$-scheme $X$ such that $X/\bf H$ is a scheme. 
Associated to each $\bf H$-module $M$ is a sheaf $\Ll (M)=\Ll _{X/\bf H}(M)\in{\rm Sh}(X/\bf H)$ where ${\rm Sh}(X/\bf H)$ is the category of sheaves of $\Oo _{X/\bf H}$-modules on $X/\bf H$, \cite[I, Section 5.8]{Jan}. The functor $\Ll : {\Rep}({\bf H})\rightarrow {\rm Sh}(X/\bf H)$ is exact and lands in the subcategory ${\QCoh}(X/\bf H)$ of quasi--coherent $\Oo _{X/\bf H}$-modules of ${\rm Sh}(X/\bf H)$, \cite[I, Proposition 5.9, (a) and (b)]{Jan}.

Now set ${\bf H}:= \fb, X:= \bg$ (so $X/{\bf H}= \bg/ \fb$) in the definition of the functor $\Ll$. Then $\Ll$ yields a functor between the bounded derived categories $\Dd ^b({\rep}({\fb}))\rightarrow \Dd ^b({\sf Coh}(\bg/ \fb))$, by \cite[I, Proposition 5.9, (c)]{Jan} and Theorem \ref{thm:inc} below. 

\subsection{Weights}

Now let $\sk$ be a field.
So $ \bg$ is a split semisimple simply connected affine algebraic group. 
Let $W = N_{ \bg}({\bt})/ \bt$ be the Weyl group where $\fb\supseteq \bt$ corresponds to $\Bm\supseteq \Tm$. 
Let $X(\bt)$ be the weight lattice and $\Phi$ be the root system.  We choose the system $\Phi ^+$ of positive roots which makes $\fb$ the \emph{negative} Borel and let $\Pi$ be the set of simple roots. We write $\fb^+$ for the Borel group opposite to $\fb$. The roots of $\fb^+$ are positive.
Let $(~,~)$ be a 
symmetric, positive definite $W$-invariant form on ${\mathbb R}\otimes_{\mathbb Z} X(\bt)$ and let $X(\bt)_+=\{\lambda \in X(\bt)
\mid(\lambda,\alpha ^{\vee})\geq 0$ for all $\alpha \in \Pi$\} be
the set of dominant weights, where $\alpha^\vee =\frac{2\alpha}{(\alpha ,\alpha)}$ for $\alpha \in \Phi$. Let $S=\{\;s_\alpha\mid \alpha\in \Pi\;\}$ be the set of simple reflections, with $s_\alpha$ the reflection in the hyperplane perpendicular to $\alpha$. Then $(W,S)$ is a Coxeter system (\cite[II 5.1]{Humphreys}).
The fundamental weights $\omega_\alpha$ satisfy $(\omega_\alpha,\beta^\vee)=\delta_{\alpha,\beta}$ for $\alpha, \beta\in \Pi$.
Observe that $W$, $\Phi$, $\Pi$, $X(\bt)$, \ldots, are independent of the choice of the field $\sk$. They only depend on  $\Gm$,  $\Tm$, $\Bm$.
\\

We want to emphasize once again the fact that our convention (which is that of \cite{Jan}) is that ${\fb}$ corresponds 
to the {\it negative roots} (as opposed to the conventions of \cite{GrahamKumar}, \cite{MathG}, \cite{Polo3}, and \cite{WvdK}). 
For that reason, the translation of various results from {\it loc.cit.} that are used extensively in this paper requires some care.\\ 

Given a $\fb$-module $M$, the sheaf $\Ll(M)$ on $ \bg/ \fb$ has fibre $M$ at the point $ \fb/ \fb$. Given a weight $\lambda \in  X(\bt)$,
denote $\Ll ({\lambda})$ the corresponding line bundle on ${ \bg}/{\fb}$.
Here we still follow  \cite{Jan}, so that $\Ll(\lambda)$ 
 is associated to the one-dimensional $\fb$-module $\sk_\lambda$ of weight $\lambda$, not to the dual of that representation.

The weight lattice $X(\bt)$ has a natural partial order $\geq_d$, known as the dominance order: 
for $\lambda,\mu \in X(\bt)$ we write $\lambda \geq_d \mu$ if $\lambda - \mu$ is a sum of positive roots, with repetitions allowed.
We denote by $w_0$ the longest element of the Weyl group $W$  and let $\lambda ^{\ast} = -w_0\lambda$ for $\lambda \in X(\bt)_+$ (dual or contragredient, cf. \cite[Section 2.2]{Hummodular}).\\
 
For $\lambda \in X(\bt)$ let 
$\nabla _{\lambda}$ 
be the induced module $\ind_{\fb}^{ \bg}({\sk}_{\lambda})$ \cite[I, Section 3.3]{Jan}. 
It is finite-dimensional and it is non-zero if and only if $\lambda$ is dominant. 
For $\lambda \in X(\bt)_+$ we denote by $\Delta_\lambda$ the Weyl module $(\nabla _{\lambda^{\ast}})^{\ast}$. Then $\nabla _{\lambda}$  has simple socle known as $L(\lambda)$
and $\Delta_\lambda$ has simple head $L(\lambda)$. 
If $\sk=\Z$ then  $\nabla _{\lambda}$, $\Delta_\lambda$ are finitely generated and free over $\sk$, but heads and socles make less sense.

\begin{remark}\label{rem:Steinberg}
We will need  the Steinberg weights (see Subsection \ref{subsec:Steinberg})  to be characters of $\fb$.
 That is why we take
     ${\bg}$ simply connected.
Still, the simply connected case suffices for geometric applications to generalized flag varieties in Sections \ref{sec:FEConGmodB} and \ref{sec:caseP}. Indeed, if ${\bp}\subset{\bg}$ is 
    a parabolic subgroup  of a reductive   ${\bg}$, with $\bp$ containing a split maximal torus,
    then one may also view the generalized flag variety ${\bg/{\bp}}$ as a homogeneous space for the simply connected cover of $[{\bg},{\bg}]$.
  \end{remark}

\section{\unboldmath Rappels: $\fb$-modules}\label{sec:rappelB-mod}
Let $\sk$ be a field.
\subsection{\unboldmath Highest weight category structure on $\Rep( \bg)$}\label{subsec:hwc_repG}

Simple modules in $\Rep( \bg)$ are the $L(\lambda)$ parametrized by the dominant weight $\lambda\in X(\bt)_+$, \cite[II, Section 2]{Jan}.
By \cite{CPSHW} and \cite{Donkin1}, there is a highest weight category structure
on $\Rep( \bg)$ with the weight poset 
$(X(\bt)_+,\leq_d)$: the standard modules in this structure are the Weyl modules $\Delta _{\lambda},\lambda \in X(\bt)_+$ and the costandard modules are the  $\nabla _{\lambda},\lambda \in X(\bt)_+$, also known as dual Weyl modules.

\subsection{\unboldmath Joseph-Demazure modules}\label{subsec:J-D}

Let $ \bh\subset  \bg$ be a   closed subgroup of $ \bg$. 
Associated to it are the restriction and induction functors ${\res}_{ \bh}^{ \bg}$ and ${\ind}_{ \bh}^{ \bg}$, \cite[I, Section 3.1]{Jan}. Recall the definition of Joseph's functors ${\sH}_w, w\in W$, following \cite{WvdK}.

\begin{definition}\label{def:Joseph_functor}
Let $M$ be a $\fb$-module and  $w\in W$. 
Consider the Schubert variety $\fX_w=\overline{\fb w\fb}/\fb\subset \bg/ \fb$ associated to $w$. 
Then the functor ${\sH}_w:{\Rep}( \fb)\rightarrow {\Rep}( \fb)$ is given by $M\mapsto {H}^0(\fX_w,\Ll (M))$. 
\end{definition}

Equivalently, by \cite[Proposition 2.2.5]{WvdKTata}, the functor ${\sH}_w$ can be described as follows.
If $\alpha\in \Pi$ with corresponding simple reflection $s=s_\alpha$, 
let $\bp_\alpha$ denote the minimal parabolic generated by $\fb$ and the root subgroup with root $\alpha$.
Then $\sH_s={\res}_{\fb}^{{ \bp}_{\alpha}}{\ind}_{\fb}^{{ \bp}_{\alpha}}$. And if  $s_{1}s_{2}\dots s_{n}$ is   a reduced expression for $w$, then
$${\sH}_w=\sH_{s_1}\circ\dots\circ \sH_{s_n}.$$

The functor ${\sH}_w$ is left exact; let ${\rm R}{\sH}_w$ denote its right derived functor $\Dd ^+({\Rep}( \fb))\rightarrow \Dd ^+({\Rep}( \fb))$. It restricts to a functor between bounded derived categories  $\Dd ^b({\rep}( \fb))\rightarrow \Dd ^b({\rep}( \fb))$ (cf.\ Subsection \ref{subsec:G-linear-triang_cat} and Theorem \ref{thm:inc}).

\begin{definition}\label{def:dualJosephmod}(Dual Joseph modules $P(\lambda)$).
Let $\lambda \in X(\bt)$. Let $\lambda^+$ be the dominant weight in the Weyl group orbit of $\lambda$ and let $w$ be minimal so that $\lambda=w\lambda^+$. The dual Joseph module $P(\lambda)$ is set to be ${\sH}_w(\sk_{\lambda ^+})$.
Its $\fb$-socle is $\sk_\lambda$ \cite[Lemma 2.2.9]{WvdKTata}.
\end{definition}

\begin{lemma}\label{lem:Pindependent}
If $\lambda\in X(\bt)_+$ and $w\in W$, 
then $P(w\lambda)={\sH}_{w}(\sk_{\lambda})$.
For instance, $P(w_0\lambda)=\nabla_\lambda$.
\end{lemma}

\begin{proof}
See \cite[Lemma 2.3.1]{WvdKTata}. 
\end{proof}
\subsection{\unboldmath Relative Schubert modules}\label{subsec:Schubfilt}

\begin{definition}[Relative Schubert modules $Q(\lambda)$]\label{def:defineQ}

Let $\lambda=w\lambda^+$ as above. 
The relative Schubert module $Q(\lambda)$ is set to be the kernel of the 
restriction homomorphism $P(\lambda)\rightarrow {H}^0(\partial \fX_w,\Ll (\lambda ^+))$, where the boundary $\partial \fX_w$ is the union of the $\fX_v$ that are strictly contained in $\fX_w$. This restriction homomorphism is surjective by 
\cite[Theorems 1.4.8 and 2.3.2]{BK}, cf.\ Proposition \ref{prop:Schubert module} below. The $\fb$-socle of $Q(\lambda)$ is also $\sk_\lambda$ \cite[Remark 2.3.5]{WvdKTata}.
\end{definition}

\subsection{\unboldmath Excellent order}\label{subsec:excord}

Recall that we have fixed a Weyl group invariant inner product $(~,~)$ on $\mathbb R\otimes _{\mathbb Z}X(\bt)$.

\begin{definition}(\cite{WvdK})\label{def:excord}
Let $\lambda, \mu \in X(\bt)$. Define that $\lambda$ is less than $\mu$ in the excellent order, notation $\lambda \leq_e \mu$, if either $(\lambda, \lambda)<(\mu,\mu)$ or $\lambda = w\nu$, $\mu=z\nu$ for some $\nu \in X(\bt)_+$, $w,z\in W$ with $w\leq z$ (in the Bruhat order on $W$).
\end{definition}

\subsection{\unboldmath Antipodal excellent order}\label{subsec:antipodexcord}

\begin{definition}\label{def:antipodal}(\cite{WvdK})
Define that $\lambda \leq_a \mu$ in the antipodal excellent order if $-\lambda \leq_e -\mu$ in the excellent order.
\end{definition}

\subsection{\unboldmath Two highest weight category structures on $\rep( \fb)$}\label{subsec:hwc_repB}

\begin{theorem}(\cite[Theorem 1.6, (i)]{WvdK})\label{th:Pexc}
The category  ${\mathrm{rep} }( \fb)$ of finite-dimensional rational representations  of $ \fb
$ is a highest weight category with respect to the excellent order.
The $P(\lambda)$ are the costandard modules with socle $\sk_\lambda$
for this order.
\end{theorem}

\begin{remark}
Originally  it was assumed that in a highest weight category there are enough injectives (\cite{CPSHW}). But $\rep( \fb)$ does not have enough injectives. That is why in 
\cite[Theorem 1.6, (i)]{WvdK} one finds not $\rep( \fb)$ but $\Rep( \fb)$, even though the costandard modules are in $\rep( \fb)$.
Nowadays one also considers $(\rep( \fb),\leq_e)$ a highest weight category, in fact
a lower finite highest weight (lfhw) category in the terminology of \cite[section 3]{Coulembier}. 
We will follow the conventions of \cite{Coulembier}, so that both $(\rep( \fb),\leq_e)$ and $(\Rep( \fb),\leq_e)$ are highest weight categories.
Similar remarks apply to the next theorem.
\end{remark}

\begin{theorem}(\cite[Theorem 1.6, (ii)]{WvdK})\label{th:Qant}
The category  ${\mathrm{rep} }( \fb)$ of finite-dimensional rational representations  of $\fb$ is a highest weight category with respect to the antipodal excellent order.
The $Q(\lambda)$ are the costandard modules with socle $\sk_\lambda$
for this order.
\end{theorem}

\subsection{\unboldmath Cohomology vanishing for $\fb$-modules}

\begin{proposition}\label{prop:cohinduced}(\cite[II Proposition 4.13, B.4]{Jan})
Let $\lambda\in X(\bt)_+$. Let $\sk$ be $\Z$ or a field. Then
\begin{equation*}
\Ext ^i_{\bg}(\sk,\nabla_\lambda) = \Ext ^i_{\fb}(\sk,\sk_\lambda)=\left\{
   \begin{array}{l}
    {\sf k}, \quad {\rm for} \quad i=0, \  \lambda=0,\\
    0, \quad {\rm otherwise}. \\
   \end{array}
  \right.
\end{equation*}
\end{proposition}

\begin{theorem}(\cite[Theorem 2.20(i)]{WvdK}, \cite[Theorem 3.2.6]{WvdKTata})\label{th:Tata Theorem 3.2.6}
Let $\lambda, \mu\in X(\bt)$. Then $H^p( \fb,P(\lambda)\otimes Q(\mu))=0$ for $p>0$.
\end{theorem}

This will be subsumed by Corollary  \ref{cor:PQnabla}.

\begin{definition}
A weight $\lambda$ of a $\fb$-module $M$ is called \emph{extremal} if all other weights $\mu$
of $M$ are shorter or equally long:
$(\mu,\mu)\leq(\lambda,\lambda)$.   For instance, if  all other weights of $M$ are in the convex hull of
the $W$-orbit of $\lambda$, then $\lambda$ is extremal.
\end{definition}

\begin{lemma}\label{lem:extrPQ}
Let $\lambda\in X(\bt)$. Then $\lambda$ is the unique extremal weight of $Q(\lambda)$ and it is an extremal weight of $P(\lambda)$.
\end{lemma}
\begin{proof}
See \cite[Remark 2.3.5]{WvdKTata} and  \cite[Proposition 2.2.15]{WvdKTata}.
\end{proof}

\subsection{\unboldmath Round truncations}\label{subsec:truncations}

If $R>0$, then $\Rep(\fb)_{\leq R}$ denotes the full subcategory of $\Rep(\fb)$
whose objects are $\fb$-modules all whose weights satisfy $(\lambda,\lambda)\leq R$.
The category $\rep(\fb)_{\leq R}$ is defined similarly.
Both $\Rep(\fb)_{\leq R}$ and $\rep(\fb)_{\leq R}$ have enough injectives \cite[Theorem 3.5, Theorem 3.6]{CPSHW}.

\begin{theorem}\label{thm:inc}Let $R>0$. 
\begin{enumerate}
\item
The natural inclusion  $\rep(\fb)_{\leq R}\to\Rep(\fb)$
induces a full embedding $$D^b(\rep(\fb)_{\leq R})\to D^b(\Rep(\fb))$$
of triangulated categories.
\item
The natural inclusion  $\rep(\fb)\to\Rep(\fb)$
induces a full embedding $$D^b(\rep(\fb))\to D^b(\Rep(\fb))$$
of triangulated categories.
\end{enumerate}
\end{theorem}

\begin{proof}
The first part follows by combining \cite[Thm 3.5, Thm 3.9]{CPSHW} with Theorem
\ref{th:Pexc} or rather its original formulation in \cite{WvdK}.
The second part follows from \cite[Theorem 2.3.1]{Coulembier}.
\end{proof}

\subsection{\unboldmath $\fb$-cohomological duality}\label{subsec:key_orthogonality}

\begin{theorem}\label{th:key_orthogonality}
Let $\lambda, \mu\in X(\bt)$. 
Then $P(\lambda)$,  $Q(\mu)$ satisfy:
\begin{equation}\label{eq:P-Q-duality}
\Ext ^i_{\fb}(P(\lambda)^*,Q(\mu)) = \left\{
   \begin{array}{l}
    {\sf k}, \quad {\rm for} \quad i=0, \  \lambda=-\mu,\\
    0, \quad {\rm otherwise}. \\
   \end{array}
  \right.
\end{equation}
\end{theorem}
\begin{proof}
For $i>0$ this is Theorem \ref{th:Tata Theorem 3.2.6} with Lemma \ref{lem:dualmodule}.
Let $i=0$ and put $R=(\lambda,\lambda)$.
If $(\mu,\mu)>R$ consider a $\fb$-module map $f:P(\lambda)^*\to Q(\mu)$.
Its image lies in $\rep(\fb)_{\leq R}$, so it does not contain the socle of $Q(\mu)$ and thus the image of $f$ must be zero.
Next suppose $(\mu,\mu)\leq R$. As $\Ext ^0_{\fb}(P(\lambda)^*,Q(\mu))\cong
\Ext ^0_{\fb}(Q(\mu)^*,P(\lambda))$, we now consider a 
 $\fb$-module map $f:Q(\mu)^*\to P(\lambda)$. The only extremal weight of 
 $Q(\mu)^*$ is $-\mu$ by Lemma \ref{lem:extrPQ}, so if $f$ is nonzero, then $f$ must map the weight space
 $(Q(\mu)^*)_{-\mu}$ onto the socle $\sk_\lambda$ of $P(\lambda)$.
 So we must have $\lambda=-\mu$.
 If $\lambda=-\mu$, then we have a $\fb$-module map from $Q(\mu)^*$ onto its head $\sk_\lambda$ and a nonzero $f$ exists. 
Moreover, $f$ is determined by its restriction to the weight space $(Q(\mu)^*)_{-\mu}$, because the difference $f-g$ of two maps $f$, $g$ with the same restriction to this weight space has an image that does not contain the socle of $P(-\mu)$.
\end{proof}

\subsection{Standard modules}\label{subsec:standard}

We will be describing our main constructions in terms of the costandard modules of Theorems \ref{th:Pexc} and \ref{th:Qant}. For completeness let us now discuss standard modules.
Let $\leq_{hw}$ denote one of $\leq_e$ and $\leq_a$ and let $\leq_{hw}^*$
denote the other one, so that $$\{\leq_{hw},\leq_{hw}^*\}=\{\leq_e,\leq_a\}$$ as sets. Let $R>0$.
Sending a $\fb$-module $M$ to its dual $M^*$ induces an anti-autoequivalence 
$\tau$ of $\rep(\fb)_{\leq R}$. Comparing Theorem \ref{th:key_orthogonality} with 
\cite[Proposition 3.12]{Ef}, and using Lemma \ref{lem:dualmodule}, one sees that $\tau$ sends a costandard module for 
$\leq_{hw}$ with  socle $\sk_\mu$ to a standard module for $\leq_{hw}^*$ with head $\tau(\sk_\mu)=\sk_{-\mu}$.
For instance, if $\lambda\in X(\bt)$, then $P(-\lambda)^*$ is the standard module
 for $\leq_a$ with head $\sk_\lambda$.
\subsection{\unboldmath Filtrations on $\fb$-modules}\label{subsec:filtrations}

We recall here the fundamental statements concerning filtrations on $\fb$-modules. 
As we emphasized above, one has to take some care when citing the literature, 
because our $\fb$ is not $\fb^+$.
We have to convert to our conventions.
That means that dominant often becomes antidominant and vice versa. 

We begin with recalling the definition of good filtration on a $ \bg$-module, \cite{Donkin1}.

\begin{definition}
A rational $ \bg$-module $M$ is said to admit a good filtration provided that there exists a finite or infinite increasing filtration
$$
0=M_0\subsetq M_1\subsetq M_2\subsetq \dots \subsetq M_n\subsetq \dots ; \bigcup _{i>0} M_i = M
$$
such that each $M_i/M_{i-1}=\nabla _{\lambda_i}$ for some $\lambda_i \in X(\bt)_+$.
\end{definition}
\begin{aside}
An unfortunate side effect of this definition is that if $M$ admits a good filtration, its dimension is at most countable, even if $\bg$ acts trivially on $M$. 
If one wants to keep representation theory clean, it is better to use \cite[Definition 4.1.1]{WvdKTata}, 
which is similar to the next definition.
\end{aside}

\begin{definition}\cite[Definition 2.3.6]{WvdKTata}
A $\fb$-module $M$ is said to have an excellent filtration if and only if there exists a filtration $0\subset F_0\subset F_1\subset \dots $ by $\fb$-modules such that $\cup F_i= M$ and $F_i/F_{i-1}=\oplus P(\lambda _i)$ for some $\lambda _i\in X(\bt)$. Here $\oplus$ stands for 
the direct sum of any number of copies, ranging from zero copies to infinitely many. We call a nonzero $F_i/F_{i-1}$ a layer of the filtration.
\end{definition}

\begin{remark}
Finite-dimensional modules possessing excellent filtrations are called {\it strong} modules in \cite{Math}.
\end{remark}

\begin{definition}\cite[Definition 2.3.8]{WvdKTata}
A $\fb$-module $M$ is said to have a relative Schubert filtration if and only if there exists a filtration $0\subset F_0\subset F_1\subset \dots $ by $\fb$-modules such that $\cup F_i= M$ and $F_i/F_{i-1}=\oplus Q(\lambda _i)$ for some $\lambda _i\in X(\bt)$. Here $\oplus$ stands for the direct sum of any number of copies, ranging from zero copies to infinitely many. We call a nonzero $F_i/F_{i-1}$ a layer of the filtration.
\end{definition}

\begin{remark}
Finite-dimensional $\fb$-modules possessing a relative Schubert  filtration 
are called {\it weak} modules in \cite{Math}.
\end{remark}

\begin{theorem}\cite[Corollary 5.2.7]{WvdKTata}\label{th:Joseph_conj}
Let $\lambda \in X(\bt)$ and $\mu \in X(\bt)_+$. Then $P(\lambda)\otimes \nabla_\mu$ has an excellent filtration. 
\end{theorem}
\begin{proof}This relies on the main results of \cite{MathG}.

   For completeness we recall the argument. Let $\lambda^+$ be the dominant weight in the Weyl group orbit of $\lambda$ and let $w$ be minimal so that $\lambda=w\lambda^+$. By repeated application of  the  Tensor Identity \cite[I Proposition 4.8]{Jan}, we have $P(\lambda)\otimes \nabla_\mu={\sH}_w(\sk_{\lambda^+}\otimes \nabla_\mu)$. Now $\sk_{\lambda^+}\otimes \nabla_\mu$
   has an excellent filtration by \cite[\S 5, Corollary 1]{MathG}.
   The result in the Theorem thus follows from \cite[Lemma 2.11]{WvdK}.
\end{proof}

The next result is a key ingredient in our proofs. 

\begin{corollary}\label{cor:PQnabla}
    Let $\lambda,\mu\in X(\bt)$, $\nu\in X(\bt)_+$. Then $H^p( \fb,P(\lambda)\otimes Q(\mu)\otimes \nabla_\nu)=0$ for $p>0$.
\end{corollary}
\begin{proof}
Combine Theorem \ref{th:Tata Theorem 3.2.6}
with Theorem \ref{th:Joseph_conj}.
\end{proof}

\begin{remark}
The vanishing statement of Corollary \ref{cor:PQnabla} exhibits a remarkable interplay among the three highest weight category structures: one on the category ${\mathrm{Rep} }( \bg)$ from Section \ref{subsec:hwc_repG}, 
and the other two on $\rep( \fb)$ from  Section \ref{subsec:hwc_repB}.
Note also that $\sk = P(0)=Q(0)=\nabla_0$ is costandard in all three highest weight category structures.
\end{remark}

\subsection{\unboldmath Forms of the $P(\lambda)$, $Q(\mu)$ over the integers}
For the results over $\Z$ we will use  $\Z$-forms $P(\lambda)_\Z$, $Q(\mu)_\Z$ of  $P(\lambda)$, $Q(\mu)$ respectively.
They are constructed and studied in \cite[Chapter 7]{WvdKTata}. Both  $P(\lambda)_\Z$ and $Q(\mu)_\Z$ are finitely generated and free over $\Z$. 
They share many properties with their counterparts over fields.
We often drop the $\Z$ from  the notation $P(\lambda)_\Z$, $Q(\mu)_\Z$.

\begin{remark}
{We do not pursue
 the highest weight category structures in the sense of \cite[Section 3]{Ef}
on the exact category of $\Bm$-modules that are finitely generated and free over $\Z$.
 Instead, we will invoke the Universal coefficient Theorem (Theorem \ref{th:universal coefficients}) for cohomology to move back and forth between the case that $\sk$ is a field and the case $\sk=\Z$. We need that tool anyway.}
 \end{remark}
\section{\unboldmath Rappels: Weyl groups}\label{sec:Weylgroups}

\subsection{\unboldmath The Demazure product}\label{subsec:Demazureproduct}

Let $(W,S)$ be the Coxeter presentation of the Weyl group of  $ \bg$. 
The Demazure product on $W$, also known as the Hecke product, is an associative product on $W$ defined by replacing the relation $s_{\alpha}\cdot s_{\alpha}=e$ for a simple root $\alpha$ with $s_{\alpha}\star s_{\alpha}=s_{\alpha}$. 
Equivalently, for any $w\in W$ and any simple root $\alpha$, define $w\star s_{\alpha}$ to be the longer of $ws_{\alpha}$ or $w$ 
\cite[Section 3]{BM}.
If a sequence $s_{1},\dots ,s_{{t}}$ of simple reflections defines a 
reduced expression, then the Demazure product $s_{1}\star \dots \star s_{{t}}$ 
is equal to the ordinary product $s_{1} \cdots s_{{t}}$.
If $v,w\in W$, then $\overline{\fb v\fb}~\overline{\fb w\fb}=\overline{\fb (v\star w)\fb}$.

\

\begin{lemma}\label{lem:Hwstar}
    If $v,w\in W$, then ${\sH}_v\circ {\sH}_w={\sH}_{v\star w}$.
\end{lemma}
\begin{proof}
    See \cite[1.3., Corollary and Theorem 3.1]{CPSHW} or \cite[Proposition 2.2.5]{WvdKTata}.
\end{proof}
%

\subsection{\unboldmath Total order \unboldmath $\prec$ on $W$, refining the Bruhat order}\label{subsec:totalorderonW}

  Choose a total order $\prec$ on $W$  that refines the Bruhat order $<$.
Thus $w\leq v$ implies $w\preceq v$. And $w\succ v$ implies $w\nleq v$.
To get a nice fit with Theorem \ref{th:triang} one may also arrange that $w\preceq v$ implies $ww_0\succeq vw_0$. This is optional.

\subsection{\unboldmath Combinatorics}\label{subsec:comb}

 First we recall some facts from \cite{Humphreys} about $W$ and its action on $X(\bt)$. 
\begin{lemma}\label{lew0reverse} Let $v,w\in W$. The following are equivalent
\begin{itemize}
\item $v<w$ \item $v^{-1}<w^{-1}$ \item $w_0w<w_0v$ \item $ww_0<vw_0$.
\end{itemize}
\end{lemma}
\begin{proof}For the first three, see \cite[page 119]{Humphreys}.
Use that $(ww_0)^{-1}=w_0w^{-1}$, $(vw_0)^{-1}=w_0v^{-1}$.
\end{proof}

\begin{lemma}\label{lem:ws shorter} Let $w\in W$ and let $\alpha $ be a simple root.
Then $\ell(ws_\alpha)>\ell(w)$ if and only if $w(\alpha)>0$.
\end{lemma}
\begin{proof}See \cite[page 116]{Humphreys}.
\end{proof}

\begin{lemma}\label{lem:in w0}
Let $w,z\in W$ with $\ell(wz)=\ell(w)+\ell(z)$. Then $w\leq w_0z^{-1}$.
\end{lemma}

\begin{proof}
If $wz=w_0$, then it is clear. We argue by induction on $\ell(w_0)-\ell(wz)$.
If $\ell(w_0)>\ell(wz)$, then there is a simple reflection $s$ with $\ell(swz)>\ell(wz)$.
We have $\ell(sw)=\ell(w)+1$ and $\ell(swz)=\ell(sw)+\ell(z)$. So by induction
hypothesis $sw\leq w_0z^{-1}$. And thus $w\leq sw\leq w_0z^{-1}$.
\end{proof}

\begin{definition}\label{def:WI}
Let $I$ be a subset of the set $S$ of simple reflections.
The subgroup  of $W$ generated by $I$ is called a parabolic subgroup of $W$ and it
is denoted $W_I$.
\end{definition}

\begin{lemma}
Let $\lambda$ be dominant and let $I$ be the set of simple reflections that fix 
$\lambda$.
Then  the stabilizer of $\lambda$ in $W$ is the parabolic subgroup $W_I$. 
\end{lemma}
\begin{proof}
See \cite[Prop. 1.15]{Humphreys}.
\end{proof}

\begin{lemma}\label{lem:minimal-coset-representative}
Let $I$ be a subset of the set $S$.
\\
One puts $W^I=\{\;w\in W\mid \ell(ws)>\ell(w) \mbox{ for every simple reflection $s\in I$}\; \}$.
Every $w\in W$ may be written uniquely as $vz$, where $z\in W_I$ and $v\in W^I$ is the minimal coset representative of the coset $wW_I$, \emph{i.e.} the unique element of minimal length in $wW_I$. 
The assignment of $v$ to $w$ respects the Bruhat order.
\end{lemma}

\begin{proof}Compare  \cite[1.10, 1.15]{Humphreys}. To prove the final statement,
let $w,w'\in W$ with $w'\leq w$. Write $w=vz$, $w'=v'z'$ with $v,v'\in W^I$ and
$z,z'\in W_I$. Choose a reduced expression $s_1\dots s_n$ for $w$ that is the concatenation of a reduced 
expression for $v$ and a reduced expression for $z$. As $w'\leq w$ we can describe $w'$ by a substring of $s_1\dots s_n$. So $w'=uz''$ with $u\leq v$ and $z''\leq z$. Note that $z''\in W_I$. 
Now $v'$ is the minimal coset representative of $uW_I$, and we get $v'\leq u\leq v$.
\end{proof}
For comparison with \cite[Section 3]{Ana}, note that, if $W_I$ is the stabilizer of a weight $\lambda$ and $v\in W^I$, then $\ell(v)$ is the number of reflecting hyperplanes that separate $v\lambda$ from $\lambda$.
\begin{remark}
The last statement in  Lemma \ref{lem:minimal-coset-representative} is useful to understand the transitivity of the partial order $\leq_e$, which is not immediate from its definition.
\end{remark}

 \begin{lemma}\label{lem:starsmaller}
 Let $v,w\in W$. There are $v'\leq v$, $w'\leq w$ so that $v'w=v\star w=vw'$,
$ \ell(v')+\ell(w)=\ell(v\star w)=\ell(v)+\ell(w')$.
 \end{lemma}

\begin{proof}
 See \cite[Lemma 1]{He}.
\end{proof}

\begin{lemma}\label{lem:caps}
Let $x$, $y$, $s\in W$ with $s$ simple and $x\leq y\star s$. Choose $v$ minimal so that $v\star s= x\star s$.
Then $v\leq y$.
\end{lemma}
\begin{proof}We use Lemma \ref{lem:minimal-coset-representative}.
Note that $v$ is the minimal representative of the coset $x\langle s\rangle$. If $w$ is the 
minimal representative of the coset $y\langle s\rangle$, then $v\leq w\leq y$.
\end{proof}

\begin{lemma}\label{lem:ea}
Let $\lambda$, $\mu$ be weights in the same $W$-orbit. Then $\lambda<_e\mu$ if and only if $\lambda>_a\mu$. 
\end{lemma}
\begin{proof}
If $\lambda\leq_e\mu$, then there is a dominant weight $\omega$ and there are $w, z\in W$ with $w\leq z$, $\lambda=w\omega$, $\mu=z\omega$. Then $ww_0\geq zw_0$ by Lemma \ref{lew0reverse} and $-\lambda=ww_0(-w_0\omega)$, $-\mu=zw_0(-w_0\omega)$, so $-\lambda\geq_e-\mu$, so $\lambda\geq_a\mu$. The converse is proved similarly.
\end{proof}

\section{\unboldmath \unboldmath$ \bt$-equivariant $\sK$-theory of $ \bg / \fb$}\label{sec:rappelT-equivGmodB}
In this section we introduce the theorem on triangular transition matrices. Let $\sk$ be a field.
\subsection{\unboldmath Two fixed points}

We will need results from Graham--Kumar \cite{GrahamKumar} on the $ \bt$-equivariant $\sK$-theory $\sK_{\bt}( \bg/\fb)$ of $ \bg/ \fb$.
Their $\fb$ is our $ \fb^+$. As both $\fb$ and $\fb^+$ will be needed, let us consider the $\bg$-variety  
$\mathcal B$ of Borel subgroups.
If $x$ is a  point in $\mathcal B$, then its stabilizer $ \fb(x)$ is a Borel subgroup and one identifies
$\mathcal B$ with $ \bg / \fb(x)$.
Let $x_+$ be the $ \bt$-fixed point the stabilizer of which is $ \fb(x_+)= \fb^+$. Similarly, let $x_-$ be the $ \bt$-fixed point the stabilizer of which is $ \fb(x_-)= \fb$.
Choose a representative $\dot{w_0}$ of $w_0$.
One has $x_+= \dot{w_0}x_-$. We simply write $x_+= {w_0}x_-$.
Let $\phi_0$ be the isomorphism $ \fb^+\to\fb$ sending $b$ to 
$\dot{w_0}b\dot{w_0}^{-1}$. 
If $\F$ is a $ \bg $-equivariant vector bundle on $\mathcal B$, then the fibre $\F_x$ is a $ \fb(x)$-module for $x\in \mathcal B$.
We have $\F_{x_+}\cong \phi_0^*(\F_{x_-})$.
Recall that when $M$ is a finite-dimensional $\fb$-module, we denote as in \cite{Jan} by $\Ll(M)$ the $ \bg $-equivariant vector bundle 
$\F$ with $\F_{x_-}=M$. If $M$ is a finite-dimensional $ \fb^+$-module, we denote by $\Ll^+(M)$ the $ \bg $-equivariant vector bundle 
$\F$ with $\F_{x_+}=M$. 

If $\F$ is a $ \bt$-equivariant coherent sheaf on $ \bg / \fb$ and $N$ is a finite-dimensional $ \bt$-module, then
$$\F\otimes_\sk N$$ is a $ \bt$-equivariant coherent sheaf.

When working with $ \fb^+$ one should define excellent filtrations and relative Schubert filtrations in terms
of $P^+(\lambda):=\phi_0^* (P(w_0\lambda))$ and
$Q^+(\lambda):= \phi_0^* (Q(w_0\lambda))$. They have $ \fb^+$-socles of weight $\lambda$ and are the costandard
modules of \cite{WvdK}.

One has the following counterpart to Theorem \ref{th:key_orthogonality}:
 $$H^i( \fb^+,P^+(\lambda)\otimes_\sk Q^+(\mu))=
\begin{cases}\sk &\text{if $i=0$ and $\lambda+\mu=0$}\\
0&\text{else.}
\end{cases}
$$

\subsection{\unboldmath The Steinberg basis}\label{subsec:Steinberg}

For $v\in W$ the Steinberg weight $e_v$ is given by $$e_v = v^{-1} \sum_{\alpha\in\Pi,\;v^{-1}\alpha<0} \omega_\alpha.$$
The Steinberg basis $\{\sk_{e_v}\}_{v\in W}$ consists of the corresponding one-dimensional
$\fb$-modules.

More precisely, the classes $[\sk_{e_v}]$ provide by \cite{St} a basis of the representation ring $R( \bt)=\sK_0(\rep( \bt ))$ as a module over the representation ring $R( \bg )=\sK_0(\rep( \bg ))$.
Recall that $R(\bt)$ is  isomorphic to
the group ring over $\Z$ of $X(\bt)$ and that
$R( \bg )$ equals $R( \bt)^W$, the subring of elements invariant under the Weyl group action.
We will find later (Remark \ref{cor:Steinberg_alternative}) that one still gets a basis if one replaces a few $\sk_{e_v}$ by $Q(e_v)$ or $P(-e_v)^*$.
\begin{remark}
    The order $\prec$ gives an order on the Steinberg basis, often in conflict with $<_a$ and $>_e$.
\end{remark}
Generators of $R( \bt)$ are often written $e^\lambda$ instead of $[\sk_\lambda]$.
We will find later (Remark \ref{rem:Steinberg from anti}) that the Steinberg weights can also be described as follows: a weight $\lambda$ is a Steinberg weight if and only if $e^\lambda$ does not lie in the $R(\bg)$-submodule of $R(\bt)$
generated by the $e^\mu$ with $\mu<_a\lambda$.

\subsection{\unboldmath Schubert varieties and opposite Schubert varieties}
\label{subsec:alphabet}
Let the Schubert variety $\fX^+_w$ be the closure of $ \fb^+wx_+$ and let the opposite Schubert variety
$\fX^w$ be the closure of 
$ \fb^-wx_+$. Its ``boundary'' $\partial \fX^w$ is the union of the $\fX^v$ that are strictly contained in $\fX^w$.

The $ \bt$-equivariant $\sK$-theory $\sK_{\bt}( \bg/\fb)$ of $ \bg / \fb$ is a module for the representation ring $R( \bt)$.
If $M$ is a finitely generated $\bt$-module, then its class $[M]$ in $R( \bt)$
is also known as $\Char(M)$, the formal
character of $M$.

If $\F$ is a $ \bt$-equivariant coherent  sheaf and $N$ is a finite-dimensional $ \bt$-module, then $\F\otimes_\sk N $
represents $[N]\cdot[\F]$.

If $\sk=\cC$, then we learn from \cite{GrahamKumar} that the $R( \bt)$-module $\sK_{\bt}( \bg/\fb)$ has a Schubert basis $\{[\Oo_{\fX^+_w}]\}_{w\in W}$.
It also has an `opposite Bruhat cell' basis $\{[\Oo_{\fX^w}(-\partial \fX^w)]\}_{w\in W}$.  
They are orthogonal  under the $R( \bt)$-bilinear symmetric
pairing $\langle -, -\rangle$ on $\sK_{\bt}( \bg/\fb)$ given by 
$$\langle [\F],[\G] \rangle=\sum_i(-1)^i[H^i( \bg / \fb,\F\otimes \G )]\in R( \bt).$$
That is
$$\langle [\Oo_{\fX^+_v}],[\Oo_{\fX^w}(-\partial \fX^w)]\rangle=\delta_{v,w} . $$
The pairing makes sense over any field $\sk$.

If $\F\otimes\G$ is supported on a $ \bt$-stable closed subscheme $Y$, then 
$$\langle [\F],[\G] \rangle=\sum_i(-1)^i[H^i(Y,\F\otimes\G)]\in R( \bt).$$

We put $$\tP_v=\Ll(P(-e_v))$$
and $$\tQ_v=\Ll(Q(e_v)).$$
They are $ \bg $-equivariant vector bundles, but we also view them as $ \bt$-equivariant vector bundles.

Write $\alpha_{vw}=\langle [\Oo_{\fX^+_w}],[\tQ_v] \rangle$. So
$$[\tQ_v]=\sum_w \alpha_{vw}[\Oo_{\fX^w}(-\partial \fX^w)]$$
when $\sk=\cC$.

Write $\beta_{vw}=\langle [\Oo_{\fX^w}(-\partial \fX^w)],[\tP_v] \rangle$.
So
$$[\tP_v]=\sum_w \beta_{vw}[\Oo_{\fX^+_w}]$$
when $\sk=\cC$.

Our main result concerning these matrices is
that, with a suitable reordering of rows and columns, the matrices 
$(\alpha_{wv})$ and $(\beta_{vw})$ are upper triangular and invertible. 

\begin{theorem}[Triangular transition matrices]\label{th:triang}
\
\begin{enumerate}
    \item $\alpha_{vw}=\langle [\Oo_{\fX^+_w}],[\tQ_v] \rangle$ vanishes unless 
    $w\leq v\,w_0$ in the Bruhat order.
    \item $\beta_{vw}=\rule{0pt}{1.1em}\langle [\Oo_{\fX^w}(-\partial \fX^w)],[\tP_v] \rangle$ vanishes unless $v\,w_0  \leq w$
    in the Bruhat order.
    \item If $v\,w_0  = w$,  then $\alpha_{vw}=\rule{0pt}{1.1em}[\sk_{ve_v}]$, $\beta_{vw}=[\sk_{-ve_v}]$.
\end{enumerate}
\end{theorem}

\section{\unboldmath Triangularity of transition matrices}\label{sec:triangular}

In this section
we will prove Theorem \ref{th:triang}. Let $\sk$ be a field in this section.

\subsection{Comparing extremal weights}
A  closed subset  of $\bg/\fb$ is $\fb$-invariant if and only if it is a union of Schubert varieties.
So there are only finitely many $\fb$-invariant  closed subsets of $\bg/\fb$.
\begin{proposition}    
\label{prop:Schubert module}
    Let $S, S'$ be  unions of Schubert varieties  in $ \bg /{\fb}$ and let $\lambda$ be dominant. 
\begin{itemize}\item The extremal weights of $\Gamma(S,\Ll(\lambda))$ are the $w\lambda$ 
with $wx_-\in S$. These weights have multiplicity one and $\Gamma(S,\Ll(\lambda))$ has a relative
Schubert filtration.
In a relative Schubert filtration of $\Gamma(S,\Ll(\lambda))$ the layers are the $Q(\mu)$ with $\mu$ an extremal weight of $\Gamma(S,\Ll(\lambda))$.
\item If $S'\subset S$, then $\Gamma(S,\Ll(\lambda))\to\Gamma(S',\Ll(\lambda))$ is surjective, and its kernel $M$ has a relative Schubert filtration.
The extremal weights of $M$ are the extremal weights of $\Gamma(S,\Ll(\lambda))$ that are not extremal weights of $\Gamma(S',\Ll(\lambda))$.
In a relative Schubert filtration of $M$ the layers are the $Q(\mu)$ with $\mu$ an extremal weight of $M$.
\item $H^i(S,\Ll(\lambda))=0$ for $i>0$.
\item If $M\in \rep(\fb)$ has a relative Schubert filtration and $w\in W$, then $M$ is $\sH_w$-acyclic.
\end{itemize}
\end{proposition}
\begin{proof}
    See  \cite[Proposition 2.2.15, Lemma 2.3.10, Proposition 2.3.11, Lemma 2.2.11,  Proposition A.2.6]{WvdKTata}, \cite[Theorem 1.2.8, Chapter 2]{BK}, Lemma \ref{lem:extrPQ},
    \cite[Theorem 1.9.(a)(ii)]{WvdK}.
\end{proof}

\begin{proposition}\label{prop:1dimQ}Let $\mu$ be a weight with  $|(\alpha^\vee ,\mu)|\leq1$ for all simple roots $\alpha$.
Then
 $Q(\mu)=\sk_{\mu}$.
 \end{proposition}
 \begin{proof}
 
 We may assume $\sk=\cC$, because of the base change properties \cite[Chapter 7]{WvdKTata} of the $Q(\mu)$. (When the characteristic of $\sk$ divides 
 a structure constant, a direct reasoning is more technical.)
 Suppose $Q(\mu)$ is larger than its socle $\sk_{\mu}$. Let $\nu$ be a weight of the socle of $Q(\mu)/\sk_{\mu}$
 and let $f$ be a nonzero weight vector of $Q(\mu)$, of weight $\nu$, 
 mapping to a vector of the socle of $Q(\mu)/\sk_{\mu}$.
 As the socle of $Q(\mu)$ is its weight space $\sk_{\mu}$, the vector  $f$ can not be fixed by the unipotent radical $\bu$
 of $\fb$. But it is well known
 that  $\bu(\cC)$ is generated by the $x_{-\alpha}(t)$ with $\alpha$ simple. So there must be such an $x_{-\alpha}(t)$ with 
 $x_{-\alpha}(t)(f)-f$ a nonzero vector in $\sk_{\mu}$. In particular, $\mu=\nu + n(-\alpha)$ for some integer $n$. But $\nu$
 must be strictly shorter than the unique extremal weight $\mu$ of $Q(\mu)$. However, there is no strictly shorter 
 weight in $\mu +\Z\alpha$, because for $n\in \Z$ one has 
 $$\frac{(\mu+n\alpha,\mu+n\alpha)-(\mu,\mu)}{(\alpha,\alpha)}=n(\alpha^\vee,\mu)+n^2\geq0.$$
  \end{proof}
 
  \begin{lemma}\label{lem:Q acyclic}Let $w\in W$.
If $M$ is a finite-dimensional $ \fb^+$-module with a relative Schubert filtration,
then $H^i(\overline{ \fb^+wx_+},\Ll^+(M))=0$ for $i>0$.
\end{lemma}
\begin{proof}
Compare Proposition \ref{prop:Schubert module}, \cite[I Proposition 5.11]{Jan}.
\end{proof}

Let $v\in W$. Let $I$ consist of the simple reflections that fix the dominant weight $ve_v$. 
 Let $s$ be simple.
 From  the definition of $e_v$ and Lemma \ref{lem:ws shorter} it follows that $\ell(v^{-1}s)>\ell(v^{-1})$ if and only if $s\in I$. 
 In particular, $v^{-1}$ is a minimal coset representative of $v^{-1}W_I$.

 \begin{proposition}\label{prop:extremalQ}
 Let also $u\in W$.
 \begin{enumerate}
 \item $-ve_v$ is an extremal weight of $\Gamma(\bg/{\fb},\Ll(-w_0ve_v))$, but not of 
 $$\Gamma((\partial \overline{{\fb}v\fb})\;\overline{{\fb}v^{-1}w_0x_-},\Ll(-w_0ve_v)).$$
 \item If $-uve_v\neq -ve_v$, then $-uve_v$ is an extremal weight of 
 $$\Gamma((\partial \overline{{\fb}v{\fb}})\;\overline{{\fb}v^{-1}w_0x_-},\Ll(-w_0ve_v)).$$
 \end{enumerate}
 \end{proposition}
 \begin{proof}Part (1).
 By Proposition \ref{prop:Schubert module} the  extremal weights of $\Gamma(\bg/{\fb},\Ll(-w_0ve_v))$ are the elements in the $W$-orbit of $-w_0ve_v$,
 hence of $-ve_v$.
 Now suppose $-ve_v$ is  an extremal weight of $\Gamma((\partial \overline{{\fb}v{\fb}})\;\overline{{\fb}v^{-1}w_0x_-},\Ll(-w_0ve_v))$.
Then $-ve_v$ can be written as $y(-w_0ve_v)$ with $y\leq z\star v^{-1}w_0$ for some $z<v$.
Replacing $z$ by a lesser element 
if necessary (Lemma \ref{lem:starsmaller}), we may assume $y\leq zv^{-1}w_0$ for some $z<v$.
Put $u=yw_0$. Then $-ve_v=-uve_v$ with $u\geq zv^{-1}$. But $-ve_v=-uve_v$ implies $u\in W_I$, so 
$u\geq zv^{-1}$ implies $zv^{-1}\in W_I$. Thus $z^{-1}\in v^{-1}W_I$.
This contradicts the minimality of $v^{-1}$ in its coset $v^{-1}W_I$.

Part (2). Now consider a weight of the form $-uve_v$ with $-uve_v\neq -ve_v$. We may replace $u$ by its minimal coset representative.
As $\ell(u)\geq 1$ there is a simple reflection $s$ with $s\notin I$ and $\ell(u)= \ell(us)+1$. In particular, $u\geq s$.
Now $v^{-1}s<v^{-1}$ by the construction of $e_v$. Put $z=sv$, $y=uw_0$. Then $z<v$, $-uve_v=y(-w_0ve_v)$, with
$y\leq sw_0=z\star v^{-1}w_0$.
 \end{proof}

 \begin{proposition}\label{prop:vw_0 notless w}
 Let $w\in W$ such that $vw_0\nleq w$.
 Every extremal weight of $$\Gamma(\overline{{\fb}(ww_0\star v^{-1}w_0)x_-},\Ll(-w_0ve_v))$$ is also an 
 extremal weight of 
 $$\Gamma((\partial \overline{{\fb}ww_0{\fb}})\;\overline{{\fb}v^{-1}w_0x_-},\Ll(-w_0ve_v)).$$
 \end{proposition}

 \begin{proof}
 Suppose 
$\ell((ww_0)\star (v^{-1}w_0))=\ell(ww_0)+\ell(v^{-1}w_0)$. By Lemma \ref{lem:in w0}
 $(ww_0)\leq v$, so $vw_0\leq w$, contrary to our assumption.
So $\ell((ww_0)\star (v^{-1}w_0))<\ell(ww_0)+\ell(v^{-1}w_0)$ and we have by Lemma \ref{lem:starsmaller} an element 
$y<ww_0$ with
$(ww_0)\star (v^{-1}w_0)=yv^{-1}w_0 $, $\ell(yv^{-1}w_0)=\ell(y)+\ell(v^{-1}w_0)$. Every extremal weight of $\Gamma(\overline{{\fb}(ww_0\star v^{-1}w_0)x_-},\Ll(-w_0ve_v))$
is of the form $z(-w_0ve_v)$ with $z\leq yv^{-1}w_0$.
By Proposition \ref{prop:Schubert module} it is then also an extremal weight of 
$\Gamma((\partial \overline{{\fb}ww_0{\fb}})\;\overline{{\fb}v^{-1}w_0x_-},\Ll(-w_0ve_v))$.
 \end{proof}

\begin{lemma}\label{lem:capw}
Let $S'$, $S''$ be unions of Schubert varieties with inverse images $\tilde S'$, $\tilde S''$ respectively in $\bg$. Let $w\in W$. 
Then $\tilde S'(\overline{{\fb}w{\fb}}/{\fb})\cap \tilde S''(\overline{{\fb}w{\fb}}/{\fb})=(\tilde S'\cap\tilde S'')(\overline{{\fb}w{\fb}}/{\fb})$.
\end{lemma}
\begin{proof}Clearly the right-hand side is contained in the left-hand side.
Both sides are unions of Schubert varieties.
First let $w$ be simple,
say $w=s$. Consider $x\in W$ with $\overline{{\fb}x{\fb}}/{\fb}$ contained in the left hand side. Choose $v$ minimal so that $v\star s= x\star s$.
Use Lemma \ref{lem:caps} to see that $\overline{{\fb}(v\star s){\fb}}/{\fb}$ is contained in the right hand side. But then so is $\overline{{\fb}x{\fb}}/{\fb}$. Next, if $\ell(w)>1$, choose a reduced expression $s_1\cdots s_n$ of $w$ and write $\overline{{\fb}w{\fb}}$ as $\overline{{\fb}s_1{\fb}}\cdots  \overline{{\fb}s_n{\fb}}$. Show that $\tilde S'(\overline{{\fb}w{\fb}})\cap \tilde S''(\overline{{\fb}w{\fb}})=(\tilde S'\cap\tilde S'')(\overline{{\fb}w{\fb}})$.
\end{proof}
 \begin{notation}
Let $S$ be a $\fb$-invariant closed subset  of  $ \bg /{\fb}$ and let $\F$ be a $\fb$-equivariant vector bundle on $ \bg /{\fb}$. Recall that $\sk$ is a field. Put
$$\chi(S,\F)=\langle[\Oo_S],[\F]\rangle=\sum_i(-1)^i[H^i(S,\F)]$$
 in $R(\bt)$.
\end{notation}

  \begin{proposition}\label{prop:Leray}Let $S$ be a $\fb$-invariant closed subset  of  $ \bg /{\fb}$. 
  Let $\lambda$ be a dominant weight and let $w\in W$. Let  $\tilde S$ be the inverse image of $S$ in $ \bg $.
   \begin{enumerate} 
 \item There is an isomorphism $H^0(\overline{{\fb}w{\fb}} S,\Ll(\lambda))\to H^0({\overline{{\fb}w{\fb}}/{\fb}},\Ll(H^0(S,\Ll(\lambda)))$.
 \label{Leray1}
  \item  $H^i(\rule{0pt}{1.1em}{\overline{{\fb}w{\fb}}/{\fb}},\Ll(H^0(S,\Ll(\lambda)))=0$ for $i>0$.
  \label{Leray2}
 \item 
 
 $
     \chi(\tilde S\rule{0pt}{1.1em}\overline{{\fb}w{\fb}}/\fb ,\Ll(\lambda))= \chi(S,\Ll(H^0({\overline{{\fb}w{\fb}}/{\fb}},\Ll(\lambda))).
 $
 \label{Leray3}
  \end{enumerate}
 \end{proposition}
\begin{proof}
If $\sk$ is a field of finite characteristic,
then parts \ref{Leray1} and \ref{Leray2} follow from \cite[Proposition 2.24, Theorem 1.9.(a)(ii)]{WvdK}  by induction on $\ell(w)$.
Cohomology groups like $H^0(S,\Ll(\lambda))$ are finitely generated over a base ring by \cite[Thm. III 5.2]{Hartshorne}.
The map in part 1 is induced by the multiplication $\overline{{\fb}w{\fb}} \times \tilde S\to
\overline{{\fb}w{\fb}} \tilde S$. 
We may choose $N\geq1$ so that this map is defined over $\Z[1/N]$ and such that $H^1(\overline{{\fb}w{\fb}} S,\Ll(\lambda)) $ and
$\Ll(H^0(S,\Ll(\lambda)))$ are flat over $\Z[1/N]$. (This is called ``spreading out".)
By the Universal coefficient Theorem \ref{th:universal coefficients}
both parts go through for $\sk=\Z[1/N]$ instead of  $\sk$ a field of finite characteristic, and then also when $\sk$ is a field of characteristic zero.

Proof of part \ref{Leray3}. By the previous parts the result holds when $S$ is a Schubert variety. 
There are only finitely many possibilities for $S$, so we may assume the result for any union of Schubert varieties that is
strictly contained in $S$. If $S$ is not a Schubert variety and $S\neq\emptyset$, write $S=S'\cup S''$ where $S'$, $S''$ are strictly smaller unions of Schubert varieties.
Put $X=\overline{{\fb}w{\fb}}/\fb$
and $M=H^0(X,\Ll(\lambda))$.
If $\sk$ is a field of finite characteristic, then $S'\cap S''$ is reduced by Ramanathan \cite[Proposition 1.2.1, Chapter 2]{BK}. By \cite[Corollary 1.6.6]{BK}
it is then also reduced in characteristic zero.
From the Mayer--Vietoris sequence 
$$0\to  H^0(S,\Ll(M))\to H^0(S',\Ll(M))\oplus H^0(S'',\Ll(M))\to H^0(S'\cap S'',\Ll(M))\to\cdots$$
we get
\begin{equation}\label{eq:chiM}
    \chi(S,\Ll(M))=\chi(S',\Ll(M))+\chi(S'',\Ll(M))- \chi(S'\cap S'',\Ll(M)).
\end{equation}
Similarly the Mayer--Vietoris sequence 
$$0\to H^0(\tilde SX ,\Ll(\lambda))\to H^0(\tilde S'X ,\Ll(\lambda))\oplus H^0(\tilde S''X ,\Ll(\lambda))\to H^0(\tilde S'X\cap \tilde S''X ,\Ll(\lambda))\to0,$$ 
gives 
\begin{equation}\label{eq:chilambda}
    \chi(\tilde SX ,\Ll(\lambda))=\chi(\tilde S'X ,\Ll(\lambda)) + \chi (\tilde S''X ,\Ll(\lambda))- \chi(\tilde S'X\cap \tilde S''X ,\Ll(\lambda))
.\end{equation}
By Lemma \ref{lem:capw} we know that $\tilde S'X\cap \tilde S''X =(\tilde S'\cap \tilde S'')X $.
As the right hand sides of equations (\ref{eq:chiM}), (\ref{eq:chilambda})
agree, part \ref{Leray3} follows.

\end{proof}

\begin{remark}
    Along similar lines one may show that
    $H^i(S,\Ll(H^0({\overline{{\fb}w{\fb}}/{\fb}},\Ll(\lambda)))$ vanishes for $i>0$,
    but we will not need that.
\end{remark}

\subsection{\unboldmath The matrix \unboldmath$(\beta_{vw})$}\label{subsec:beta}
Recall that $\beta_{vw}=\langle [\Oo_{\fX^w}(-\partial \fX^w)],[\tP_v] \rangle$.
We want to pair $[\Oo_{\fX^w}(-\partial \fX^w)]$
with $[\tP_v]=[\Ll(P(-e_v))]$. 
As $0\to\Oo_{\fX^w}(-\partial \fX^w)\to \Oo_{\fX^w}\to \Oo_{\partial \fX^w} \to0$
is exact, it suffices to compute the difference between $\langle [ \Oo_{\fX^w}],[\Ll(P(-e_v))]\rangle$ and $\langle [ \Oo_{\partial \fX^w}],[\Ll(P(-e_v))]\rangle$.

Note that $\fX^w=\overline{{\fb}ww_0x_-}$.

And note that $P(-e_v)=\Gamma(\overline{{\fb}v^{-1}w_0x_-},\Ll(-w_0ve_v))$.

We want to prove that $\langle[\fX^w(-\partial \fX^w)],[\tP_v]\rangle$ vanishes unless $vw_0\leq w$.
But Proposition \ref{prop:Leray} shows that $\langle [ \Oo_{\fX^w}],[\Ll(P(-e_v))]\rangle$  is just the character of the module
$\Gamma(\overline{{\fb(}ww_0\star v^{-1}w_0)x_-},\Ll(-w_0ve_v))$.
Similarly, Propositions \ref{prop:Leray} and \ref{prop:Schubert module} give that  $\langle [ \Oo_{\partial \fX^w}],[\Ll(P(-e_v))]\rangle$ is just the character of the module
$\Gamma(\partial^{w,v},\Ll(-w_0ve_v))$,
where $\partial^{w,v}$ equals 
$(\partial \overline{{\fb}ww_0{\fb}})\;\overline{{\fb}v^{-1}w_0{\fb}}x_-$. Both modules have a relative Schubert filtration by Proposition \ref{prop:Schubert module}
and we can get a grip on them by inspecting the extremal weights.

There are three cases.

\begin{itemize}
\item $vw_0=w$. So that is about  the kernel of the surjective map from $\Gamma(\overline{{\fb}w_0x_-},\Ll(-w_0ve_v))$ to $\Gamma(\partial^{w,v},\Ll(-w_0ve_v))$.
By Proposition \ref{prop:extremalQ} the kernel is isomorphic to $Q(-ve_v)$ which equals $\sk_{-ve_v}$ by Proposition \ref{prop:1dimQ}.
\item $vw_0\nleq w$. Then one sees with Proposition \ref{prop:vw_0 notless w}
that 
the kernel of the surjective map from $\Gamma(\overline{{\fb}(ww_0\star v^{-1}w_0)x_-},\Ll(-w_0ve_v))$ to $\Gamma(\partial^{w,v},\Ll(-w_0ve_v))$ vanishes.
\item $vw_0< w$. No claims here.
\end{itemize}

We conclude that the matrix 
 $(\beta_{vw})=(\langle  [\Oo_{\fX^w}(-\partial \fX^w)],[\tP_v]\rangle)$ behaves  as claimed in Theorem \ref{th:triang}.

\subsection{\unboldmath The matrix \unboldmath$(\alpha_{vw})$}\label{subsec:alpha}

So let us turn to $(\alpha_{vw})=(\langle [\Oo_{\fX_w}],[\tQ_v]\rangle)$.

Now we are dealing with the situation where the Borel subgroup has positive roots, and anti-dominant weights
are to be used as in \cite{WvdKTata}. For instance,  Proposition \ref{prop:Leray} gives
 \begin{proposition}\label{prop:Leray+}
  Let $\lambda$ be an anti-dominant weight. And let $S$ be a ${\fb}^+$-invariant closed subset of $ \bg /{\fb}^+$.
   \begin{enumerate} 
 \item $H^0(\overline{{\fb}^+w{\fb}^+} S,\Ll^+(\lambda))\to H^0({\overline{{\fb}^+w{\fb}^+}/{\fb}^+},\Ll^+(H^0(S,\Ll^+(\lambda))))$ is an isomorphism.
 \label{Leray+1}
  \item  $H^i({\overline{{\fb}^+wx_+}},\Ll^+(H^0(S,\Ll^+(\lambda))))=0$ for $i>0$.
  \label{Leray+3}
 \qed\end{enumerate}
 
 \end{proposition}

By Lemma \ref{lem:Q acyclic} we have 

$\alpha_{vw}=\langle [\Oo_{\fX_w}],[\tQ_v]\rangle=[\Gamma(\overline{{\fb}^+wx_+},\Ll^+(\phi_0^* (Q(e_v))))]$. Now 
$\phi_0^* (Q(e_v))$ has $\fb^+$-socle of weight $w_0e_v$ and it is therefore the kernel of the surjection
$$\Gamma(\overline{{\fb}^+w_0v^{-1}w_0x_+},\Ll^+(w_0ve_v))\onto\Gamma(\partial(\overline{{\fb}^+w_0v^{-1}w_0x_+}),\Ll^+(w_0ve_v)).$$
Combining with Proposition \ref{prop:Leray+}
we see that $\alpha_{vw}$ is the character of the kernel of the surjective map from
$\Gamma(\overline{({\fb}^+w{\fb}^+})(\overline{{\fb}^+w_0v^{-1}w_0x_+}),\Ll^+(w_0ve_v))$ to $\Gamma(\overline{({\fb}^+w{\fb}^+})\partial(\overline{{\fb}^+w_0v^{-1}w_0x_+}),\Ll^+(w_0ve_v))$.

Let $v\in W$. Let $I$ again consist of the simple reflections that fix the dominant weight $ve_v$ and let $J=w_0Iw_0$, so that $J$ consists of the simple reflections that fix $w_0ve_v$. 
\begin{lemma}\label{lem:stoboundary}We have
\begin{enumerate}
\item $w_0v^{-1}w_0$  is a minimal coset representative in $W/W_J$.
\item $\partial\rule{0pt}{1.1em}(\overline{{\fb}^+w_0v^{-1}w_0x_+})$ contains the union of the 
$\overline{{\fb}^+w_0v^{-1}w_0sx_+}$ with $s$ simple, $s\notin J$.
\end{enumerate}
\end{lemma}
\begin{proof}Let  $\alpha$ be a simple root.
As $e_v$ is a Steinberg weight, $w_0v^{-1}(\alpha)>0$ if and only if $s_\alpha\notin I$.
So $w_0v^{-1}w_0(-w_0(\alpha))<0$ if and only if $s_{-w_0\alpha}\notin J$.
Thus if $\beta$ is a simple root, then $w_0v^{-1}w_0(\beta)<0$ if and only if $s_{\beta}\notin J$. Now use Lemma \ref{lem:ws shorter}.
\end{proof}
\begin{proposition}\label{prop:extremal2}
 Let also $u\in W$.
 \begin{enumerate}
 \item $ve_v$ is an extremal weight of $$\Gamma(\overline{({\fb}^+vw_0{\fb}^+})(\overline{{\fb}^+w_0v^{-1}w_0x_+}),\Ll^+(w_0ve_v))))$$ but not of 
$$\Gamma((\overline{{\fb}^+vw_0\fb^+})
\partial(\overline{{\fb}^+w_0v^{-1}w_0x_+}),\Ll^+(w_0ve_v)))).$$
 \item If $uve_v\neq ve_v$, then $uve_v$ is an extremal weight of $$\Gamma(\overline{({\fb}^+vw_0\fb^+})
\partial(\overline{{\fb}^+w_0v^{-1}w_0x_+}),\Ll^+(w_0ve_v)))).$$
 \end{enumerate}
 \end{proposition}
\begin{proof}Part (1).
 The extremal weights of 
 $$M:=\Gamma(\overline{({\fb}^+vw_0{\fb}^+})(\overline{{\fb}^+w_0v^{-1}w_0x_+}),\Ll^+(w_0ve_v))))$$ are  elements in the $W$-orbit of $w_0e_v$.
 We have
 $\ell(vw_0)+\ell(w_0v^{-1}w_0)=\ell(w_0)$, so  $(vw_0)\star (w_0v^{-1}w_0)=w_0$, and $w_0w_0ve_v$ is one of the weights of $M$.
 Now suppose $ve_v$ is  an extremal weight of 
$\Gamma((\overline{{\fb}^+vw_0\fb^+})
\partial(\overline{{\fb}^+w_0v^{-1}w_0x_+}),\Ll^+(w_0ve_v))))$.
Then $ve_v$ can be written as $y(w_0ve_v)$ with $y\leq vw_0\star z$ for some $z<w_0v^{-1}w_0$.
Replacing $z$ by a lesser element if necessary (Lemma \ref{lem:starsmaller}), we may assume $y\leq vw_0z$ for some $z<w_0v^{-1}w_0$.
Put $u=w_0y$. Then $ve_v=yw_0ve_v=w_0uw_0ve_v$ with $u\geq w_0vw_0z$. But $uw_0ve_v=w_0ve_v$ implies $u\in W_J$, so 
$u\geq w_0vw_0z$ implies $w_0vw_0z\in W_J$. This contradicts the minimality of $w_0v^{-1}w_0$ in its coset $w_0v^{-1}w_0W_J$ (Lemma \ref{lem:stoboundary}).

Part (2). Now consider a weight of the form $uve_v$ with $uve_v\neq ve_v$. 
Then $w_0uw_0w_0ve_v\neq w_0ve_v$ so $w_0uw_0\notin W_J$. There is a simple $s$, $s\notin J$, with $w_0uw_0\geq s$.
Then $uw_0\leq w_0s=(vw_0)\star (w_0v^{-1}w_0s)$ and $\overline{{\fb}^+w_0v^{-1}w_0sx_+}$ is contained in 
$\partial(\overline{\fb^+w_0v^{-1}w_0x_+})$ by Lemma \ref{lem:stoboundary}. So $uw_0w_0ve_v$ is an extremal weight of
$\Gamma((\overline{{\fb}^+vw_0\fb^+})
\partial(\overline{{\fb}^+w_0v^{-1}w_0x_+}),\Ll^+(w_0ve_v))))$.
 \end{proof}

 \begin{proposition}\label{prop:w notless vw_0}
 Let $w\in W$ such that $w\nleq vw_0$.
 Every extremal weight of $$\Gamma(\overline{({\fb}^+w{\fb}^+})(\overline{{\fb}^+w_0v^{-1}w_0x_+}),\Ll^+(w_0ve_v))))$$  is also an extremal weight of 
$$\Gamma(\overline{({\fb}^+w{\fb}^+})\partial(\overline{{\fb}^+w_0v^{-1}w_0x_+}),\Ll^+(w_0ve_v)))).$$

 \end{proposition}
 
 \begin{proof}
 The extremal weights of  $\Gamma(\overline{({\fb}^+w{\fb}^+})(\overline{{\fb}^+w_0v^{-1}w_0x_+}),\Ll^+(w_0ve_v))))$ are of the form
 $uw_0ve_v$ with $u\leq w\star (w_0v^{-1}w_0)$. Suppose $\ell(w)+\ell(w_0v^{-1}w_0)=\ell(ww_0v^{-1}w_0)$.
 Then $w\leq w_0w_0vw_0$, contrary to the assumption. So $\ell(w)+\ell(w_0v^{-1}w_0)<\ell(ww_0v^{-1}w_0)$
 and we may by Lemma \ref{lem:starsmaller} chose $z<w_0v^{-1}w_0$ with $w\star (w_0v^{-1}w_0)=w\star z$. As $u\leq w\star z$, we see that $uw_0ve_v$ is an extremal 
 weight of $\Gamma(\overline{({\fb}^+w{\fb}^+})\partial(\overline{{\fb}^+w_0v^{-1}w_0x_+}),\Ll^+(w_0ve_v))))$
 \end{proof}

Notice that $Q(ve_v)=\sk_{ve_v}$ because $ve_v$ is dominant.
One can now deal with the matrix  $(\alpha_{vw})$
in the same manner as for  $(\beta_{vw})$. 

This ends the proof of Theorem \ref{th:triang}.
We next develop its consequences.

\section{\unboldmath \unboldmath$\ind$-vanishing for $\fb$-modules}\label{sec:b-cohom_vanish}
In this section we draw conclusions about $\fb$-modules by combining  the vanishing in Corollary \ref{cor:PQnabla} with Theorem \ref{th:triang} on triangular transition matrices. Let $\sk$ be a field or $\Z$, until Section \ref{subsec:dualFEC}.
\subsection{\unboldmath Cohomological descent from \unboldmath$ \bg/ \fb$ to $\fb$}\label{subsec:cohdesc}

\begin{proposition}\cite[I, Proposition 5.12]{Jan}\label{prop:IndvsCoh}
There is for each $\fb$--module $M$ and each $n\in \mathbb N$ a canonical isomorphism of $\sf k$--modules
$$
{\sf For}({\rm R}^n{\rm Ind}_{\fb}^{\bg}M)\simeq {H}^n({\bg/\fb},\Ll (M)).
$$
 where $\sf For$ is the forgetful functor $\Dd ^b(\rep( \bg))\rightarrow \Dd ^b( {\kmod})$.

\end{proposition}
Let $\Ff=\Ll ({ F})$ be a $ \bg$-equivariant sheaf on $ \bg/ \fb$ given by a finitely generated $\fb$-module $ F$, which is projective over $\sk$. By  \cite[I Proposition 4.10]{Jan} and Proposition \ref{prop:IndvsCoh} there are isomorphisms 
\begin{equation}\label{eq:cohdescent}
{H}^p({\fb},\sk[{ \bg}]\otimes { F})={\sf For}(\RR^p\ind _{\fb}^{ \bg}({ F}))={H}^{p}({ \bg/\fb},\Ff),
\end{equation} 
so it is all about computing sheaf cohomology of a $ \bg$-equivariant coherent sheaf via derived induction. 
 
 \subsubsection*{\unboldmath`Peter-Weyl' filtration of $\sk[ \bg]$}
 In the usual simplified notation, the left regular representation $\rho_\ell$ of $\bg$ on $\sk[\bg]$ is defined by $\rho_\ell(g)(f):x\mapsto f(g^{-1}x)$ for $g\in\bg$, $f\in \sk[\bg]$  (\cite[I 2.7]{Jan}). 
 The right regular representation $\rho_r$ of $\bg$ on $\sk[\bg]$ is similarly defined by $\rho_r(g)(f):x\mapsto f(xg)$ for $g\in\bg$, $f\in \sk[\bg]$.
By \cite[II, Proposition 4.20, B.8]{Jan}, we know that  $\sk[ \bg]$ as a ${ \bg}\times { \bg}$-module via $\rho _l\times \rho _r$ admits a good filtration whose layers are $\nabla _{\lambda}\otimes \nabla _{-w_0\lambda}$ 
with $\lambda \in X(\bt)_+$, each occurring with multiplicity one; that is, $\sk[ \bg]=\varinjlim A_i$ such that $A_i/A_{i-1}=\nabla _{\lambda}\otimes \nabla _{-w_0\lambda}$ for one $i$ per dominant $\lambda$. The action of ${\fb}$ on $\sk[ \bg]$ in equation (\ref{eq:cohdescent})
is by way of $\rho _r$. 

\

Let $\Ff=\Ll ({ F})$ as above.
\begin{lemma}\label{lem:cohdescentlemma}
Assume $H^p(\fb,{ F}\otimes\nabla _{\mu})=0$ for $p>0$ and all $\mu \in X(\bt)_+$. 
\begin{enumerate}\item Then ${H}^{p}({ \bg/\fb},\Ff)=0$ for all $p> 0$ and ${H}^{0}({ \bg/\fb},\Ff)$ has a good filtration with associated graded 
$$\bigoplus _{\lambda \in X(\bt)_+}\nabla _{\lambda}\otimes{H}^0({\fb}, \nabla _{-w_0\lambda}\otimes { F}).$$
\item If moreover $H^0(\fb,{ F}\otimes\nabla _{\mu})=0$ for all $\mu \in X(\bt)_+$, then ${H}^{0}({ \bg/\fb},\Ff)=0$.
\end{enumerate}
\end{lemma}

\begin{proof}
\

(1) By (\ref{eq:cohdescent}), ${H}^{p}({ \bg/\fb},\Ff)={H}^p({\fb},\sk[{ \bg}]\otimes { F})$, 
and using the good filtration on $\sk[ \bg]$ we obtain 
${H}^p({\fb},\sk[{ \bg}]\otimes { F})={H}^p({\fb},\varinjlim A_i\otimes { F})={H}^p({\fb},\varinjlim (A_i\otimes { F}))$, 
because direct limits commute with tensoring with finitely generated projective $\sk$-modules. 

By \cite[I. Lemma 4.17]{Jan}, ${H}^p({\fb},\varinjlim (A_i\otimes { F}))=\varinjlim{H}^p({\fb}, A_i\otimes { F})$. 
Now ${H}^p({\fb}, (A_i/A_{i-1})\otimes { F})={H}^p({\fb},\nabla _{\lambda}\otimes \nabla _{-w_0\lambda}\otimes { F})=\nabla _{\lambda}\otimes {H}^p({\fb},\nabla _{-w_0\lambda}\otimes { F})=0$ for $p>0$, as ${\fb}$ acts on $\sk[ \bg]$ 
by way of $\rho _r$, and ${H}^p({\fb},\nabla _{-w_0\lambda}\otimes { F})=0$ for $p>0$ by the assumption; the long cohomology sequence then gives both statements.

(2) This follows from the preceding item.
\end{proof}

\begin{theorem}\label{th:highInd}
Let $\lambda,\mu\in X(\bt)$.  Then
\begin{enumerate}
\item $\RR^i\ind_{\fb}^ \bg (P(\lambda)\otimes Q(\mu))=0$ for $i>0$.
\item $\ind_{\fb}^ \bg \rule{0pt}{1.1em}(P(\lambda)\otimes Q(\mu))$ has a good filtration.
\end{enumerate}
\end{theorem}
\begin{proof}
Corollary \ref{cor:PQnabla} also holds when $\sk=\Z$, by the Universal coefficient Theorem \ref{th:universal coefficients}. The results then  
 follow from (\ref{eq:cohdescent}) and Lemma \ref{lem:cohdescentlemma} above.
\end{proof}

\begin{corollary}\label{cor:P-Q_pairing}
 $\langle[\Ll(P(\lambda))],[\Ll(Q(\mu))]\rangle=[\ind_{\fb}^ \bg (P(\lambda)\otimes Q(\mu))]$.\qed
\end{corollary}

\begin{theorem}\label{th:indSteinberg_vanishing}
Let $v,w\in W$. If $\sk=\cC$, then
\begin{enumerate}
\item
$\ind_{\fb}^ \bg (P(-e_v)\otimes Q(e_w))$ vanishes unless $w\leq v$.
\item
$\ind_{\fb}^ \bg \rule{0pt}{1.1em}(P(-e_v)\otimes Q(e_v))=\sk$.
\end{enumerate}
\end{theorem}
\begin{proof}
\begin{gather*}
    [\ind_{\fb}^ \bg (P(-e_v)\otimes Q(e_w))]=
    \langle[\tP_v],[\tQ_w]\rangle=\\\sum_{y, z\in W}\alpha_{wz}\beta_{vy}\langle[\Oo_{X^+_y}],[\Oo_{X^z}(-\partial X^z)]\rangle
=\sum_{y\in W}\alpha_{wy}\beta_{vy}.
\end{gather*}
Now apply Theorem \ref{th:triang}.
\end{proof}
We wish to get rid of the restriction $\sk=\cC$.
As Euler characteristics are robust, this provides no difficulty:

The $\ind_{\Bm}^{\Gm} (P(\lambda)_\Z\otimes Q(\mu)_\Z)$ are 
free $\Z$-modules of finite rank, so
by Theorem \ref{th:highInd} and the Universal coefficient Theorem \ref{th:universal coefficients} we get that Theorem \ref{th:indSteinberg_vanishing} and part 1 of Theorem \ref{th:highInd} generalize from $\sk=\cC$
to $\Z$, which in turn implies the case where $\sk$ is a field, again by the Universal coefficient Theorem.

Thus

\begin{theorem}\label{th:indPQ}
Let $v,w\in W$.
\begin{enumerate}
\item $\RR^i\ind_\fb^\bg (P(-e_v)\otimes Q(e_v))=\begin{cases}\sk &\text{if $i=0$}\\
0&\text{else.}\end{cases}$
\item If $w\not\leq v$ then $\RR^i\ind_\fb^\bg (P(-e_v)\otimes Q(e_w))=0$ for all $i$.
\end{enumerate}
\end{theorem}

\begin{corollary}\label{cor:left orthogonal}
Let $M, N\in \rep( \bg)$ with $M$, $N$ flat over $\sk$.
\begin{enumerate}
\item $\Ext^i_{\fb}(M\otimes P(-e_v)^*,N\otimes Q(e_v))=\Ext^i_ \bg (M,N) $ for all $i$.
\item If $w\succ v$ then $\Ext^i_{\fb}(M\otimes P(-e_v)^*,N\otimes Q(e_w))=0$ for all $i$.
\end{enumerate}
\end{corollary}
\begin{proof}
$(1).$ We have a spectral sequence  \cite[I Proposition 4.5]{Jan}
$$\Ext^i_ \bg (M\otimes N^*,\RR^j\ind_{\fb}^ \bg (P(-e_v)\otimes Q(e_w)))\Rightarrow \Ext^{i+j}_{\fb}(M\otimes N^*,P(-e_v)\otimes Q(e_w)).$$
By Theorem \ref{th:indPQ}, $\RR^j\ind_{\fb}^ \bg (P(-e_v)\otimes Q(e_w))=0$ for $j>0$. Thus, the above spectral sequence degenerates, and with Theorem \ref{th:indPQ}  this gives
$\Ext^{i}_{\fb}(M\otimes N^*,P(-e_v)\otimes Q(e_v))=\Ext^i_ \bg (M\otimes N^*,\RR^0\ind_{\fb}^ \bg (P(-e_v)\otimes Q(e_v)))=
\Ext^i_ \bg (M\otimes N^*,\sk)=\Ext^i_ \bg (M,N)$.

$(2).$ If $w\succ v$ then $w\not\leq v$, and the above spectral sequence still degenerates and $\Ext^i_{\fb}(M\otimes P(-e_v)^*,N\otimes Q(e_w))=0$ for all $i$.
\end{proof}

\begin{remark}\label{rem:fr left orthogonal}
If $\sk=\Z$, then this proof needs the assumption that $M$, $N$ are flat. But see  Corollary \ref{cor:Z left orthogonal} below.
\end{remark}

\section{\unboldmath Rappels: triangulated categories}\label{sec:rappel_triangulated}

Recall that $\sk$ is a field or $\Z$. Given a $\sk$-linear triangulated category $\D$, equipped with a shift functor $[1]\colon \D \rightarrow \D$ and two
objects $A,B\in \D$, we let $\Hom ^{\bullet}_{\D}(A,B)$ denote 
the graded $\sk$-module $\bigoplus _{i\in \mathbb Z}{\Hom} _{\D}(A,B[i])$. A full triangulated subcategory $\A \subset \D$ is a full subcategory which is closed under shifts and taking cones.

A strictly full subcategory is a full subcategory that contains with every object also the objects isomorphic to it.

Given a smooth scheme $X$ over $\sf k$, we write $\Dd ^b(X)$ for the bounded derived category of coherent sheaves on $X$.

\subsection{\unboldmath Triangulated hull }

\begin{definition}[Triangulated hull]
Given a set $S$ of objects of
a triangulated category $\mathcal D$, its triangulated hull $\hull(S)$ is the smallest strictly full triangulated subcategory that contains $S$.
\end{definition}

\subsection{\unboldmath Thick hull }
\begin{definition}[Thick] A triangulated subcategory of
a triangulated category is called thick if it is closed under taking direct summands of objects.
\end{definition}

\begin{definition}[Thick hull]
Given a  set $S$ of objects of
a triangulated category $\mathcal D$, its thick hull $\hull(S)_{\oplus}$ is the smallest strictly full triangulated subcategory that contains $S$ and is thick.
\end{definition}

\subsection{\unboldmath Admissible subcategories}\label{subsec:amdsubcat}

We follow the exposition of  \cite[Section 2.1]{Ef}.
Let $\N\subset \D$ be a full triangulated subcategory. The right orthogonal to $\N$ is the full subcategory $\N ^{\perp}\subset \D $ consisting of all objects $X$ such that $\Hom _{\D}(Y,X)=0$ for any $Y\in \N$. The left orthogonal $^{\perp}\N$ is defined analogously. The orthogonals are also triangulated subcategories.

\begin{definition}\label{def:admissible}
{A full triangulated subcategory $\mathcal A$ of $\D$ is called
{\it right admissible} if the inclusion functor ${\mathcal A}\hookrightarrow \D$ has a right adjoint. Similarly, $\mathcal A$ is called {\it left
  admissible} if the inclusion functor has a left adjoint. Finally,
$\mathcal A$ is {\it admissible} if it is both right and
left admissible.}
\end{definition}

If the subcategory $\N$ is right (resp., left) admissible, then $\N$ is thick and one can consider the Verdier localization functor with respect to $\N$. 

\begin{proposition}\cite[Proposition 3.2.8]{Kr}\label{prop:thick version}
Let $\N\subset \D$ be a thick subcategory. Then the following are equivalent:

\begin{itemize}

\item[(1)] The inclusion functor $\N\subset \D$ has a right adjoint (\emph{i.e.}, $\N$ is right admissible in $\D$).

\item[(2)] For each object $b\in \D$ there exists an exact triangle $a\rightarrow b\rightarrow c$ with $a\in \N,c\in \N ^{\perp}$.

\item[(3)] The canonical functor $\D \rightarrow \D /\N$ has a right adjoint.

\item[(4)] The composite $\N ^{\perp}\hookrightarrow \D \rightarrow \D/\N$ is a triangulated equivalence.

\end{itemize}

\end{proposition}

We have the following variation

\begin{lemma}\cite[Lemma 3.1]{Bondal}
\label{lem:notorious}
Let $\N\subset \D$ be a full triangulated subcategory, and consider its right orthogonal $\N ^{\perp}\subset \D$. Then the following are equivalent:

\begin{itemize}

\item[(1)]  The category $\D$ is generated by $\N$ and $\N ^{\perp}$, \emph{i.e.}
$\hull(\N ^{\perp}\cup\N)=\D$.

\item[(2)]  For each object $x\in \D$ there exists an exact triangle $b\rightarrow x\rightarrow c$ with $b\in \N,c\in \N ^{\perp}$.

\item[(3)] The embedding functor $i_{\ast}:\N\hookrightarrow \D$ has a right adjoint $i^!:\D\rightarrow \N$.

\item[(4)] $\N$ is thick and the embedding functor $j_{\ast}:\N ^{\perp}\hookrightarrow \D$ has a left adjoint $j^{\ast}:\D\rightarrow \N ^{\perp}$.

\end{itemize}

\end{lemma}
\begin{proof}First assume (2).
Taking for $x$ a summand of an object of $\N$ one gets a triangle $b\to x\to c$ with $\Hom_\D(b,c)=\Hom_\D(x,c)=0$,
hence $\Hom_\D(c,c)=0$, hence $b\cong x$. We conclude that (2) implies that $\N$ is thick.
That (3)$\Rightarrow$(2)$\Leftrightarrow$(1) is shown in
 \cite{Bondal}. So in all four cases $\N$ is thick and we are in the situation of Proposition \ref{prop:thick version}.
\end{proof}

 In \cite[Lemma 3.1]{Bondal} the thickness is missing, but it is necessary because of the following example.

\begin{example}(Henning Krause)
Let $\N$ be {non-thick with thick hull $\D$}. Thus $\N$ is dense \cite[Definition 1.4]{Tho2}, but not thick. Then $\N^\perp=0$ and 
the embedding functor $j_{\ast}:\N ^{\perp}\hookrightarrow \D$ has a left adjoint $j^{\ast}:\D\rightarrow \N ^{\perp}$. But $\hull(\N ^{\perp}\cup\N)\neq \D$.
\end{example}

\begin{lemma}\label{lem:one way}
Let $\Mmm$, $\N$ be full triangulated subcategories of $\D$ with $\Mmm\subset \N ^{\perp}$. If $\hull(\Mmm\cup\N)=\D$,
then $\Mmm= \N ^{\perp}$.
\end{lemma}
\begin{proof}Clearly $\hull(\N ^{\perp}\cup\N)=\D$.
By Lemma \ref{lem:notorious}  we have a left adjoint $j^*$ of the inclusion of $\N^\perp$ into $\D$.
 Then $j^*$ is the identity on $\N^\perp$ and it vanishes on $\N$. 
We thus have that
\begin{equation}\label{eq:image is perp}
j^*(\D)=\N^\perp,
\end{equation}
\begin{equation}
j^*(\N)=0,
\end{equation}
\begin{equation}
j^*(\Mmm)= \Mmm.
\end{equation}
Now consider the full subcategory of $\D$ whose objects are the $X$ with $j^*(X)\in \Mmm$.
It is triangulated and contains both $\N$ and  $\Mmm$. So it contains 
$\hull(\Mmm \cup \N)$ and must be all of $\D$.
In other words, $j^*(\D)= \Mmm$. Taken together with equation (\ref{eq:image is perp}) this shows
\begin{equation}
\N^\perp= \Mmm.
\end{equation}
\end{proof}

\begin{lemma}\label{lem:hullofadmissible}
 If $\N$, $\C$
are admissible subcategories of $\D$ with $\N \subset \C^\perp$, then  $\hull(\N\cup\C)$ is also admissible in $\D$.
\end{lemma}
\begin{proof}
Let us show that the inclusion functor $\hull(\N\cup\C)\subset\D$
has a right adjoint. Consider an object $b\in\D$ as in part (2) of Proposition \ref{prop:thick version}. We have an exact triangle $c\to b\to a$ with 
$c\in \C$, $a\in\C^\perp$.
Next we have an exact triangle 
$n\to a\to m$ with $n\in \N$,
$m\in \N^\perp$. As both $n$ and $a$ are in $\C^\perp$, so is $m$. We see that $m\in\hull(\N\cup\C)^\perp$. 
The composite map $b\to a\to m$ can be completed to an exact triangle
$x\to b\to m$. By the octahedral axiom $x$ fits into an exact triangle $c\to x\to n$. It follows that $x\in \hull(\N\cup\C)$ and $x\to b\to m$ is as in part (2) of Proposition \ref{prop:thick version}.

Dually, the inclusion functor $\hull(\N\cup\C)\subset\D$
has a left adjoint.
\end{proof}
\subsection{\unboldmath Semi-orthogonal decompositions}\label{subsec:sod}

If a full triangulated subcategory $\N\subset \D$ is right admissible then by Lemma \ref{lem:notorious}, 
every object $X\in \D$ fits into a  exact triangle
\begin{equation}
 Y\to X\to Z
\end{equation}
with $Y\in \N$ and $Z\in \N^{\perp}$. 
This exact triangle
is unique 
(up to unique isomorphism) by the proof of \cite[Lemma 3.1]{Bondal}.
One then
says that there is a semi-orthogonal decomposition $\langle \N^{\perp}, \ \N\rangle $ of $\D$ into
the subcategories $\N^{\perp}$, $\N$. More generally,
assume given a sequence of full triangulated subcategories $\N _1,\dots,\N_n \subset \D$. 

\begin{definition}\label{def:semdecomposition-filtration}
{A semi-orthogonal decomposition $\D=\langle \N _1,\dots,\N _n\rangle$ of a triangulated category $\D$ is a sequence of full triangulated subcategories $\N _1,\dots,\N _n$ in $\D$ such that $\N _i\subsetq \N_j^{\perp}$ for $1\leq i < j\leq n$ and $\hull(\N _1,\dots,\N _n)=\D$.
One may show with Lemma \ref{lem:notorious} that
for every object $X\in \D$ there then exists a chain of morphisms in $\D$,
\[
\xymatrix@C=.5em{
0_{\ } \ar@{=}[r] & X_n \ar[rrrr] &&&& X_{n-1} \ar[rrrr] \ar[dll] &&&& X_{n-2}
\ar[rr] \ar[dll] && \ldots \ar[rr] && X_{1}
\ar[rrrr] &&&& X_0 \ar[dll] \ar@{=}[r] &  X_{\ } \\
&&& A_{n} \ar@{->}^{[1]}[ull] &&&& A_{n-1} \ar@{->}^{[1]}[ull] &&&&&&&& A_1\ar@{->}^{[1]}[ull] 
}
\]
such that a cone $A_k$ of the morphism $X_k\rightarrow X_{k-1}$ belongs to $\N _k$ for $k=1,\dots ,n$.}
More generally, if one has a finite totally ordered set $I=\{w_1,\dots,w_n\}$,
one may define a semi-orthogonal decomposition indexed by $I$ by replacing $i$ with $w_i$ in the above.
\end{definition}

\subsection{\unboldmath Exceptional collections}\label{subsec:FEC}

Recall that $\sk$ is a field or $\Z$. Let $\kMod$ denote the category of $\sk$-modules, and $\kmod$ the subcategory of finitely generated $\sk$-modules.
Exceptional collections in $\sf k$-linear triangulated categories are a special case of semi-orthogonal decompositions with each component of the decomposition being equivalent to $\Dd ^b({\kmod})$. 

\begin{definition}\label{def:exceptcollection}Let $\D$ be a $\sf k$-linear triangulated category.

An object $E \in \D$ of  $\D$ is said to be exceptional if there is an isomorphism of graded $\sf k$-algebras 
\begin{equation}
\Hom _{\D}^{\bullet}(E,E) = {\sf k}.
\end{equation}
A collection of exceptional objects $(E_0,\dots,E_n)$ in $\D$ is called 
exceptional if for $1 \leq i < j \leq n$ one has
\begin{equation}
\Hom _{\D}^{\bullet}(E_j,E_i) = 0.
\end{equation}
The collection $(E_0,\dots,E_n)$ in $\D$ is said to be {\it full} if 
$\hull(E_0,\dots,E_n)^{\perp} = 0.$
\end{definition}
\subsection{\unboldmath Admissible subcategories from exceptional collections}\label{subsec:admfromex}

\begin{lemma}\label{lem:hullofexceptional}
Let $\D=\Dd^b (\bg/\bp)$ and $X\in \D$ be an exceptional object.
Then $\hull(X)$ is admissible in $\D$.
\end{lemma}
\begin{proof}Observe that 
$\hull(X)$ is equivalent to $\Dd^b({\kmod})$ through $F\mapsto \RHom_\D(X,F)$ with inverse $E\mapsto E\otimes_\sk^L X$.
   For $E\in \Dd ^b(\kmod)$ we put $E^*=\RHom_{\Dd^b({\kmod}) }(E,\sk)$.
   We claim that the left adjoint of the inclusion functor $\tau:\hull(X)\subset\D$ is $$\sigma:F\mapsto (\RHom_\D(F,X))^*\otimes_\sk^L X.$$ 
We have
\begin{gather*}
\RHom_{\hull(X)}( (\RHom_\D(F,X))^*\otimes_\sk^L X,E\otimes_\sk^L  X )\cong\\
\RHom_{\Dd^b({\kmod}) }( (\RHom_\D(F,X))^*,E)\cong
E\otimes_\sk^L\RHom_{\D}(F,  X)\cong\\
\RHom_{\D}(F,E\otimes_\sk^L  X),
\end{gather*}
for $E\in \Dd ^b(\kmod)$, $F\in \D$. 

   Similarly, the right adjoint of the inclusion functor $\tau$ is $$F\mapsto \RHom_\D(X,F)\otimes_\sk^L X.$$
\end{proof}
Assume given an exceptional collection $(E_0,\dots,E_n)$ in $\D=\Dd^b (\bg/\bp)$. 
Using Lemmas  \ref{lem:hullofadmissible}, \ref{lem:hullofexceptional} one shows that the subcategory $\hull(E_0,\dots,E_n)\subsetq \D$ is admissible in $\D$, which is well known, but the argument is usually given over fields \cite[Theorem 3.2]{Bondal}. The argument over $\Z$ is basically the same.
If the collection is full, then ${\D}
= \hull(E_0,\dots,E_n)$.

\section{\unboldmath Generating the categories \unboldmath$\rep({\fb})$ and $\Dd^b (\rep ({\fb}))$}\label{sec:B-generation}
Recall that $\sk$ is a field or $\Z$.
\subsection{\unboldmath The categories \unboldmath$\Dd^b (\rep ({\fb}))$ and $\Dd^b (\rep ( \bg))$}
\label{subsec:bounded-triang_cat}

We  denote by $\Dd(\Rep ({\fb}))$ (resp., $\Dd(\Rep ( \bg))$) the unbounded derived category of $\Rep(\fb)$ (resp., of $\Rep ( \bg)$), 
and by $\Dd^b (\rep ({\fb}))$ (resp., $\Dd^b (\rep ( \bg))$) the bounded derived category of the smaller category $\rep(\fb)$ (resp., of $\rep ( \bg)$).

Let $\rep_\fr( \fb)$ denote the full subcategory of $\rep( \fb)$ consisting of the representations that are free over $\sk$.
Define $\rep_\fr( \bg)$
similarly.

We have the exact bifunctors:
$$-\;\otimes_\sk-:\rep_\fr( \fb)\times \rep_\fr( \fb)\to \rep_\fr( \fb)$$ 
and (the internal $\Hom _\sk$ on $\rep_\fr( \fb)$):
$$\Hom_\sk(-,-):\rep_\fr( \fb)\times \rep_\fr( \fb)\to \rep_\fr( \fb).$$ 
Mapping $M$ to $M^*$, see Lemma \ref{lem:dualmodule}, provides an anti-autoequivalence of $\rep_\fr( \fb)$.

\begin{proposition}[Resolution property]\label{prop:resolution}
    Let $ H$ be a flat affine group scheme over  $\sk$.
    Then for every finitely generated $ H$-module $N$ there is an exact sequence $$0\to L\to M\to N\to0$$ with the
    $ H$-modules $L$, $M$, finitely generated and free over $\sk$. 
\end{proposition}
\begin{proof}
    This is a special case of \cite[Proposition 3]{Serre}.
\end{proof}

The Proposition implies in standard fashion that every
bounded complex in $\rep( \fb)$ is quasi-isomorphic to a bounded complex with objects in $\rep_\fr( \fb)$.
The above functors can thus be extended as exact bifunctors to the bounded derived category $\Dd ^b(\rep( \fb))$: 
$$
(-)\otimes_\sk (-): \Dd ^b(\rep( \fb))\times \Dd ^{b}(\rep( \fb))\rightarrow \Dd ^b(\rep( \fb))
$$
and 
$$\Hom_\sk(-,-):\Dd ^b(\rep( \fb))\times \Dd ^b(\rep( \fb))\to \Dd ^b(\rep( \fb)).$$ 
In other words, we are using that by the dual of \cite[Theorem 10.22, Remark 10.23]{Buehler},
 the bounded derived category of the abelian category $\rep(\fb)$ is equivalent to the bounded 
derived category of its exact subcategory $\rep_\fr(\fb)$.

We often write $\otimes_\sk $ as $\otimes$. We put $M^*=\Hom_\sk(M,\sk)$ for $M\in \Dd ^b(\rep( \fb))$.
By Lemma \ref{lem:dualmodule}, there is an isomorphism $\Hom _\sk(M, N)=M^{\ast}\otimes N$ for $M,N\in \Dd ^b(\rep( \fb))$. There is $(\Hom _\sk,\otimes _\sk)$-adjunction:
$$
\Hom _\sk (-\otimes _\sk-,-)=\Hom _\sk (-,\Hom _\sk(-,-)).
$$
One gets $$\Hom _\sk (L\otimes_\sk M,-)=\Hom _\sk(L,M^*\otimes_\sk-))$$
for $L$, $M\in \Dd ^b(\rep( \fb))$.



\begin{corollary}\label{cor:Z left orthogonal}
Let $\sk=\Z$.
Let $M, N\in \rep(\Gm)$ and $v,w\in W$.
\begin{enumerate}
\item $\Ext^i_ \Bm (M\otimes P(-e_v)^*,N\otimes Q(e_v))=\Ext^i_ \Gm (M,N) $ for all $i$.
\item If $w\not\leq v$ then $\Ext^i_ \Bm \rule{0pt}{1.1em}(M\otimes P(-e_v)^*,N\otimes Q(e_w))=0$ for all $i$.
\end{enumerate}
\end{corollary}
\begin{proof}We know this already for $M, N\in \rep_\fr(\Gm)$, cf. Remark \ref{rem:fr left orthogonal}.
By Proposition \ref{prop:resolution} the results follow.
\end{proof}

\subsection{\unboldmath The category \unboldmath$\Dd^b (\rep ({\fb}))$ as a $\Dd^b (\rep ( \bg))$-linear category}\label{subsec:G-linear-triang_cat}

The restriction functor $\res_{\fb}^{ \bg}: \Dd^+ (\Rep ( \bg))\rightarrow \Dd^+ (\Rep ({\fb}))$ is $t$-exact. 
Its right adjoint is the induction functor $\Rind_{\fb}^{ \bg}$ and the fact that $\res_{\fb}^{ \bg}$ gives a fully faithful embedding is a consequence of the Generalized Tensor Identity \cite[I Proposition 4.8]{Jan} and the Kempf vanishing theorem, \cite[Theorem 1.2]{CPSWvdK}, \cite[B.3]{Jan}: 
it implies that $\Rind_{\fb}^{ \bg}\res_{\fb}^{ \bg}={\sf id}_{\Dd^+ (\Rep ( \bg))}$. 
More generally, $\Rind_{\bp}^{ \bg}\res_{\bp}^{ \bg}={\sf id}_{\Dd^+ (\Rep ( \bg))}$
by the same argument.

\begin{proposition}
The functor $\Rind_{\bp}^{ \bg}: \Dd ^+(\Rep( \bp))\rightarrow \Dd ^+(\Rep( \bg))$ restricts to a functor $\Rind_{\bp}^{ \bg}: \Dd ^b(\rep( \bp))\rightarrow \Dd ^b(\rep( \bg))$. 
\end{proposition}

\begin{proof}
See \cite[I Proposition 5.12]{Jan}.
\end{proof}

It follows from the above that $\Dd ^b(\rep ( \bg))$ identifies with a right admissible subcategory of $\Dd^b(\rep ({\fb}))$. 
Both categories are monoidal and $\Dd^b(\rep ({\fb}))$ is invariant under the monoidal action by  $\Dd ^b(\rep ( \bg))$. 
Denote $\Dd ^+_{ \bg}( \bg/ \fb):=\Dd ^+({\QCoh}^{ \bg}( \bg/ \fb))
$ and $\Dd ^b_{ \bg}( \bg/ \fb):=\Dd ^b({\sf Coh}^{ \bg}( \bg/ \fb))$,
where ${\QCoh}^{ \bg}( \bg/ \fb)$ (resp. ${\Coh}^{ \bg}( \bg/ \fb)$) denotes the category of $\bg$-equivariant
quasicoherent (resp., $\bg$-equivariant coherent) sheaves on $\bg/ \fb$.
By \cite[Appendix A2]{MR}, ${\QCoh}^{ \bg}( \bg/ \fb)$ has enough injectives.

\begin{proposition}\label{prop:Hom-ind_factorisation}
Let ${ F}\in \Dd ^b(\rep( \fb))$ and let $\Ff \in \Dd ^b_{ \bg}( \bg/ \fb)$ be the associated complex of sheaves on $ \bg/ \fb$. 
Then the functor $\RHom _{{\sf QCoh}(\bg/\fb)}(\Ff,-): \Dd ^+_{ \bg}( \bg/ \fb)\rightarrow \Dd ^+(\kMod)$ factors canonically through a functor $\Dd ^+_{ \bg}( \bg/ \fb)\rightarrow \Dd ^+(\Rep( \bg))$. 
\end{proposition}

\begin{proof}
Put $X= \bg/ \fb, H= \bg$ in \cite[Appendix, Corollary A.5]{MR}.
\end{proof}

\begin{proposition}\label{prop:ForRind}
Let ${F},{G}\in \Dd ^b(\rep( \fb))$ and let $\Ff=\Ll ({F})$, ${\mathcal G}=\Ll ({G})$ be the associated complexes of equivariant coherent sheaves. Then there is a canonical functorial isomorphism
    $\RHom_{\Dd^b(\bg/\fb)}(\Ff,\mathcal G)\cong {\sf For}(\Rind_{\fb}^{ \bg}(F^{\ast}\otimes G))$ where $\sf For$ is the forgetful functor $\Dd ^b(\rep( \bg))\rightarrow \Dd ^b( {\kmod})$.
\end{proposition}

\begin{proof}
Recall from Subsection \ref{subsec:sheafification_functor} that the functor $\Ll:\rep(\fb)\to \Coh(\bg/\fb)$ is exact and monoidal. By equation (\ref{eq:cohdescent}) in Subsection \ref{subsec:cohdesc} we have an isomorphism of functors
\begin{equation}\label{eq:ForInd=Gamma}
{\sf For}({\rm R}{\rm Ind}_{\fb}^{\bg}(-))\simeq \Rr\Gamma(\bg/\fb,\Ll(-))
\end{equation}
on $\Dd ^b(\rep( \fb))$.

Let now $F$, $G\in\rep_\fr(\fb)$. From (\ref{eq:ForInd=Gamma}), we get
\begin{equation}
{\sf For}({\rm R}{\rm Ind}_{\fb}^{\bg}(F^{\ast}\otimes _{\sf k}G))\simeq \Rr\Gamma(\bg/\fb,\Ll(F^{\ast}\otimes _{\sf k}G))\simeq \RHom_{\Dd^b(\bg/\fb)}(\Ll(\Ff),\Ll ({\mathcal G})),
\end{equation}
where the second isomorphism follows from the monoidality of $\Ll$ and from the fact that $\Ll(\Ff)$ is locally free on ${\bg/\fb}$. These isomorphisms are bifunctorial in $F$, $G$ and 
we have $\Dd ^b(\rep( \fb))\cong\Dd ^b(\rep_\fr( \fb))$ by Subsection \ref{subsec:bounded-triang_cat}. The proposition follows.
\end{proof}

Let ${\sf Inv}^{ \bg}$ denote the derived functor of invariants $\Dd ^+(\Rep( \bg))\rightarrow \Dd ^+(\kMod)$.

\begin{proposition}\label{prop:derived_B-cohomology}
Let ${F}\in \Dd ^b(\rep( \fb)), G\in \Dd ^+(\rep( \fb))$ and let $\Ff=\Ll ({F})$, ${\mathcal G}=\Ll ({G})\in \Dd ^+_{ \bg}( \bg/ \fb)$ be the associated equivariant complexes of coherent sheaves. Then there is a canonical functorial isomorphism
$$
{\sf Inv}^{ \bg}\circ \RHom _{{\sf QCoh}(\bg/\fb)}(\Ff,\mathcal G)\xrightarrow{\sim} \RHom _{\Dd ^+(\Rep( \fb))}({F},{G}).
$$
\end{proposition}

\begin{proof}
Put $X= \bg/ \fb$ and $ H= \bg$ in \cite[Appendix, Proposition A.6]{MR}.
\end{proof}

Let $G\in \Dd ^b(\rep( \fb))$ be a bounded complex. Then the above isomorphism restricts to an isomorphism of functors ${\sf Inv}^{ \bg}\circ \RHom _{\Dd ^b(\bg/\fb)}(\Ff,\mathcal G)\xrightarrow{\sim} \RHom _{\Dd ^b(\rep( \fb))}({F},{G}).$

\subsection{\unboldmath \unboldmath$ \bg$-linear semi-orthogonal decompositions}\label{subsec:G-linear-sod}

\begin{definition}\label{def:G-linear-cat}
A triangulated category $\D$ is called $ \bg$-linear if $\D$ is equipped with a monoidal action of $\Dd ^b(\rep( \bg))$, \emph{i.e.}\ there is a bifunctor $\Dd ^b(\rep( \bg))\times \D\rightarrow \D$. Compare \cite[Section 2.3]{KuzHPD}.
\end{definition}

Our main concern is the category $\Dd ^b(\rep( \fb))$ which becomes a $ \bg$-linear triangulated category under the restriction functor $\res_{\fb}^{ \bg}:\Dd ^b(\rep( \bg))\rightarrow \Dd ^b(\rep( \fb))$.
Later we will also need $\bp$, so let us use $\bp$ instead of $\fb$, having in mind $\bp=\fb$ as an important case.

\begin{definition}\label{def:G-linear-SOD}
A $ \bg$-linear semi-orthogonal decomposition of $\Dd^b(\rep ( \bp))$ is a semi-orthogonal decomposition $ \langle{\sf A}_1,\dots, {\sf A}_k\rangle$ of $\Dd^b(\rep ( \bp))$ as in Definition \ref{def:semdecomposition-filtration}, in which
the ${\sf A}_i$ are full triangulated $ \bg$-linear subcategories. 
\end{definition}

The following proposition is just a variation of \cite[Lemma 2.7]{Kuzbasech}:

\begin{proposition}\label{prop:G-linear_sod}
A pair of $\bg$-linear subcategories ${\sf A},{\sf B}\subsetq \Dd ^{b}(\rep( \bp))$ is semi-orthogonal (\emph{i.e.}\  ${\sf A}\subsetq{\sf B}^\perp$) if and only if 
the equality $\Rind_{ \bp}^{ \bg}(N^{\ast}\otimes _k M)=0$ holds for all $M\in {\sf A}$, $N\in {\sf B}$.   
\end{proposition}

\begin{proof}

($\Leftarrow$). Assume given $M\in {\sf A}$, $N\in {\sf B}$ and assume $\Rind_{ \bp}^{ \bg}(N^{\ast}\otimes _k M)=0$. 
Then 
\begin{eqnarray*}&\RHom _{\Dd ^{b}(\rep( \bp))}(N,M)=\RHom _{\Dd ^{b}(\rep( \bp))}(\sk,N^*\otimes_\sk M)=\\
&\RHom _{\Dd ^{b}(\rep(\bg))}(\sk,\Rind_{ \bp}^{ \bg}(N^{\ast}\otimes _k M))=0.
\end{eqnarray*}

($\Rightarrow$). Let $M\in {\sf A}$, $N\in {\sf B}$ and 
$\RHom _{\Dd ^{b}(\rep( \bp))}(N,M)=0$ for all such $M$, $N$.
Let $L$ be an arbitrary object of $\Dd ^{b}(\rep( \bg))$. Then 
\begin{eqnarray*}&
\RHom _{\Dd ^{b}(\rep( \bg))}(L,\Rind_{ \bp}^{ \bg}(N^{\ast}\otimes _{\sk}M))=\RHom _{\Dd ^{b}(\rep( \bp))}(\res _{ \bp}^{ \bg}L,N^{\ast}\otimes _{\sk}M)=\\&
\RHom _{\Dd ^{b}(\rep( \bp))}(N\otimes _{\sk}\res _{ \bp}^{ \bg}L,M)=0,
\end{eqnarray*}
where the last equality holds since the subcategory $\sf B$ is $ \bg$-linear and hence stable under tensoring with objects of $\Dd ^{b}(\rep( \bg))$. 
It follows that $\RHom _{\Dd ^{b}(\rep( \bg))}(L,\Rind_{ \bp}^{ \bg}(N^{\ast}\otimes _{\sk}M))=0$ for an arbitrary object $L\in \Dd ^{b}(\rep( \bg))$; hence, $\Rind_{ \bp}^{ \bg}(N^{\ast}\otimes _{\sk}M)=0$.
\end{proof}

\begin{definition}\label{def:G-linear_fun}
A functor $\Phi: \Dd^b (\rep( \bg))\rightarrow \Dd^b (\rep( \bp))$ is called $ \bg$-linear  if for all $M\in \Dd ^b(\rep( \bg))$, there is an isomorphism of functors $\Phi (M)=\res_{ \bp}^{ \bg}(M)\otimes \Phi (\sk)$.
\end{definition}

Let us simplify notation and write $\RHom _{ \Dd ^b(\rep( \bp))}$ as $\RHom _{ \bp}$,  and $\RHom _{ \Dd ^b(\rep( \bg))}$ as $\RHom _{ \bg}$. We will also suppress $\res_{ \bp}^{ \bg}(-)$ at a $\bg$-linear functor $\Phi$ as in Definition \ref{def:G-linear_fun} writing simply instead $\Phi (M)=M\otimes \Phi (\sk)$ for $M\in \Dd ^b(\rep( \bg))$; as the functor $\res_{ \bp}^{ \bg}$ if a fully faithful embedding of $\Dd ^b(\rep( \bg))$ into $\Dd ^b(\rep( \bp))$ by Section \ref{subsec:G-linear-triang_cat}, this shall lead to no confusion.

\begin{proposition}\label{prop:G-linear_exceptional}
\
\begin{enumerate}
\item Assume given a fully faithful $ \bg$-linear functor $\Phi: \Dd ^b(\rep( \bg))\rightarrow \Dd ^b(\rep( \bp))$. Let $E$ denote $\Phi(\sk)$. Then $\RHom _{ \bp}(E,E)=\sk.$

\item Each object $E\in \Dd ^b(\rep( \bp))$ such that $\Rind_{ \bp}^{ \bg}(E^{\ast}\otimes E)=\sk$ gives a fully faithful $ \bg$-linear functor $\Phi_E: \Dd ^b(\rep( \bg))\rightarrow \Dd ^b(\rep( \bp))$ where $\Phi_E(-)=(-)\otimes E$.

\item Under the assumptions of (1), the object $E$ satisfies $\Rind_{ \bp}^{ \bg}(E^{\ast}\otimes E)=\sk$.
\end{enumerate}

\end{proposition}
\begin{proof}
(1) We have $\RHom _{ \bp}(E,E)=\RHom _{\bg}(\sk,\sk)=\sk$ where the last isomorphism follows from Proposition \ref{prop:cohinduced}.

(2) For $M,N\in \Dd ^b(\rep( \bg))$ we have
\begin{eqnarray*}& \RHom _{ \bp}(\Phi_E(M),\Phi_E(N))=\RHom _{ \bp}(M\otimes E,N\otimes E)=\RHom_\bp(M\otimes N^\ast,E^\ast\otimes E)=\\
&\RHom_\bg(M\otimes N^\ast,\Rind_\bp^\bg( E^\ast\otimes E))=\RHom_\bg(M\otimes N^\ast,\sk)=
\RHom_\bg(M,N).
\end{eqnarray*}

(3) Under the assumptions of (1), for $M\in \Dd ^b(\rep( \bg))$ we have 
\begin{eqnarray*}&\RHom _{ \bg}( M, \sk)=\RHom _{ \bp}(\Phi(M), \Phi(\sk))=\RHom_\bp(M\otimes E, E)=\\&\RHom_\bp(M, E^\ast\otimes E)=\RHom_\bg(M,\Rind_\bp^\bg( E^\ast\otimes E))
\end{eqnarray*}
By the Yoneda Lemma it follows that $\Rind_{ \bp}^{ \bg}(E^{\ast}\otimes E)=\sk$.
\end{proof}

\begin{remark}\label{rem:ind-to-B-exceptional}
{If $\Rind_{ \bp}^{ \bg}(E^{\ast}\otimes E)=\sk$, then $E$ is exceptional in 
$\Dd ^b(\rep( \bp))$,
 since $\RHom _\bp(E,E)=\RHom _\bp(\sk,E^{\ast}\otimes E)=
\RHom _\bg(\sk, \Rind_{ \bp}^{ \bg}(E^{\ast}\otimes E))$ by \cite[I Proposition 4.5]{Jan}. Using again Proposition \ref{prop:cohinduced}, we obtain $\RHom _\bg(\sk, \Rind_{ \bp}^{\bg}(E^{\ast}\otimes E))=\RHom _\bg(\sk,\sk)=\sk$.}
\end{remark}

\subsection{\unboldmath Generating \unboldmath$\Dd ({\QCoh}(\bg/\bp))$}\label{subsec:generatingG/P}

\begin{definition}
Let $\D$ be a compactly generated triangulated category. A set $S$ 
of compact objects of $\D$ is called a generating set if $\Hom _{\D}(S,X)=0$ implies $X=0$ and $S$ is closed under the shift functor, \emph{i.e.} $S =S[1]$.
\end{definition}

\begin{proposition}\label{prop:ample_line_bundle_gen_set}
Let $X$ be a quasi--compact, separated scheme, and $\Ll$ be an ample line bundle on $X$. Then the set $\langle \Ll ^{\otimes m}[n]\rangle, m,n\in \mathbb Z$ is a generating set for $\Dd ({\QCoh}(X))$.
\end{proposition}

\begin{proof}
See \cite[Example 1.10]{Neem}.
\end{proof}

We now choose an ample line bundle $ \Ll$ on $ \bg/ \bp$.
Such a line bundle exists by \cite[II 8.5]{Jan}.

\begin{corollary}\label{cor:line_bundles-generation}
The set of line bundles $\langle \Ll^{\otimes m} [n]\rangle, m,n\in \mathbb Z$ is a generating set for $\Dd({\QCoh}( \bg/ \bp))$.
\end{corollary}

\begin{corollary}\label{cor:generatingGmodB}
The smallest thick full triangulated subcategory of $\Dd^b( \bg/ \bp)$ containing $\langle \Ll^{\otimes m} [n]\rangle, m,n\in \mathbb Z$ is $\Dd^b( \bg/ \bp)$.
\end{corollary}

\begin{proof}Any object of $\Dd^b( \bg/ \bp)$ is a perfect complex and therefore compact in $\Dd(\QCoh( \bg/ \bp))$ by  \cite[Example 1.13]{Neem}. The result thus
follows from \cite[part 2.1.3 of Theorem 2.1]{Neem}.
\end{proof}

The goal of the rest of this section -- and the two sections after it -- is to construct a collection of objects $X_p\in \Dd ^b(\rep( \fb))$, $p\in W$, each satisfying the condition in (2) of Proposition \ref{prop:G-linear_exceptional}. By Remark \ref{rem:ind-to-B-exceptional}
each of those objects will also be exceptional. Furthermore, the collection of objects $X_p$ will produce a $ \bg$-linear semi-orthogonal decomposition of $\Dd^b(\rep ({\fb}))$; the ultimate statement is Theorem \ref{th:semi-orthogonal}.

\subsection{\unboldmath Generating \unboldmath$\Dd^b (\rep ( \bg))$}\label{subsec:generatingrepG}

\begin{proposition}\label{prop:stronggeneratingrepG}
The smallest strictly full triangulated subcategory of $\Dd^b (\rep ( \bg))$ that contains the set of modules $\nabla _{\lambda}$,
$\lambda \in X(\bt)_+$, is the whole $\Dd^b (\rep ( \bg))$.
\end{proposition}

\begin{proof}First we show that $\rep(\bg)$ lies in the subcategory.
We choose a real valued additive injective height function $\hgt$ on the weight lattice which is positive on positive roots.
Consider a nonzero $M\in\rep( \bg)$. Say $\mu$ is its highest weight with respect to $\hgt$. 
Thus $\hgt(\mu)\geq0$. Assume that all representations with a smaller highest weight are in the subcategory. We have a map $\Delta_\mu\to \nabla_\mu$ whose kernel and cokernel have
lower weights, so  $\Delta_\mu$ lies in the subcategory.
Let $M_{\mu,\mathrm {triv}}$ denote the weight space $M_{\mu}$ provided with a trivial $ \bg$-action.
By \cite[Proposition 21]{FvdK} there is a natural map 
$\Delta_\mu\otimes_\sk M_{\mu,\mathrm {triv}}\to M$ whose kernel and cokernel have lower weights.  
Now notice that $\Delta_\mu\otimes_\sk M_{\mu,\mathrm {triv}}$ lies in the subcategory. So $M$ does too.
So $\rep(\bg)$ lies in the subcategory. 
Now use that every object of $\Dd^b(\rep(\bg))$  is quasi-isomorphic to a bounded complex in $\rep( \bg)$ and that bounded complexes are repeated cones of objects of minimal cohomological amplitude. (cf.\ ``Stupid  truncations'' \cite[2.5]{Kuzbasech}.) 
\end{proof}

\begin{proposition}\label{prop:G linear hull}
Let $S$ be a set of objects of $\Dd^b(\rep(\fb))$. 

Then $\hull(\{\nabla_\nu\otimes M\mid \nu\in X_+,M\in S\})$ is $\bg$-linear.
\end{proposition}
\begin{proof}By Subsection \ref{subsec:bounded-triang_cat} we may assume that the members of $S$ are represented by complexes that are flat over $\sk$.
By Proposition \ref{prop:stronggeneratingrepG} we have $\hull(\{\nabla_\nu\mid \nu\in X(\bt)_+\})= \Dd^b (\rep ( \bg))$. 
So $\hull(\{\nabla_\nu\otimes M\mid \nu\in X(\bt)_+,M\in S\})$ equals $\hull(\{\Dd^b (\rep ( \bg))\otimes M\mid M\in S\})$.
\end{proof}
\subsection{\unboldmath Generating \unboldmath$\Dd^b (\rep ({\fb}))$}\label{subsec:D^brepB-generation}

\begin{theorem}\label{th:derived_generation}
Given a $p\in W$, the triangulated hull in $\Dd(\Rep(\fb))$ of the two categories
\begin{eqnarray*}
&\hull(\{\nabla_\lambda\otimes Q(e_v)\}_{v\succ p,\lambda\in X(\bt)_+}),\\
&\hull(\{\nabla_\lambda\otimes P(-e_v)^*\}_{v\preceq p,\lambda\in X(\bt)_+}) 
\end{eqnarray*}
is $\Dd^b (\rep ({\fb}))$.
\end{theorem}

The proof of this theorem will take the rest of the section. We will categorify the theorem of Steinberg that says that
the $[\sk_{e_v}]$ generate $R(\bt)$ as an $R(\bg)$-module and we will apply the same reasoning with a few $\sk_{e_v}$'s replaced with $Q(e_v)$ or $P(-e_v)^*$. Our arguments are similar to the proof of \cite[Theorem 2]{Ana}.

\subsection{\unboldmath Generating \unboldmath$\rep({\fb})$}\label{subsec:repB-generation}

  We say that a full subcategory of an abelian category has the 2 out of 3 property if,
  whenever $0\to N_1\to N_2 \to N_3 \to 0$ is exact and two of the $N_i$ are in the subcategory, then so is the third.
  
  \subsection{\unboldmath The set \unboldmath$\{M_v\}_{v\in W}$}\label{subsec:Mv's}
  
   Let us be given a set $\{M_v\}_{v\in W}$ of objects of $\rep_\fr({\fb})$  with the following properties.
  The multiplicity of the weight $e_v$ in $M_v$ is one. Every weight $\lambda$ of $M_v$ satisfies $(\lambda,\lambda)\leq (e_v,e_v)$.
  If $\lambda $ is a weight of $M_v$ with $(\lambda,\lambda)= (e_v,e_v)$, then  $\lambda$ is a weight of $P(-e_v)^*$. So all weights $\lambda$
  of $M_v$ precede $e_v$ in the antipodal excellent order, notation $\lambda\leq_ae_v$.
Examples of possible choices of $M_v$ are $\sk_{e_v}$, $Q(e_v)$, $P(-e_v)^*$.
  
  \begin{theorem}[Generation]\label{th:generation}
  The smallest strictly full additive subcategory  that 
  \begin{itemize}
  \item contains the $M_v$,  
  \item has the
  2 out of 3 property and 
  \item contains with every $\sk_\lambda$  also
   $\sk_\lambda\otimes\nabla_{\omega_\alpha}$ for every fundamental representation $\nabla_{\omega_\alpha}$,
  \end{itemize}
   is the category $\rep({\fb})$ of finitely generated ${\fb}$-modules.
  \end{theorem}
  
We first prove a lemma and two propositions.
   \begin{lemma}\label{lem:aed}
    Let $\lambda\in X(\bt)$, $\alpha\in \Pi$.
    Then $s_\alpha\lambda\leq_e\lambda$ if and only if $s_\alpha\lambda\geq_d\lambda$.    
\end{lemma}

\begin{proof}
    Write $\lambda=w\nu$ with $\nu$
    dominant and $w$ minimal. 
    Note that $w$ is unique by Lemma \ref{lem:minimal-coset-representative}.
    Suppose $s_\alpha\lambda<_e\lambda$.  
    Then $s_\alpha w$ is the minimal element $z$ of $W$ with $z\nu= s_\alpha\lambda$
    and we have $s_\alpha w<w$, so $w^{-1}s_\alpha<w^{-1}$. 
    By Lemma \ref{lem:ws shorter} we have $w^{-1}\alpha<0$, so $(w^{-1}\alpha,\nu)\leq0$, hence  $(\alpha,\lambda)=(\alpha,w\nu)\leq0$.
    As
   $s_\alpha\lambda\neq\lambda$, it follows that  $s_\alpha\lambda>_d\lambda$.
 
  Next suppose $s_\alpha\lambda>_e\lambda$. Then 
 $s_\alpha s_\alpha\lambda>_ds_\alpha\lambda$ by the above, so that 
 $s_\alpha\lambda<_d\lambda$.
\end{proof}

   \begin{proposition}\label{prop:touch}
  Let $\lambda$ be a weight in the $W$-orbit of the dominant weight $\varpi$. 
  Take $w$ minimal so that $\lambda=w\varpi$. Let there be a simple root $\alpha$  such that $(\alpha^\vee,\varpi)=1$ and 
  $w\leq w s_{\alpha}$. Let $\mu\in w(\varpi-\omega_\alpha)+W\omega_\alpha$. Then $(\mu,\mu)\leq (\varpi,\varpi)$ and if $(\mu,\mu)=(\varpi,\varpi)$,
  then there is $v\geq w$ with $\mu=v\varpi$.
  So $\mu\leq_a\lambda$ by Lemma \ref{lem:ea}.
  \end{proposition}
  
  \begin{proof}
  Put $\tau=\varpi-\omega_\alpha$. Note that $\tau$ is dominant. Let $W_I$, $W_J$, $W_K$ be the stabilizers in $W$ of $\tau$, $\omega_\alpha$, $\varpi$ respectively,
  with notations as in definition \ref{def:WI}. So $K=I\sqcup \{\alpha\}$ and $J=S-\{\alpha\}$.
  Choose $z$ minimal so that $w^{-1}\mu=\tau+z\omega_\alpha$. Assume $(\mu,\mu)\geq(\varpi,\varpi)$.
  Choose a reduced expression $s_ks_{k-1}\cdots s_1$ for $z$ and put $z_0=\id$, $z_i=s_iz_{i-1}$. 
  As $z_{i+1}\omega_\alpha>_ez_i\omega_\alpha$, we have by Lemma \ref{lem:aed} that
the path $z_0\omega_\alpha,\cdots, z_k\omega_\alpha$ from $\omega_\alpha$ to $z\omega_\alpha$ is strictly descending for $\leq_d$.
  Along the path $(\tau,z_i\omega_\alpha)$ 
  can only go down, but $(\tau,u\omega_\alpha)<(\tau,v\omega_\alpha)$ implies 
  $(\tau+u\omega_\alpha,\tau+u\omega_\alpha)<(\tau+v\omega_\alpha,\tau+v\omega_\alpha)$.
  So the path must consist of steps in directions perpendicular
  to $\tau$. So $z\in W_I$.
  Now $w$ is minimal in its coset $wW_K$, and moreover we have $w\leq w s_{\alpha}$. That makes it minimal in $wW_I$ also.
  This shows $wz\geq w$. Now note that $\mu=w(\tau+z\omega_\alpha)=w(z\tau+z\omega_\alpha)=wz\varpi$.
  \end{proof}
  
  \begin{proposition}\label{prop:fartouch}
  Let $\lambda$ be a weight in the $W$-orbit of the dominant weight $\varpi$. 
  Assume $(\alpha^\vee,\varpi)>1$ for some
  simple root $\alpha$. Take $w$ minimal so that $\lambda=w\varpi$.  Let $\mu\in w(\varpi-\omega_\alpha)+W\omega_\alpha$. 
  Then $(\mu,\mu)\leq (\varpi,\varpi)$ and if $(\mu,\mu)=(\varpi,\varpi)$,
  then there is $v\geq w$ with $\mu=v\varpi$.  So $\mu\leq_a\lambda$ by Lemma \ref{lem:ea}.
  \end{proposition}
  \begin{proof}
  Put $\tau=\varpi-\omega_\alpha$. Let $W_I$, $W_J$  be the stabilizers in $W$ of $\tau$, $\omega_\alpha$ respectively.
  Observe that $W_I$ is also the stabilizer $W_K$ of $\varpi$.
  Choose $z$ minimal so that $w^{-1}\mu=\tau+z\omega_\alpha$.
Assume $(\mu,\mu)\geq(\varpi,\varpi)$. 
  Along the path $z_0\omega_\alpha,\cdots, z_k\omega_\alpha$ from $\omega_\alpha$ to $z\omega_\alpha$, given by an irreducible expression for $z$, the inner product $(\tau,z_i\omega_\alpha)$ 
  can only go down, so the path must consist of steps in directions perpendicular
  to $\tau$. So $z\in W_I=W_K$.
  Now $w$ is minimal in its coset $wW_K=wW_I$.  This shows $wz\geq w$. Now note that $\mu=w(\tau+z\omega_\alpha)=w(z\tau+z\omega_\alpha)=wz\varpi$.
  \end{proof}

  \begin{proof}[\rm \bf  Proof of Theorem \ref{th:generation}]
  Let $\varpi$ be dominant and let $\lambda$ be a weight in its $W$-orbit. Choose $w$ minimal so that 
  $\lambda=w\varpi$.
  Assume that for all $v\geq w$ with $v\varpi\neq \lambda$ the representation $\sk_{v\varpi}$ 
   is in the subcategory.
Assume also that for all weights $\mu$ with $(\mu,\mu)<(\lambda,\lambda)$ 
  the representation $\sk_\mu$ of weight $\mu$ is in the subcategory. 
  In other words, assume that $\sk_\mu$  is in the subcategory for all $\mu$ with $\mu<_a\lambda$.
  
  We claim that then    $\sk_\lambda$ is in the subcategory.
  The theorem easily follows from the claim by induction along $<_a$.
  
  So let us prove the claim.
  There are several cases.
  If $\lambda$ is a Steinberg weight $e_v$, then one uses the given properties of $M_v$.
  So we may assume it is not a Steinberg weight.
   
  There are two cases.
  The first case is that $(\beta^\vee,\varpi)\leq 1$ for all $\beta\in \Pi$.
  As $\lambda$ is not a Steinberg weight there must be an $\alpha\in\Pi$ with $(\alpha^\vee,\varpi)=1$ but $w\alpha$ positive.
 Then $\ell(ws_\alpha)>\ell(w)$ by Lemma \ref{lem:ws shorter} and we are in the situation of Proposition \ref{prop:touch}. 
 Let $\tau=\varpi-\omega_\alpha$. Then $\tau$ is strictly shorter than $\lambda$, so $\sk_{w\tau}$ is in our subcategory, and therefore
  $N=\sk_{w\tau}\otimes \nabla_{\omega_\alpha}$ is in the subcategory. The weights $\nu$ of $N$  lie in the convex hull of $w\tau+W\omega_\alpha$. 
  By Proposition \ref{prop:touch} the weights $\mu$
  in $w\tau+W\omega_\alpha$ are either strictly shorter than $\lambda$
  or they are of the form $v\varpi$ with $v\geq w$, and $\mu\leq_a\lambda$ in either case.
  Put $B_{<(\lambda,\lambda)}=\{\nu\in {\mathbb R}\otimes_{\mathbb Z} X(\bt)\mid (\nu,\nu)<(\lambda,\lambda)\}$,
  the open ball of elements strictly shorter than $\lambda$. As $w\tau+W\omega_\alpha$ is contained in the closure of
  $B_{<(\lambda,\lambda)}$, the union of $w\tau+W\omega_\alpha$ and $B_{<(\lambda,\lambda)}$ is convex
  and thus contains all weights of $N$. A weight $\nu\in B_{<(\lambda,\lambda)}$ satisfies  $\nu\leq_a\lambda$
  and so does a weight $\nu$ in $w\tau+W\omega_\alpha$.
  So all the weights $\nu$ of $N$ 
  satisfy $\nu\leq_a\lambda$.
  So $\sk_\nu$ is in the subcategory for all weights $\nu$ of $N$ different from $\lambda$ and $\lambda$ has multiplicity one in $N$. It follows that $\sk_\lambda$ is in the subcategory.
  
 The second case is that $(\alpha^\vee,\varpi)> 1$ for some $\alpha \in \Pi$. Put $\tau=\varpi-\omega_\alpha$. Again $\tau$ is strictly shorter than $\varpi$.
  By Proposition \ref{prop:fartouch} and the argument above with a convex hull,
  the weights  $\nu$ of $N=\sk_{w\tau}\otimes \nabla_{\omega_\alpha}$ are either stricly shorter than $\lambda$
  or they are of the form $v\varpi$ with $v\geq w$.
 So $\sk_\nu$ is in the subcategory for all weights of $N$ different from $\lambda$ and $\lambda$ has multiplicity one in $N$. It follows that $\sk_\lambda$ is in the subcategory.
  \end{proof}

  \begin{remark}\label{cor:if not Steinberg then tensor}
  The above proof  shows the following.
If $\sigma$ is not a Steinberg weight, then  
 there is a fundamental weight $\omega_\alpha$ and a weight $\mu$ with $(\mu,\mu)<(\sigma,\sigma)$,
 so that the $\fb$-module $L=\nabla_{\omega_\alpha}\otimes \sk_\mu$  has
$\sigma$ as a weight  of  multiplicity one, and so that all weights  $\nu$ of $L$ satisfy 
$\nu\leq_a\sigma$.
\end{remark}

\begin{remark}\label{cor:Steinberg_alternative}Let $\sk$ be an algebraically closed field. Let $W$ be partitioned
arbitrarily into three subsets $W_1$, $W_2$, $W_3$. One gets a basis of $R({\bt})$ over $R( \bg )$ by taking
   $[\sk_{e_v}]$ for ${v\in W_1}$, $[Q(e_v)]$ for ${v\in W_2}$,  $[P(-e_v)^*]$ for ${v\in W_3}$. Indeed $\sk_{e_v}$, $Q(e_v)$,  $P(-e_v)^*$ are permissible choices for $M_v$.
 \end{remark}

 \begin{remark}
{
The results of this section give an alternative proof of Steinberg's theorem \cite{St} that the $[\sk_{e_w}]$ generate
$R({\bt})$ over $R( \bg )$. See also Lemma \ref{cor:ordered_generate} or \cite[Theorem 2]{Ana}. Our proof  assumes only that $\bg$ is split and $\sk$ is a field or $\Z$, while
Steinberg takes the field algebraically closed. But by \cite[\S 3]{Serre} this makes no difference.
 }
 \end{remark}
    
\begin{proof}[\rm \bf  Proof of Theorem \ref{th:derived_generation}]
By Theorem \ref{th:generation}, for any $p\in W$ the set
 $\{\nabla_\lambda\otimes Q(e_v)\}_{v\succ p,\lambda\in X(\bt)_+}\cup\{\nabla_\lambda\otimes P(-e_v)^*\}_{v\preceq p,\lambda\in X(\bt)_+}$ generates $\rep({\fb})$ as an abelian category in a specific manner, corresponding in the derived
 category with taking cones, shifts, or using the $\bg$-linear structure of Proposition \ref{prop:G linear hull}. Therefore the hull contains contains the heart of the standard $t$-structure on $\Dd^b (\rep ({\fb}))$.
 Considering the canonical truncation of an object of $\Dd^b (\rep ({\fb}))$ and resolving each cohomology via shifts of objects of the heart $\rep ({\fb})$,
 we get the statement. One could also use stupid truncations, cf.\ \cite[2.5]{Kuzbasech}.
\end{proof}

\section{\unboldmath Construction of the objects \unboldmath$X_p$ and $Y_p$}\label{sec:X_p&Y_p}
Recall that $\sk$ is a field or $\Z$.
By Theorem \ref{th:derived_generation}, we have
\begin{equation}\label{eq:trianghullP-Q}
\hull( \{\nabla_\lambda\otimes Q(e_v)\}_{v\succ  p,\lambda\in X(\bt)_+}\cup\{\nabla_\lambda\otimes P(-e_v)^*\}_{v\preceq  p,\lambda\in X(\bt)_+})
=\Dd ^b(\rep( \fb))
\end{equation}
for all $p\in W$.

\subsection{\unboldmath Cut at \unboldmath$p\in W$}\label{subsec:cut_at_p}

Introduce the following notation: 
\begin{itemize}
\item ${\sf Q}_{\succeq p}:=\hull(\{\nabla_\lambda\otimes Q(e_v)\}_{v\succeq  p,\lambda\in X(\bt)_+})$

\item ${\sf Q}_{\succ p}:=\hull(\{\nabla_\lambda\otimes Q(e_v)\}_{v\succ  p,\lambda\in X(\bt)_+})$

\item ${\sf P}_{\preceq p}:=\hull(\{\nabla_\lambda\otimes P(-e_v)^*\}_{v\preceq  p,\lambda\in X(\bt)_+})$

\item ${\sf P}_{\prec p}:=\hull(\{\nabla_\lambda\otimes P(-e_v)^*\}_{v\prec  p,\lambda\in X(\bt)_+})$

\end{itemize}
We will show eventually (Lemma \ref{lem:samehull}, Lemma \ref{lem:hullofadmissible}, Proposition \ref{prop:sfX_p-admissibility}) that these are admissible subcategories of $\Dd ^b(\rep( \fb))$ .
\begin{remark}
    Do not confuse ${\sf Q}_{\succ p}$
    with $\hull(\{\nabla_\lambda\otimes Q(e_v)\}_{e_v>_a  e_p,\lambda\in X(\bt)_+})$.
    The latter may seem more natural, as the $Q(\mu)$ are costandard modules for $>_a$ (Theorem \ref{th:Qant}). 
    The Steinberg weights
    show two faces, and this is essential. On the one hand they are indexed by the Weyl group and are thus
    ordered by the Bruhat order, on the other hand they can be characterised  in terms of $>_a$, see Remark \ref{rem:Steinberg from anti}.
\end{remark}

We put $\D=\Dd ^b(\rep( \fb))$. Then (\ref{eq:trianghullP-Q}) gives that \begin{equation}\label{eq:Q>pPleqp-generation}
\D=\hull({\sf Q}_{\succ p} \cup {\sf P}_{\preceq p}).
\end{equation}
By Corollary \ref{cor:left orthogonal} and Corollary \ref{cor:Z left orthogonal}, for all $p\in W$:
\begin{equation}\label{eq:QinperP}
{\sf Q}_{\succ p}\subsetq{\sf P}_{\preceq p}^{\perp}.
\end{equation}
Note that we are taking the right orthogonal in $\D$.

By Lemma \ref{lem:one way}  we have 
\begin{equation}\label{eq:P perp is Q}
{\sf P}_{\preceq p}^\perp= {\sf Q}_{\succ p}.
\end{equation}
Thus, by Definition \ref{def:semdecomposition-filtration}, $\D$ has a semi-orthogonal decomposition:
\begin{equation}\label{eq:p-q-SOD-1}
\D=\langle \;{\sf Q}_{\succ p},{\sf P}_{\preceq p}\; \rangle .
\end{equation}
Then the inclusion ${\sf Q}_{\succ p}\hookrightarrow \D$ has a left adjoint and the inclusion ${\sf P}_{\preceq p}\hookrightarrow \D$ has a right adjoint. So, ${\sf Q}_{\succ p}$ is left admissible in $\D$ and ${\sf P}_{\preceq p}$ is right admissible in $\D$. 
Dually to (\ref{eq:P perp is Q}) we also have 
\begin{equation}
{\sf P}_{\preceq p}={}^\perp {\sf Q}_{\succ p}.
\end{equation}
Moreover: 
\begin{proposition}\label{prop:p-q-SOD-2}
\begin{equation}\label{eq:p-q-SOD-2}
\D=\langle\; {\sf Q}_{\succeq p},{\sf P}_{\prec p}\;\rangle 
\end{equation}
\end{proposition}

\begin{proof}
    If $p\neq\id$, then $p$ has a predecessor $p'$ for the total order $\prec$ on $W$ and ${\sf Q}_{\succeq p}={\sf Q}_{\succ p'}$,
    ${\sf P}_{\prec p}={\sf P}_{\preceq p'}$.
    And if $p=\id\in W$, then actually $Q(e_\id)=P(-e_\id)^*=\sk$, so that ${\sf Q}_{\succeq \id}=\hull({\sf Q}_{\succ \id}\cup {\sf P}_{\preceq \id})=\D$.
\end{proof}
\subsection{\unboldmath Defining \unboldmath$X_p$ and $Y_p$}\label{subsec:X_p-Y_p}

Let $X_p$ be the image of $P(-e_p)^*$ under the left adjoint of the inclusion of ${\sf Q}_{\succeq p}$ into $\D$.
More precisely, denote by $i^{\sf Q}_{\succeq p}:{\sf Q}_{\succeq p}\hookrightarrow \D$ the embedding functor and let ${i^{\sf Q{\ \ast}}_{\succeq p}}$ be its left adjoint. Then
\begin{equation}\label{eq:def_of_X_p}
X_p:={i^{\sf Q{\ \ast}}_{\succeq p}}(P(-e_p)^*).
\end{equation}
By definition, $X_p\in {\sf Q}_{\succeq p}$. 
We have an exact triangle 
\begin{equation}\label{eq:def_triang_X_p}
P(-e_p)^*\rightarrow X_p\rightarrow {\sf cone}(P(-e_p)^*\rightarrow X_p).
\end{equation}
Note that ${\sf cone}(P(-e_p)^*\rightarrow X_p)\in{\sf P}_{\prec p}=\hull(\{\nabla_\lambda\otimes P(-e_v)^*\}_{v\prec  p,\lambda\in X(\bt)_+})$. Therefore, 
$$X_p\in\hull(  P(-e_p)^*\cup {\sf P}_{\prec p})\subsetq\hull(\{\nabla_\lambda\otimes P(-e_v)^*\}_{v\preceq  p,\lambda\in X(\bt)_+})={\sf P}_{\preceq p}.$$
So, 
\begin{equation}\label{eq:X in Q and P}
X_p\in {\sf Q}_{\succeq p}\cap {\sf P}_{\preceq p}.    
\end{equation}
Now let $Y_p$ be the image of $Q(e_p)$ under the 
right adjoint of the inclusion of
${\sf P}_{\preceq p}=\hull(\{\nabla_\lambda\otimes P(-e_v)^*\}_{v\preceq p,\lambda\in X(\bt)_+})$ into $\D$. More precisely, denote by $i^{\sf P}_{\preceq  p}:{\sf P}_{\preceq p}\hookrightarrow \D$ the embedding functor and let $i^{\sf P{\ !}}_{\preceq p}$ be its right adjoint. Then
\begin{equation}\label{eq:def_of_Y_p}
Y_p:=i^{\sf P{\ !}}_{\preceq p}(Q(e_p)).
\end{equation}
By definition, $Y_p\in {\sf P}_{\preceq p}$. We have an exact triangle 
\begin{equation}\label{eq:def_triang_Y_p}
 Y_p\rightarrow Q(e_p)\rightarrow {\sf cone}(Y_p\rightarrow Q(e_p)).
\end{equation}
Note that ${\sf cone}(Y_p\rightarrow Q(e_p))\in {\sf P}_{\preceq p}^{\perp}=
{\sf Q}_{\succ p}$. Therefore, 
$$Y_p\in \hull( Q(e_p)\cup {\sf Q}_{\succ p})\subsetq\hull(\{\nabla_\lambda\otimes Q(e_v)\}_{v\succeq  p,\lambda\in X(\bt)_+})={\sf Q}_{\succeq p}.$$
So, \begin{equation}\label{eq:Y in Q and P}Y_p\in {\sf Q}_{\succeq p}\cap {\sf P}_{\preceq p} .
\end{equation}
\begin{lemma}\label{lem:samehull}Let $p\in W$.
\begin{itemize}
    \item $\hull(\{\nabla_\lambda\otimes X_v\}_{v\preceq p,\lambda\in X(\bt)_+})=\hull(\{\nabla_\lambda\otimes P(-e_v)^*\}_{v\preceq  p,\lambda\in X(\bt)_+}).$

    \item $\hull(\{\nabla_\lambda\otimes Y_v\}_{v\succeq  p,\lambda\in X(\bt)_+})=\rule{0pt}{1.1em}
    \hull(\{\nabla_\lambda\otimes Q(e_v)\}_{v\succeq  p,\lambda\in X(\bt)_+}).$
\end{itemize}
In particular, $\hull(\{\nabla_\lambda\otimes X_v\}_{v\in W}) =\D$.
\end{lemma}

\begin{proof}
As $\cone(P(-e_p)^*\to X_p)$ lies in $\hull(\{\nabla_\lambda\otimes P(-e_v)^*\}_{v\prec  p,\lambda\in X(\bt)_+})$, we may argue by induction on the size of $\{\;v\in W\mid v\preceq p\;\}$. 

Similarly,
$\cone(Y_v\to Q(e_v))$ lies in $\hull(\{\nabla_\lambda\otimes Q(e_v)\}_{v\succ  p,\lambda\in X(\bt)_+})$ and we may use induction on the size of 
$\{\;v\in W\mid v\succeq p\;\}$.
\end{proof}

\begin{remark}\label{rem:ends}
    If $p=\id$, then $P(-e_p)^*=\sk$ and ${\sf Q}_{\succeq p}=\D$, so $X_p=\sk$.

    Let $\rho=\sum_{\alpha\in \Pi}\omega_\alpha$. If $p=w_0$, then $Q(e_p)=\sk_{-\rho}$ by Proposition \ref{prop:1dimQ} and 
    ${\sf P}_{\preceq p}=\D$, so $Y_p=\sk_{-\rho}$.
\end{remark}

\subsection{\unboldmath Computing morphisms between \unboldmath$X_p$ and $Y_p$}\label{subsec:X_p-Y_p-morphisms}

The objects $X_v,Y_v$ are the key ingredient for constructing $\bg$-linear semi-orthogonal decompositions of $\Dd^b (\rep(\fb))$. Combined with the results of Section \ref{sec:X_p&Y_p-isom} below, the objects $X_v$ will give the sought-for $\bg$-linear subcategories of $\D=\Dd^b (\rep(\fb))$. The goal of this section is to compute $\Hom_\D^{\bullet}(X_v,Y_v)$ and also variations on $\Hom_\D^{\bullet}(X_v,Y_v)$ that  take $\bg$-linear structure into account. This computation is achieved by the results of Section \ref{sec:b-cohom_vanish}.
We may assume by Subsection \ref{subsec:bounded-triang_cat} that the $X_v$, $Y_v$ are represented by complexes in $\rep_\fr(\fb)$.
 
\begin{proposition}\label{prop:X_p-G-exceptional-coll}
Let $\D=\Dd^b (\rep ({\fb}))$.
Let $M, N\in\rep(\bg)$. Let $v,w\in W$.
\begin{enumerate}
\item $\Hom_\D(M\otimes X_v,N\otimes Y_v[i])=\Ext^i_ \bg (M,N) $ for all $i$.
\item If $w\succ  v$ then $\Hom_\D(M\otimes X_v,N\otimes Y_w[i])=0$ for all $i$.
\end{enumerate}
\end{proposition}
\begin{proof}
$(1).$ 
Tensoring the triangles (\ref{eq:def_triang_X_p}) and (\ref{eq:def_triang_Y_p}) with modules $M$ and $N$, respectively, we get
exact triangles
\begin{equation}\label{eq:otimesM1}
 M\otimes P(-e_p)^*\rightarrow M\otimes X_p\rightarrow M\otimes {\sf cone}(P(-e_p)^*\rightarrow X_p)
\end{equation}
and 
\begin{equation}\label{eq:otimesN1}
N\otimes Y_p\rightarrow N\otimes Q(e_p)\rightarrow N\otimes {\sf cone}(Y_p\rightarrow Q(e_p)).
\end{equation}
We have that ${\sf cone}(P(-e_p)^*\rightarrow X_p)\in {\sf P}_{\prec p}$ and ${\sf cone}(Y_p\rightarrow Q(e_p))\in {\sf Q}_{\succ p}$. Both subcategories ${\sf P}_{\prec p}$ and ${\sf Q}_{\succ p}$ are $ \bg$-linear by Proposition \ref{prop:G linear hull}, so we also have
$M\otimes {\sf cone}(P(-e_p)^*\rightarrow X_p)\in {\sf P}_{\prec p}$ and $N\otimes {\sf cone}(Y_p\rightarrow Q(e_p))\in {\sf Q}_{\succ p}$. 

We have $\Hom _{\D}(M\otimes {\sf cone}(P(-e_p)^*\rightarrow X_p),N\otimes Y_p[i])=0$ for all $i$ since $Y_p\in {\sf Q}_{\succeq p}\cap {\sf P}_{\preceq p}\subset {\sf Q}_{\succeq p}$ and ${\sf Q}_{\succeq p}={\sf P}_{\prec p}^{\perp}$ by 
 (\ref{eq:p-q-SOD-2}). Applying $\Hom _{\D}(-,N\otimes Y_p)$ to the triangle (\ref{eq:otimesM1}), we then obtain $\Hom _{\D}(M\otimes X_p,N\otimes Y_p[i])=\Hom _{\D}(M\otimes P(-e_p)^*,N\otimes Y_p[i])$ for all $i$.
 
Next, we have $\Hom _{\D}(M\otimes P(-e_p)^*,N\otimes {\sf cone}(Y_p\rightarrow Q(e_p)[i]))=0$ for all $i$ since $P(-e_p)^*\in  {\sf P}_{\preceq p}$ and ${\sf Q}_{\succ p}={\sf P}_{\preceq p}^{\perp}$ by 
 (\ref{eq:p-q-SOD-1}). Applying $\Hom _{\D}(M\otimes P(-e_p)^*,-)$ to the triangle (\ref{eq:otimesN1}), we then obtain $\Hom _{\D}(M\otimes P(-e_p)^*,N\otimes Y_p[i])=\Hom _{\D}(M\otimes P(-e_p)^*,N\otimes Q(e_p)[i])$ for all $i$. The latter group is isomorphic to $\Ext ^i_{\fb}(M\otimes P(-e_p)^*,N\otimes Q(e_p))$. Now the statement in (1) follows by Corollary \ref{cor:left orthogonal}, (1).

$(2).$ The second part follows from equation (\ref{eq:QinperP}).
\end{proof}

As is clear from Proposition \ref{prop:X_p-G-exceptional-coll}, it is desirable to know  that $X_v$ is isomorphic to $Y_v$ for all $v\in W$.

\section{\unboldmath Isomorphism of \unboldmath$X_p$ with $Y_p$}\label{sec:X_p&Y_p-isom}
In this section we show that $X_p$
is isomorphic with $Y_p$, using a refinement of Theorem \ref{th:generation}. This is where the Steinberg weights tie everything together.
Recall that $\sk$ is a field or $\Z$.

\

Let $\lambda \in X(\bt)$. We denote by $\lambda ^+$ the dominant weight in the $W$-orbit of $\lambda$ and we let $w\in W$ be the minimal element of $W$
such that $\lambda ^+=w\lambda$. 

Then put 
$$\pre_\lambda=\{v\in W\mid e_v<_a\lambda\},$$
where $<_a$ is the antipodal excellent order \ref{def:antipodal}.

We work again with the $M_v$ of Subsection \ref{subsec:Mv's}.
We have the following technical variation on Theorem \ref{th:generation}.

\begin{theorem}
\label{th:cone generation}
  Let $\lambda,\mu\in X(\bt)$, $\mu\neq\lambda$.
  Assume that $\mu$
  is a weight of $P(-\lambda)^*$ or
  that $(\mu,\mu)<(\lambda,\lambda)$.
  
  Then
  $\sk_\mu$ is an object of the triangulated hull of
$\{\nabla_\nu\otimes M_v\mid \nu\in X(\bt)_+,~v\in\pre_\lambda\}.$
  \end{theorem}
  
\begin{proof}
We adapt the proof of Theorem \ref{th:generation}. First let us rephrase what we need to show. We need to show that if $\mu<_a\lambda$, then $\sk_\mu$ is in the triangulated hull of 
$$\{\nabla_\nu\otimes M_v\mid \nu\in X(\bt)_+,~v\in W, ~e_v<_a\lambda\}.$$
By  Proposition \ref{prop:G linear hull} this hull is $\bg$-linear.
We want to show by induction along $<_a$ that $\sk_\mu$ is in the  hull.
 So assume $\sk_\sigma$ is in the hull for $\sigma<_a\mu$.

There are two cases: $\mu$ is a Steinberg weight or it is not. If it is a Steinberg weight, then it is an $e_v$ with $e_v<_a\lambda$. And all weights $\sigma$ of $M_v$ that are distinct from $\mu=e_v$ satisfy $\sigma<_a\mu$.
If $\mu$ is not a Steinberg weight, then by Remark \ref{cor:if not Steinberg then tensor} there is a module $N$  in  $\{\nabla_\nu\otimes \sk_\tau\mid \nu\in X(\bt)_+,~(\tau,\tau)<(\mu,\mu)\}$ 
with $\mu$ a weight  of $N$ of multiplicity one, and all weights  $\sigma$ of $N$ satisfying 
$\sigma\leq_a\mu$. But $(\tau,\tau)<(\mu,\mu)$ implies $\tau<_a\mu$.
\end{proof}
For completeness we mention
\begin{lemma}\label{cor:ordered_generate}
    Let $\lambda\in X(\bt)$. Then $e^\lambda$ lies in the $R(\bg)$-submodule of $R(\bt)$
generated by the $[M_v]$ with $e_v\leq_a\lambda$.
\end{lemma}
\begin{proof}
    `Same proof' by induction along $\leq_a$.
\end{proof}
\begin{remark}\label{rem:Steinberg from anti}
    Thus $\lambda$ is a Steinberg weight if and only if $e^\lambda$ does not lie in the $R(\bg)$-submodule of $R(\bt)$
generated by the $e^\mu$ with $\mu<_a\lambda$.
 If $\lambda$ is a Steinberg weight, then $\lambda=e_w$ where $w\in W$ is the element of minimal length making $w\lambda$ dominant.
\end{remark}

\

Let $f_p$ be the natural map from $P(-e_p)^*$ to $Q(e_p)$ (Theorem \ref{th:key_orthogonality}). Recall from the proof of Theorem \ref{th:key_orthogonality} that $f_p$ factors as a surjection $P(-e_p)^*\to\sk_{e_p}$, followed by an injection $\sk_{e_p}\to Q(e_p)$.
We have an exact triangle 
$$\kernel(f_p)[1]\to \cone(f_p)\to \coker(f_p).$$ 
Take $M_v=Q(e_v)$ for $v\succ p$ and $M_v=P(-e_v)^*$ for $v\prec p$. We do not need to specify $M_p$, because $p\notin \pre_{e_p}$.

\begin{corollary}\label{cor:cone-f_p}
Let $p\in W$. Then $\cone(f_p)$ belongs to the hull of the union of the
following three sets:
\begin{eqnarray*}
\{\;\nabla_\lambda\otimes Q(e_v)&\mid& \lambda \mbox{ is dominant,  } v\succ  p\mbox{ and }(e_v,e_v)<(e_p,e_p)\;\}\cup\\
\{\;\nabla_\lambda\otimes Q(e_v)&\mid& \lambda \mbox{ is dominant, }  v\succ p\mbox{ and }v^{-1} pe_p=e_v\neq e_p\;\}\cup\\
\{\;\nabla_\lambda\otimes P(-e_v)^*&\mid& \lambda \mbox{ is dominant, } v\prec p \mbox{ and }(e_v,e_v)<(e_p,e_p)\;\}.
\end{eqnarray*}
In particular, $\cone(f_p)$ belongs to the hull of $ {\sf P}_{\prec p}\cup {\sf Q}_{\succ p}$.
\end{corollary}

\begin{proof}
Apply Theorem \ref{th:cone generation} with $\lambda=e_p$ to 
$\kernel(f_p)$ and $\coker(f_p)$.
\end{proof}
\begin{remark}
    Observe that the proof involves the partial order $<_a$, but the conclusion, that $\cone(f_p)$ belongs to the hull of $ {\sf P}_{\prec p}\cup {\sf Q}_{\succ p}$, refers to $\prec$. This conclusion just
    uses that only $M_v$ with $v\neq p$ are used. It is at the Steinberg weights that the Bruhat order meets the antipodal excellent order. Both orders are important. Compare Proposition \ref{prop:p-q-SOD-2} and Lemma \ref{cor:ordered_generate}.
\end{remark}
We are now in a position to prove the isomorphism $X_p=Y_p$. 
Note that $i^{\sf P{\ !}}_{\preceq p}Q(e_p)=Y_p$ and $i^{\sf P{\ !}}_{\preceq p}(X_p)=X_p$ as $i^{\sf P{\ !}}_{\preceq p}={\sf id}_{{\sf P}_{\preceq p}}$ on ${\sf P}_{\preceq p}$. 
Moreover, the functor $i^{\sf P{\ !}}_{\preceq p}$ annihilates ${\sf Q}_{\succ p}$
as ${\sf Q}_{\succ p}={\sf P}_{\preceq p}^{\perp}$. That is:
\begin{eqnarray*}
i^{\sf P{\ !}}_{\preceq p}({\sf Q}_{\succ p})=0,
\end{eqnarray*}
and 
\begin{eqnarray*}
i^{\sf P{\ !}}_{\preceq p}({\sf P}_{\prec p})\subset {\sf P}_{\prec p}.
\end{eqnarray*}
As $\cone(f_p)$ belongs to the hull of $ {\sf P}_{\prec p}\cup {\sf Q}_{\succ p}$, we conclude that \begin{eqnarray}i^{\sf P{\ !}}_{\preceq p}(\cone(f_p))\in {\sf P}_{\prec p}.\end{eqnarray}
Consider the exact triangle 
\begin{equation}
P(-e_v)^*\rightarrow Q(e_v)\rightarrow \cone(f_p).
\end{equation}
Applying to the triangle the functor $i^{\sf P{\ !}}_{\preceq p}$ and remembering 
that $i^{\sf P{\ !}}_{\preceq p}(Q(e_v))=Y_p$, we obtain 
\begin{equation}
P(-e_v)^*\rightarrow Y_p\rightarrow i^{\sf P{\ !}}_{\preceq p}(\cone(f_p)) 
\end{equation}
with $i^{\sf P{\ !}}_{\preceq p}(\cone(f_p))\in {\sf P}_{\prec p}$.
Thus
\begin{equation}\label{eq:cone in P} \cone(P(-e_v)^*\rightarrow Y_p)\in  {\sf P}_{\prec p}.
\end{equation} 
Now apply to the triangle 
\begin{equation}
P(-e_v)^*\rightarrow Y_p\rightarrow \cone(P(-e_v)^*\rightarrow Y_p) 
\end{equation}
the projection functor ${i^{\sf Q{\ \ast}}_{\succeq p}}$ onto ${\sf Q}_{\succeq p}$.
We have ${i^{\sf Q{\ \ast}}_{\succeq p}}(P(-e_v)^*)=X_p$ and ${i^{\sf Q{\ \ast}}_{\succeq p}}(Y_p)=Y_p$, as $Y_p\in{\sf Q}_{\succeq p}\cap {\sf P}_{\preceq p}\subset {\sf Q}_{\succeq p}$ by (\ref{eq:Y in Q and P}). Further ${i^{\sf Q{\ \ast}}_{\succeq p}}(\cone(P(-e_v)^*\rightarrow Y_p))=0$ because of (\ref{eq:cone in P}), as ${i^{\sf Q{\ \ast}}_{\succeq p}}$ annihilates ${\sf P}_{\prec p}$. Thus, $X_p=Y_p$.

\section{\unboldmath Semi-orthogonal decomposition of \unboldmath$\Dd ^b(\rep(\fb))$ as a 
$\bg$-linear category}\label{sec:SODofB-modoverG-mod}

 Recall that $\sk$ is a field or $\Z$.
\begin{theorem}\label{th:X_p-G-exceptional-coll}
Let $\D=\Dd^b (\rep ({\fb}))$.
Let $M$, $N$ be finitely generated $ \bg $-modules. Let $v,w\in W$.
\begin{enumerate}
\item $\Hom_\D(M\otimes X_v,N\otimes X_v[i])=\Ext^i_ \bg (M,N) $ for all $i$.
\item If $w\succ  v$ then $\Hom_\D(M\otimes X_v,N\otimes X_w[i])=0$ for all $i$.
\end{enumerate}
\end{theorem}
\begin{proof}
We may replace $Y_p$ with $X_p$ in Proposition \ref{prop:X_p-G-exceptional-coll}.
\end{proof}

In fact we have
\begin{theorem}
\label{th:Z-X_p-G-exceptional-coll}
Let $\D=\Dd^b (\rep ({\fb}))$.
Let $M, N\in \Dd^b(\rep(\bg))$ and $v,w\in W$.
\begin{enumerate}
\item $\RHom_\D(M\otimes X_v,N\otimes X_v)=\RHom_ {\Dd^b(\rep(\bg))} (M,N) $,
\item If $w\succ  v$ then $\RHom_\D\rule{0pt}{1.1em}(M\otimes X_v,N\otimes X_w)=0$.
\end{enumerate}
\end{theorem}
\begin{proof}

Part (1).   It suffices to treat the case $M=\sk$
because if $L\in \Dd^b(\rep(\bg))$, and $L^*$ is defined as in Subsection \ref{subsec:bounded-triang_cat}, then
$\RHom_\D(L\otimes X_v,N\otimes X_v)=\RHom_\D( X_v,L^*\otimes N\otimes X_v)$ and
$\RHom_  {\Dd^b(\rep(\bg))}  (L,N) =\RHom_  {\Dd^b(\rep(\bg))} (\sk,L^*\otimes N)$.  By Theorem \ref{th:X_p-G-exceptional-coll} 
we  know the result when $N\in \rep(\bg)$ and $\RHom_ \bg (\sk,N) $ is concentrated in at most one degree. 
For instance, this is the case when $N=\nabla_\nu$ for some dominant $\nu$.
So it holds for  $N$ in the hull $\Dd^b(\rep(\bg))$ of the $\nabla_\nu,\nu\in X(\bt)_+$. 

Part (2) is easier.
\end{proof}

\begin{corollary}
\label{cor:Rind_B^G(X_p*otimesX_p)}

\

\begin{enumerate}
\item $\Rind_{\fb}^{ \bg}(X_v^{\ast}\otimes X_v)=\sk$.
\item If $w\succ  v$ then $\Rind _{\fb}^{ \bg}(X_v^{\ast}\otimes X_w)=0$.
\end{enumerate}
\end{corollary}
\begin{proof}Part (1): Let $\Phi _v:\Dd^b(\rep(\bg))\rightarrow \D$ be the $\bg$-linear functor given by $\Phi_v(-)=(-)\otimes X_v$. By Theorem \ref{th:Z-X_p-G-exceptional-coll} (1), the functor $\Phi_v$ is fully faithful. Then Proposition \ref{prop:G-linear_exceptional} (3) implies the claim.

Part (2): Let $\Phi _v, \Phi _w$ be the two $\bg$-linear functors defined as in 1) above. By Theorem \ref{th:Z-X_p-G-exceptional-coll} (2), the pair of $\bg$-linear categories $\Phi _w(\Dd^b(\rep(\bg)),\Phi _v(\Dd^b(\rep(\bg))$ is semiorthogonal in $\D$.
Then Proposition \ref{prop:G-linear_sod} implies the claim.
\end{proof}

\begin{proposition}\label{prop:sfX_p-admissibility}
    Let $p\in W$.
    The 
strictly full subcategory $\tS$ of
$\D=\Dd ^b(\rep(\fb))$ generated by $\{M\otimes X_p \mid M\in \Dd ^b(\rep( \bg))\}$
is an admissible subcategory.
\end{proposition}

\begin{proof}
Recall that by Proposition \ref{prop:resolution} we may extend the exact bifunctors
 $$-\;\otimes_\sk-:\rep_\fr( \fb)\times \rep_\fr( \fb)\to \rep_\fr( \fb)$$ and $$\Hom_\sk(-,-):\rep_\fr( \fb)\times \rep_\fr( \fb)\to \rep_\fr( \fb)$$ 
 to $\D$. Also recall that $M^*$ means $\Hom_\sk(M,\sk)$ for
 $M\in\D$.
Then the right adjoint of the inclusion of $\tS$ into $\D$ is $$\Rind_ \fb^ \bg(-\otimes_\sk X_p^*)\otimes_\sk X_p$$ 
because 
\begin{gather*}
\Hom_\tS( E\otimes_\sk X_p,\Rind_ \fb^ \bg(F\otimes_\sk X_p^*)\otimes_\sk X_p )=\\
\Hom_\D( E\otimes_\sk X_p, \Rind_ \fb^ \bg(F\otimes_\sk X_p^*)\otimes_\sk 
X_p)=^{Thm\ \ref{th:Z-X_p-G-exceptional-coll}}\\
\Hom_{\Dd^b(\rep(\bg))}(E, \Rind_ \fb^ \bg(F\otimes_\sk X_p^*))=
\Hom_{\D}(E , F\otimes_\sk X_p^*)=\\
\Hom_\D(E\otimes_\sk X_p , F) 
\end{gather*}
for $E\in \Dd ^b(\rep( \bg))$, $F\in \D$.

The left adjoint is $$(\Rind_ \fb^ \bg((-)^*\otimes_\sk X_p))^*\otimes_\sk X_p$$
because 
\begin{gather*}
\Hom_\tS( (\Rind_ \fb^ \bg(F^*\otimes_\sk X_p))^*\otimes_\sk X_p,E\otimes_\sk X_p)=\\
\Hom_\D( E^*\otimes_\sk X_p,\Rind_ \fb^ \bg(F^*\otimes_\sk X_p)
\otimes_\sk X_p)=^{Thm\ \ref{th:Z-X_p-G-exceptional-coll}}\\
\Hom_{\Dd^b(\rep(\bg))}(E^*, \Rind_ \fb^ \bg(F^*\otimes_\sk X_p))=
\Hom_{\D}(E^* , F^*\otimes_\sk X_p)=\\
\Hom_\D(F,E\otimes_\sk X_p )
\end{gather*}
for $E\in \Dd ^b(\rep( \bg))$, $F\in \D$.
\end{proof}

\begin{theorem}\label{th:semi-orthogonal}
For $v\in W$, let ${\sX }_v$ denote the 
triangulated hull in $\D=\Dd ^b(\rep(\fb))$ of $\{M\otimes X_v \mid M\in \Dd ^b(\rep( \bg))\}$.
Then the category $\D$ has a $ \bg$-linear semi-orthogonal decomposition 
\begin{equation}
\D= \langle \sX _v\rangle _{v\in W}
\end{equation}
with respect to the order $\prec$ on the Weyl group $W$. Each subcategory $\sX _v$ is equivalent to $\Dd ^b(\rep( \bg))$.
\end{theorem}
\begin{proof}
Let $v\in W$ and let $\Phi _v:\Dd^b (\rep( \bg ))\rightarrow \D$ be the $ \bg$-linear functor $M\rightarrow M\otimes _{\sk} X_v$. 
Now Theorem \ref{th:Z-X_p-G-exceptional-coll}
 gives that each $\Phi _v$ is fully faithful.

Let now $w\succ v$. 
Similarly to the proof of 
Corollary \ref {cor:Rind_B^G(X_p*otimesX_p)} (2), the pair of $\bg$-linear categories $\Phi _w(\Dd^b(\rep(\bg)),\Phi _v(\Dd^b(\rep(\bg))$ is semiorthogonal in $\D$. Thus, the collection $( \sX _v) _{v\in W}$ is a $\bg$-linear semi-orthogonal sequence of full triangulated subcategories of $\D$ with respect to the order $\prec$ on $W$.
Lemma \ref{lem:samehull} states that the triangulated hull of $( \sX _v )_{v\in W}$ coincides with $\D$. 
Thus, $\langle \sX _v\rangle _{v\in W}$ is a $ \bg$-linear semi-orthogonal decomposition of $\D$.
\end{proof}

\section{\unboldmath Full exceptional collections in \unboldmath$\Dd ^b( \bg/ \fb)$}\label{sec:FEConGmodB}
Recall that $\sk$ is a field or $\Z$. Recall the functor  $\Ll:
  \Dd ^b({\rep}({\fb}))\rightarrow \Dd ^b( \bg/\fb)$ (Subsection \ref{subsec:sheafification_functor}).
We put $\tX_v:=\Ll(X_v)$ for $v\in W$. 

\

Theorem \ref{th:semi-orthogonal} implies the following:

\begin{theorem}\label{th:coh-semi-orthogonal}Let $\D=\Dd^b (\bg/\fb)$.
Let $v,w\in W$.
\begin{enumerate}
\item $\Hom_\D(\tX_v,\tX_v[i])=\begin{cases}\sk &\text{if $i=0$},\\
0&\text{else.}\end{cases}$\\

\

\item If $w\succ  v$ then $\Hom_\D(\tX_v,\tX_w[i])=0$ for all $i$.\\
\item The triangulated hull of $\{\tX_v\mid v \in W\}$ is $\D$.
\end{enumerate}
In other words, the collection of objects $(\tX_v)_{v \in W}$ is a full exceptional collection in $ \D$.
\end{theorem}

\begin{proof}
$(1)$. By Proposition \ref{prop:ForRind} and  Corollary \ref{cor:Rind_B^G(X_p*otimesX_p)} (1) we have 
$$\RHom_\D(\tX_v,\tX_v)=\RHom _{\D}(\Ll (X_v),\Ll (X_v))={\sf For}( \Rind_{\fb}^{ \bg}(X_v^{\ast}\otimes X_v))=\sk.$$

$(2)$. This is the same argument as in (1), using Corollary \ref{cor:Rind_B^G(X_p*otimesX_p)} (2).

$(3).$  
 Observe that by the previous items (1) and (2), the collection of $\tX_v, v\in W$ is exceptional. 
Thus, $\hull( \tX_v, v\in W)\subsetq \D$  is an admissible (hence thick) subcategory of $\D$, see Section \ref{sec:rappel_triangulated}. 

Recall that the functor $\Ll$ is monoidal; thus, $\Ll (\nabla_{\lambda}\otimes X_v)=\Ll (\nabla_{\lambda})\otimes\tX_v$ for $v\in W$, $\lambda\in X(\bt)_+$. 
Now $\Ll (\nabla_{\lambda})$ is a ($ \bg$-equivariant) trivial vector bundle
$\nabla_{\lambda}\otimes_\sk \Oo_{\bg/ \fb}$, so the image of 
$\hull(\{\nabla_{\lambda}\otimes X_v\}_{v\in W,\lambda\in X(\bt)_+})$ under $\Ll$ coincides with $\hull( \tX_v, v\in W)$. 

By Lemma \ref{lem:samehull},  $\hull(\{\nabla_\lambda\otimes X_v\}_{v\in W,\lambda\in X(\bt)_+})$ 
equals $\Dd^b (\rep( \fb))$, so $\hull( \tX_v, v\in W)$ contains the image of 
$\rep( \fb)$.  This image contains an ample line bundle together with its powers. 
By Subsection \ref{subsec:admfromex} the subcategory $\hull( \tX_v, v\in W)$ is admissible, hence thick.
Thus we get 
$\hull( \tX_v, v\in W)=\D=\Dd^b (\bg/\fb)$ by Corollary \ref{cor:generatingGmodB}.
\end{proof}

\begin{remark}\label{rem:basechangeG-to-point}
{\rm Theorem \ref{th:coh-semi-orthogonal} would be a simple instance of base change for semi-orthogonal decompositions \cite[Theorem 5.6]{Kuzbasech}, if we had the base change theorem at our disposal in the extended setting of quotient stacks. Specifically, it is about the following particular case of base change that in the classical setting of schemes can be stated in elementary terms.

Consider a flat morphism $\pi\colon X\rightarrow S$ between smooth projective varieties, a closed point $s\in S$ and the base change diagram along the embedding $i_s\colon s\hookrightarrow S$:

\begin{figure}[H]
$$\xymatrix @C5pc @R4pc {
X_s\ar[r]^i\ar@<-0.1ex>[d]_{\pi _s} & X \ar@<-0.4ex>[d]^{\pi} \\
       s\ar[r]^{i_s} & S}
$$
\end{figure}

Assume given an $S$-linear semi-orthogonal decomposition $\langle \D _1,\dots ,\D _n\rangle$ of $\Dd ^b(X)$, such that each admissible subcategory $\D _i$ is equivalent to $\Dd ^b(S)$. Then each embedding functor $\Phi _k:\D _k\rightarrow \Dd ^b(X), k=1,\dots,n$ is given by $\Phi _k(-)=(-)\otimes E_k$ where $E_k,k=1,\dots, n$ is a collection of objects of $\Dd ^b(X)$ with the following two properties: 1) $R\pi _{\ast}R{\mathcal Hom}(E_i,E_j)=0$ for $i>j$ and 2) $R\pi _{\ast}{\mathcal Hom}(E_k,E_k)=\Oo _S$ for all $k$ (cf. Propositions \ref{prop:G-linear_sod} and \ref{prop:G-linear_exceptional}). Restricting the objects $E _k\in \Dd ^b(X)$ to the fibre $X_s$, one obtains a collection of objects $\Ee _k:=i^{\ast}E_k, k=1,\dots, n$. Now the claim is that the collection $\Ee _k, k=1,\dots, n$ is exceptional in $\Dd ^b(X_s)$: for instance, to see that $\RHom _{X_s}(\Ee _i,\Ee _j)=0$ for $i>j$, by \cite[Lemma 2.32]{KuzHPD} base change holds for the above Cartesian square, thus \
\begin{eqnarray}
&\RHom _{X_s}(\Ee _i,\Ee _j)=R{\pi _s}_{\ast}(\Ee _i^{\ast}\otimes ^{ L}\Ee _j)=R{\pi _s}_{\ast}i^{\ast}(E_i^{\ast}\otimes ^{ L}E _j)= \\
& i^{\ast}_sR\pi _{\ast}(E_i^{\ast}\otimes ^{L}E _j)=i^{\ast}_sR\pi _{\ast}R{\mathcal Hom}(E_i,E_j)=0
\end{eqnarray}
by 1) above. Similarly for $\RHom _{X_s}(\Ee _i,\Ee _i)={\sf k}$ for $i=1,\dots, n$. Finally, the objects $\Ee _i,i=1,\dots n$ generate $\Dd ^b(X_s)$: for if there was a non-trivial object $\Ff \in \langle \Ee _1,\dots ,\Ee _n\rangle  ^{\perp}$, its pushforward ${i_s}_{\ast}\Ff \in \Dd ^b(X)$ would be a non-trivial object in the right orthogonal to the semi-orthogonal decomposition $\langle \D _1,\dots ,\D _n\rangle$ of $\Dd ^b(X)$, a contradiction.

Provided that the base change theorem holds in the setting of quotient stacks, Theorem \ref{th:coh-semi-orthogonal} would then follow by setting in the above diagram $X=[{\sf pt}/\fb$] (the classifying stack of $\fb$), $S=[{\sf pt}/\bg]$ (the classifying stack of $\bg$), $s={\sf pt}$ with $X_s=\bg/\fb$.
}
\end{remark}

\subsection{\unboldmath Variations} \label{subsec:var}

The $\tX_p$ may depend on the choice of the total order $\prec$
on $W$. Therefore it is not clear that $\tX_p$ is perpendicular to
$\tX_q$ when $\ell(p)=\ell(q)$.
In fact computer-assisted computations in $\sK_ {\bt}( \bg /{\fb})$ indicate 
that this fails already for type ${\bf B}_3$ with $\ell(p)=\ell(q)=3$ and for types ${\Cc}_3$ and $\mathbf A_4$ with $\ell(p)=\ell(q)=2$. That is, it fails for at least one choice of $\prec$.
It follows that in these types some $\tX_p$ really depend on the choice of the total order $\prec$.
\subsection{\unboldmath {Baby case: rank one}}
In rank one we are dealing with $\bg=\mathrm{SL}_2$, $W=\{s,\id\}$.
Let $\rho$ be the fundamental weight. We now use Remark \ref{rem:ends}.
One has $e_s=-\rho$, $e_\id=0$, $P(0)=\sk$, $Q(-\rho)=\sk_{-\rho}$, ${\sf Q}_{\succeq \id}={\sf P}_{\preceq s}=
\Dd ^b(\rep(\fb))$, $X_s=Y_s=Q(-\rho)=\sk_{-\rho}$, $X_\id=P(0)^*=\sk$, $\tX_s=\Ll({-\rho})=\Oo(-1)$, $\tX_\id=\Oo$,
where $\Oo$ is the structure sheaf of the projective line
$\bg/\fb$.
One gets the familiar full
exceptional collection $(\Oo(-1),\; \Oo)$ on the projective line.


\section{\unboldmath Generalised flag varieties}\label{sec:caseP}
Recall that $\sk$ is a field or $\Z$.
Let $\bp$ be a parabolic subgroup containing $\fb$. 
We seek a full exceptional collection on $\bg/\bp$.
Let $W_\bp$ be the Weyl group of $\bp$,
generated inside $W$ by the $s_\alpha$ with $\alpha$ simple and
$\bp_\alpha\subset \bp$.
Let $W^\bp$ be the set of minimal coset representatives of $W/W_\bp$, cf.\  Lemma \ref{lem:minimal-coset-representative}. 
If $\sk=\mathbb C$, then according to Steinberg \cite{St} we may use as generators of  the $R(\bg)$-module $R(\bp)=\sK_0(\rep( \bp))$,  
the classes of the
irreducible $\bp$-modules with highest weight $e_v$ where $v$
runs over $W^\bp$, and $e_v$ is still defined as in \ref{subsec:Steinberg}. Inspired by that,
we restrict our total order $\prec$ from $W$ to $W^\bp$.
Note our convention that the notation $e_v$ keeps the meaning it had when $\bp=\fb$. But  the $e_v$ with $v\in W^\bp$ will be more relevant than the other $e_v$.
\begin{remark}The map $W\to W^\bp$ which sends $w\in W$ to the minimal representative of the coset $wW_\bp$ is a poset map, by Lemma \ref{lem:minimal-coset-representative}.
If one has chosen a total order on $W^\bp$ refining the (restriction to $W^\bp$ of) the Bruhat order, then this
chosen order can be extended to
a total order on $W$ that refines the Bruhat order on $W$. We already had to choose $\prec$ on $W$ in section \ref{subsec:totalorderonW}, so it makes sense to keep that order $\prec$ and restrict it to $W^\bp$.
\end{remark}

For $v\in W^\bp$ we will find ${\hat X}_v\in \Dd^b (\rep ( \bp ))$  such that
 
\begin{theorem}
\label{th:Z-X^P_p-G-exceptional-coll}
Let $\D=\Dd^b (\rep ( \bp ))$.
Let $M, N\in \Dd^b(\rep(\bg))$ and $v,w\in W^\bp$.
\begin{enumerate}
\item $\RHom_\D(M\otimes {\hat X}_v,N\otimes {\hat X}_v)=\RHom_ {\Dd^b(\rep(\bg))} (M,N) $,
\item If $w\succ  v$ then $\RHom_\D(M\otimes {\hat X}_v,N\otimes {\hat X}_w)=0$.
\end{enumerate}
\end{theorem}

\begin{theorem}\label{th:Psemi-orthogonal}
For $v\in W^\bp$, let ${\hat{\sX}} _v$ denote the the 
triangulated hull in 
$\D=\Dd ^b(\rep(\bp))$ of $\{M\otimes {\hat X}_v \mid M\in \Dd ^b(\rep( \bg))\}$.
Then the category $\D$ has a $ \bg$-linear semi-orthogonal decomposition 
\begin{equation}
\D= \langle {\hat{\sX}} _v\rangle _{v\in W^\bp}
\end{equation}
with respect to the order $\prec$ on  $W^\bp$. Each subcategory ${\hat{\sX}} _v$ is equivalent to $\Dd ^b(\rep( \bg))$.
\end{theorem}

\begin{definition}
 We put $\hat\tX_v:=\Ll_{\bg/\bp}({\hat X}_v)$ for $v\in W^\bp$, in the notation of Subsection \ref{subsec:sheafification_functor}. 
\end{definition}

\begin{theorem}\label{th:Pcoh-semi-orthogonal}Let $\D=\Dd^b (\bg / \bp)$.
Let $v,w\in W^\bp$.
\begin{enumerate}
\item $\Hom_\D(\hat\tX_v,\hat\tX_v[i])=\begin{cases}\sk &\text{if $i=0$},\\
0&\text{else.}\end{cases}$\\

\item If $w\succ  v$ then $\Hom_\D(\hat\tX_v,\hat\tX_w[i])=0$ for all $i$.\\
\item The triangulated hull of $\{\hat\tX_v\mid v \in W^\bp\}$ is $\D$.
\end{enumerate}
In other words, the collection of objects $(\hat\tX_v)_{v \in W^\bp}$ is a full exceptional collection in $ \D$.
\end{theorem}

To prove Theorems \ref{th:Psemi-orthogonal} and \ref{th:Pcoh-semi-orthogonal} we will need to find replacements, often with hatted notation, of the key ingredients used in 
the case $\bp=\fb$.
This will take the rest of the section.

Recall that by Subsection \ref{subsec:G-linear-triang_cat}  we may view $\rep(\bg)$ as a subcategory of $\rep(\fb)$
and $\Dd^b(\rep(\bg))$ as a subcategory of
$\Dd^b(\rep(\fb))$, so that we may suppress 
$\res^\bg_\fb$ in the notation. For similar reasons we may 
view $\rep(\bp)$ as a subcategory of $\rep(\fb)$ and $\Dd^b(\rep(\bp))$ as a subcategory of
$\Dd^b(\rep(\fb))$ and we may suppress $\res^\bp_\fb$.
Then $\Dd^b(\rep(\bg))$ becomes a subcategory of $\Dd^b(\rep(\bp))$. And when we say that a certain $\fb$-module $M$ is a $\bp$-module, this will mean that $M=\res^\bp_\fb\ind_\fb^\bp(M)$.

\begin{remark}
The strategy for constructing $\hat\tX_v$ is the same as for $\tX_v$, but
the construction does not immediately imply the precise relation between $\tX_v$, $\hat\tX_v$ and $\bg/\fb\to \bg/\bp$.
More specifically, one expects that $\tX_v$ is always
the pull back of $\hat\tX_v$ for $v\in W^\bp$. This is indeed the case:
\end{remark}

\begin{theorem}\label{th:P and B}
Let $v\in  W^\bp$ and let $\pi:\bg/\fb\to \bg/\bp$ be the natural map.
Then ${\hat{X}} _v$ equals ${{X}} _v$ in $\Dd^b(\rep(\fb))$ and $\tX_v=\pi^*\hat\tX_v$.
\end{theorem}

Recall that 
$\Pi$ is the set of simple roots and 
$(\omega_\alpha,\beta^\vee)=\delta_{\alpha,\beta}$.

\begin{notation}
Let $\Pi^\bp$ be the set of simple roots $\alpha$ for which $\bp_\alpha\subsetq \bp$.
Let $\Pi^{\notin \bp}$ be the set of simple roots outside $\Pi^\bp$.
We say that $\lambda$ is $\bp$-dominant if $(\alpha^\vee,\lambda)\geq0$ for all $\alpha\in \Pi^\bp$.
The set of $\bp$-dominant weights is denoted $X^\bp_+$.
Every $W_\bp$-orbit of weights intersects $X^\bp_+$ in a unique element.
Let $\bg_\bp$ be the semisimple subgroup of $\bp$ with $\Pi^\bp$ as simple roots. (So $\bt\bg_\bp$ is a Levi subgroup $\bl_\bp$ of $\bp$ and $\bg_\bp$ is the derived subgroup of the Levi subgroup.)
Note that $\bg_\bp$ is simply connected.
\end{notation}

The definition of relative Schubert filtrations extends in an obvious way to the case $\sk=\Z$,
as does the definition of $\sH_w$ for $w\in W$, cf.\ \cite[II Chapter 13]{Jan}.

\begin{lemma}\label{lem:relativeHsacyclic}
Let $w\in W$.
    If $M\in\rep(\fb)$ has a relative Schubert filtration then it is $\sH_w$-acyclic.
\end{lemma}
\begin{proof}
Use Proposition \ref{prop:Schubert module} and the Universal coefficient Theorem \ref{th:universal coefficients}.
\end{proof}

\begin{lemma}\label{lem:HwQk}
    Let $\lambda\in X(\bt)$ and $w\in W$.
    Then $\sH_w(Q(\lambda))=\sH_w(Q(\lambda)_\Z)\otimes_\Z\sk$.
\end{lemma}
\begin{proof}By Lemma \ref{lem:relativeHsacyclic} this follows from the Universal coefficient Theorem \ref{th:universal coefficients}.
\end{proof}

\begin{lemma}\label{lem:extQQ}
    Let $\mu, \nu\in X(\bt)$ with $\nu\not<_a\mu$.
    Then \begin{equation}\label{eq:ExtQQ}
\Ext^i(Q(\mu),Q(\nu)) = \left\{
   \begin{array}{l}
    {\sf k}, \quad {\rm for} \quad i=0,\ \mu=\nu, \\
    0, \quad {\rm otherwise}. \\
   \end{array}
  \right.
\end{equation}
\end{lemma}
\begin{proof}
    First let $\sk$ be a field. If $i>0$ or $\mu\neq \nu$, then the vanishing  follows  from  \cite[Lemma 3.8(b)]{CPSHW}.
    As in the proof of Theorem \ref{th:key_orthogonality}, a $\fb$-module map $Q(\mu)\to Q(\mu)$ is determined by its restriction to the weight space of weight $\mu$.
    So (\ref{eq:ExtQQ}) holds when $\sk$ is a field.
    By the Universal coefficient Theorem and Lemma \ref{lem:dualmodule}, it also holds when $\sk=\Z$.
\end{proof}

Consider $M\in \rep_\fr(\fb)$ with a relative Schubert filtration.  Let $\mu$ be minimal for $<_a$ amongst the $\nu$ for which (a nonzero direct sum of copies of)
$Q(\nu)$ occurs in the filtration.

\begin{lemma}\label{lem:getQ}
The natural map $\Hom_\fb(Q(\mu),M)\otimes_\sk Q(\mu)\to M$ is injective and its
cokernel has a relative Schubert filtration.
\end{lemma}
\begin{proof}
This follows from Lemma \ref{lem:extQQ} by induction on the length of the relative Schubert filtration of $M$.
\end{proof}

\begin{lemma}\label{lem:Qexercise}
Let $M\in\rep_\fr(\Bm)$ so that for every field $\sk'$ the $\Gm_{\sk'}$-module 
$M\otimes_\Z\sk'$ has a relative Schubert filtration. Then $M$ has a relative Schubert filtration.
\end{lemma}
\begin{proof}
This is \cite[Exercise 7.2.8 (ii)]{WvdKTata}.
We will argue by induction on the rank of the $\Z$-module $M$.
The multiplicity $\text{mult}(\nu)$ with which $Q(\nu)\otimes_\Z\sk'$ occurs in a relative Schubert
filtration of $M\otimes_\Z\sk'$ depends only on the character $[M]\in R(\bt)$, not on the field. Let us take $\mu$ as above and consider the natural map $f:\Hom_\Bm(Q(\mu)_\Z,M)\otimes_\sk Q(\mu)_\Z\to M$. If $\sk'$ is a field then by Lemma \ref{lem:getQ} we may identify
$f\otimes_\Z\sk'$ with a natural map as in Lemma \ref{lem:getQ}, so $f\otimes_\Z\sk'$ is injective and has a cokernel with a relative Schubert filtration. As this holds for all fields $\sk'$ the cokernel of $f$ must lie in $\rep_\fr(\Bm)$, by the elementary divisor theorem. As the cokernel has lower rank than $M$, it has a relative Schubert filtration by the induction hypothesis.
We claim that the image of $f$ also has a relative Schubert filtration. As $f$ is injective, we only need to inspect the source of $f$.
By another application of the Universal coefficient Theorem \ref{th:universal coefficients} with Lemma \ref{lem:dualmodule} the $\Z$-module
 $\Hom_\Bm(Q(\mu)_\Z,M)$ is  free  with a rank $\text{mult}(\mu)$.
 The Lemma follows.
\end{proof}

\begin{lemma}\label{lem:HwPk}
    Let $\lambda\in X(\bt)$ and $w\in W$. Then
 $\sH_w(P(\lambda))=\sH_w(P(\lambda)_\Z)\otimes_\Z\sk$.
\end{lemma}
\begin{proof}
By Lemma \ref{lem:Qexercise} and Proposition \ref{prop:Schubert module} the module $P(\lambda)_\Z$ has a relative Schubert filtration and we may use Lemmas \ref{lem:relativeHsacyclic},
\ref{lem:HwQk}.
\end{proof}
\begin{lemma}\label{lem:PisP}Let $v\in W^\bp$, $w\in W$.
\begin{itemize}\item
Then $P(-e_v)$ is a $\bp$-module and therefore $P(-e_v)^*$ is a $\bp$-module.
\item If $e_w$ is  $\bp$-dominant, then $w\in W^\bp$.
\end{itemize}
\end{lemma}
\begin{proof}
We want to show that $P(-e_v)=\res^\bp_\fb\ind_\fb^\bp(P(-e_v))$, \emph{i.e.} that the natural map $\res^\bp_\fb\ind_\fb^\bp(P(-e_v))\to P(-e_v)$ is an isomorphism. It suffices to do this when $\sk$ is a field, because of
Lemma \ref{lem:HwPk}.
Let $\bp=\bu_{\bp}\bl_{\bp}$ be a Levi decomposition of $\bp$ where $\bu_{\bp}$ is the unipotent radical of $\bp$ and $\bl_{\bp}$ is the Levi component. Let $w_0^{\bp}\in W_\bp$ be the longest element and consider a reduced decomposition $w_0^{\bp}=s_{1}\cdots s_{l}$ of $w_0^\bp$. The  functors ${\sf H}_{w_0^{\bp}}$ and $\res^\bp_\fb\ind_\fb^\bp$  are identical, because 
    $\overline{\fb w_0^{\bp}\fb}/\fb=\bp/\fb$. So we now want to prove that ${\sf H}_{w_0^{\bp}}(P(-e_v))=P(-e_v)$. 
By Lemma \ref{lem:Hwstar}
we have ${\sf H}_{w_0^{\bp}}=\sH_{s_{1}}\circ\cdots\circ\sH_{s_{l}}$.
So it suffices to show that $\sH_s(P(-e_v))=P(-e_v)$ for $s\in W_\bp$ simple.
 Now $v$ is a minimal coset representative in $W/W_\bp$, so $vs>v$ for $s\in W_\bp$ simple.
Then $s v^{-1}>v^{-1}$, so $s v^{-1}w_0<v^{-1}w_0$ by Lemma \ref{lew0reverse}, and $s\star v^{-1}w_0=v^{-1}w_0$.
Thus $
{\sH}_{s}{\sH}_{v^{-1}w_0}(\sk_{-w_0ve_v})={\sH}_{v^{-1}w_0}(\sk_{-w_0ve_v})$ by Lemma \ref{lem:Hwstar}. And then 
$\sH_s(P(-e_v))={\sH}_{s}{\sH}_{v^{-1}w_0}(\sk_{-w_0ve_v})={\sH}_{v^{-1}w_0}(\sk_{-w_0ve_v})=P(-e_v)$.
 
Let  $e_w$ be $\bp$-dominant and  $\alpha\in \Pi^\bp$. Then $s_\alpha e_w\leq_d e_w$, so $s_\alpha e_w\geq_e e_w$ by Lemma \ref{lem:aed}. 
 First consider the case that $s_\alpha e_w= e_w$. 
 Now $we_w$ is dominant and $w^{-1}$ is minimal amongst the $z$ with $zwe_w=e_w$, so 
 then $s_\alpha w^{-1}>w^{-1}$.
Next consider the case that $s_\alpha e_w\neq e_w$. Again
we must have $s_\alpha w^{-1}>w^{-1}$,
as otherwise $s_\alpha e_w=s_\alpha w^{-1} we_w\leq_e w^{-1} we_w=e_w$.
 In either case we have $w<ws_\alpha $.

\end{proof}

\begin{lemma}\label{lem:indpreserve}Let $\lambda\in X^\bp_+$.
Let $N\in\rep_\fr(\fb)$ be such that all weights $\mu$ of $N$ satisfy $\mu\leq_a\lambda$.
Then all weights $\mu$ of $\ind_\fb^\bp(N)$ satisfy $\mu\leq_a\lambda$ and the natural map $\ind_\fb^\bp(N)\to N$ 
induces an isomorphism $\ind_\fb^\bp(N)_\lambda\to N_\lambda$.
\end{lemma}

\begin{proof}First let $\sk$ be a field.
Consider a reduced decomposition  $s_{1}\cdots s_{l}$ of $w_0^\bp$. 
Recall that $\res^\bp_\fb\ind_\fb^\bp=\sH_{s_{1}}\circ\cdots\circ\sH_{s_{l}}$. 
 Therefore we may assume $\bp$ is a minimal parabolic, say
$\bp=\bp_\alpha$. Thus $\bp/\fb$ is a projective line and we may use \cite[II Proposition 5.2]{Jan}.
If $\mu$ is a weight and $(\alpha^\vee,\mu)\geq0$,
then $s_\alpha \mu\leq_a \mu$, $\RR^1\ind_\fb^\bp(\mu)=0$, 
$(\ind_\fb^\bp(\mu))_\mu=\sk_\mu$ and all weights $\nu$ of $\ind_\fb^\bp(\mu)$ lie on the line segment joining $\mu $ with 
$s_\alpha \mu$. In particular, they satisfy $\nu\leq_a\mu$. And if $(\alpha^\vee,\mu)<0$,
then $\ind_\fb^\bp(\mu)=0$ and all weights $\nu$ of
$\RR^1\ind_\fb^\bp(\mu)$ lie strictly between the endpoints of the line segment, which implies $\nu<_a\mu$ for such $\nu$.
If $N$ has just weight $\lambda$, then it is  well known that $\ind_\fb^\bp(N)_\lambda\to N_\lambda$
is an isomorphism. 
Now use induction on the number of weights of $N$.  This settles the Lemma when $\sk$ is a field.

Moreover, if $\sk$ is a field, then $\RR^1\ind_\fb^\bp(\sk_\nu)_\lambda=0$ for $\nu \leq_a \lambda$ because $\RR^1\ind_\fb^\bp(Q(\nu))$ vanishes by 
Proposition \ref{prop:Schubert module}, so that
$\RR^1\ind_\fb^\bp(\sk_\nu)$
is a quotient of $\ind_\fb^\bp(Q(\nu)/\sk_\nu)$.
If $\sk=\Z$, observe that by the Universal coefficient Theorem $\RR^1\ind_\fb^\bp(N)_\lambda=0$. Then use 
the Universal coefficient Theorem once more to finish.
\end{proof}

\subsection{\unboldmath Generating \unboldmath$\rep( \bp )$}\label{subsec:repP-generation}

  \subsubsection*{\unboldmath The set \unboldmath$\{\hat M_v\}_{v\in W^\bp}$}
  For $\bp$-dominant $\lambda$ we use $\hat\nabla_\lambda$ as another notation for $\ind_\fb^\bp\sk_\lambda$. 
   Let us be given sets $\{\hat M_v\}_{v\in W^\bp}$,  $\{\hat N_v\}_{v\in W^\bp}$
   of objects of $\rep_\fr( \bp )$  with the following properties.
   For each ${v\in W^\bp}$  there is a diagram of $\bp$-modules
   $$\hat M_v\stackrel{g}\longrightarrow \hat N_v\stackrel{f}\longleftarrow \hat\nabla_{e_v} $$
   such that every weight $\mu$ of $\kernel(f)$, $\kernel(g)$, $\coker(f)$, $\coker(g)$ satisfies $\mu<_a e_v$.

  Examples of possible choices of $\hat M_v$ are $\hat\nabla_{{e_v}}$ with $f=g=\id$, $\ind_\fb^\bp Q(e_v)$ with $g=\id$, $P(-e_v)^*$
  with $f=\id$. 
  
  The multiplicity of the weight $e_v$ is one in $\hat M_v$, $\hat N_v$, $\hat\nabla_{e_v}$. 
  Every weight $\lambda$ of $\hat M_v$ satisfies $\lambda\leq_a e_v$.

\

  The following Theorem is similar to \cite[Theorem 2]{Ana}, which is proved directly in the context of $\bg$ and $\bp$. Our argument differs in that it refers back to the proof for the $\bp=\fb$ case.
  \begin{theorem}[Generation]\label{th:Pgeneration}
  The smallest strictly full additive subcategory of $\rep(\bp)$ that 
  \begin{itemize}
  \item contains the $\hat M_v$,  
  \item has the
  2 out of 3 property and 
  \item contains with every $\hat\nabla_\lambda$ also
   $\hat\nabla_\lambda\otimes\nabla_{\omega_\alpha}$ for every fundamental representation $\nabla_{\omega_\alpha}$,
  \end{itemize}
   is the category $\rep( \bp )$ of finitely generated $ \bp $-modules.

  \end{theorem}  
  
  More specifically
  
  \begin{theorem}[Generation of initial interval]\label{th:intervalgeneration}
  Let $\tau\in X^\bp_+$. Let $V\in\rep(\bp)$ such that every weight $\mu$ of $V$ satisfies $\mu<_a\tau$.
Then  $V$ is an object of
  the smallest strictly full additive subcategory $\tS$ of $\rep(\bp)$ 
  satisfying 
  \begin{itemize}
  \item $\tS$ contains $\hat M_v$ for $v\in W^\bp$ with $e_v<_a\tau$,  
  \item  $\tS$ has the
  2 out of 3 property,
 \item  If  $\tS$ contains $\hat\nabla_\lambda$, then $\tS$ also contains
   $\hat\nabla_\lambda\otimes\nabla_{\omega_\alpha}$, for every fundamental representation $\nabla_{\omega_\alpha}$. 
  \end{itemize}
  \end{theorem}  

\begin{proof}
We argue by induction along the well-ordered partial order $\leq_a$. So we assume the Theorem when $\tau$ is replaced by 
a $\sigma\in X_+^\bp$ with $\sigma<_a\tau$. Let $\sigma\in X_+^\bp$ with $\sigma<_a\tau$.
Thus $V\in \tS$ when all weights $\xi$
of $V$ satisfy $\xi<_a\sigma$.

Step 1. First we  wish to show that $\tS$ contains at least one $\bp$-module $N$ with $N_\sigma=\sk_\sigma$ and with weights $\xi$
that satisfy $\xi\leq_a\sigma$.
If $\sigma$ is a Steinberg weight $e_v$, then $v\in W^\bp$ by Lemma \ref{lem:PisP}
and we simply take $N=\hat M_v$.
If $\sigma$ is not a Steinberg weight,  then by Remark \ref{cor:if not Steinberg then tensor}
 there is a fundamental weight $\omega_\beta$ and a weight $\mu$ with $(\mu,\mu)<(\sigma,\sigma)$,
 so that the $\fb$-module $L=\nabla_{\omega_\beta}\otimes \sk_\mu$  has
$\sigma$ as a weight  of  multiplicity one, and so that all weights  $\nu$ of $L$ satisfy 
$\nu\leq_a\sigma$. Now take $N=\ind_\fb^\bp L=\nabla_{\omega_\beta}\otimes\ind_\fb^\bp \sk_\mu$. 
By Lemma \ref{lem:indpreserve} every weight $\nu$ of $N$ satisfies $\nu\leq_a\sigma$ and 
$N_\sigma=\sk_\sigma$. In particular, $\ind_\fb^\bp \sk_\mu$ is nonzero, so  $\mu\in X^\bp_+$ and $\ind_\fb^\bp \sk_\mu=\hat\nabla_\mu$.
And every weight $\nu$ of $\hat\nabla_\mu$ satisfies $(\nu,\nu)<(\sigma,\sigma)$, hence $\nu<_a\sigma$, so $\hat\nabla_\mu$  lies in $\tS$ and  
$N=\nabla_{\omega_\beta}\otimes\hat\nabla_\mu$ lies in $\tS$.

Step 2. Next we wish to show that $\hat\nabla_\sigma$ lies in $\tS$. To this end we look for more $\bp$-modules $N$ 
in $\tS$ with $N_\sigma=\sk_\sigma$ 
and such that the weights $\xi$ of $N$
satisfy $\xi\leq_a\sigma$.
Start with the $N$ from Step 1.
Let $N^1$ be the span of the weight spaces $N_\mu$ with $(\mu,\omega_\alpha)<(\sigma,\omega_\alpha)$ for at least one
$\alpha\in \Pi^{\notin \bp}$. Then $N^1$ is a $\bp$-submodule that lies in $\tS$. So we may replace $N$ with
$N/N^1$ and further assume $N^1=0$. 
Let $N^2$ be the span of the weight spaces $N_\mu$ with $(\mu,\omega_\alpha)\leq(\sigma,\omega_\alpha)$ for all
$\alpha\in \Pi^{\notin \bp}$. Then $N^2$ is a $\bp$-submodule and $N/N^2$ lies in $\tS$. So we may replace $N$ with
$N^2$ and further assume $N^1=0$, $N=N^2$. 
Let $N^3$ be the $\bp$-submodule of $N$ generated by $N_\sigma$.
Then $N/N^3$ is in $\tS$, so we may replace $N$ with $N^3$. Now the unipotent radical of $\bp$ acts trivially
on $N$ and the projection $p$ of $N$ onto its weight space $N_\sigma$ is $\fb$-equivariant. By Lemma \ref{lem:indpreserve} the map $\ind_\fb^\bp(p):N\to \hat\nabla_\sigma$ induced by $p$ has kernel and cokernel in $\tS$. So $\hat\nabla_\sigma$ lies  in $\tS$.

Step 3.
Finally we want to show that $\tS$ contains every  $\bp$-module $V$ all whose weights $\mu$ satisfy $\mu<_a\tau$.
Consider such a $V$.
We may and shall assume  that $\tS$ contains every  $\bp$-module
 whose set of weights is a proper subset of the set of weights of $V$.
Say $V$ is nonzero. Choose an extremal weight $\sigma$ of $V$ that is $\bp$-dominant.
Let $V^1$ be the span of the weight spaces $V_\mu$ with $(\mu,\omega_\alpha)<(\sigma,\omega_\alpha)$ for at least one
$\alpha\in \Pi^{\notin \bp}$. This is a $\bp$-submodule, and if $V^1$ is nonzero, then $V$ is in $\tS$ because both
$V^1$ and $V/V^1$ are. 
Let $V^2$ be the span of the weight spaces $V_\mu$ with $(\mu,\omega_\alpha)\leq(\sigma,\omega_\alpha)$ for all
$\alpha\in \Pi^{\notin \bp}$. This is a $\bp$-submodule, and if $V^2\not= V$, then $V$ is in $\tS$ because both
$V^2$ and $V/V^2$ are. 
So we further assume $V^1=0$ and $V=V^2$. Then  the unipotent radical of $\bp$ acts trivially
on $V$. Let $V^3$ be the $\bp$-submodule generated by $V_\sigma$.
Then $V/V^3$ has fewer weights, so $V/V^3\in \tS$. So if $V^3\in\tS$, then $V\in\tS$.
 So we may assume $V=V^3$. The weights $\xi$ of $V$ now
satisfy $\xi\leq_a\sigma$.
The projection $p$ of $V$ onto its weight space $V_\sigma$ is $\fb$-equivariant.
 We have  $\ind_\fb^\bp(V_\sigma)=\hat\nabla_\sigma\otimes (V_{\sigma})_\triv$, 
where $(V_{\sigma})_\triv$ is $V_{\sigma}$
with trivial $\bp$-action, cf.\ \cite[Proposition 17]{FvdK}. Consider the map $f:V\to \ind_\fb^\bp(V_\sigma)=\hat\nabla_\sigma\otimes (V_{\sigma})_\triv$, corresponding to $p$. Both $\kernel(f)$ and $\coker(f)$ lie in $\tS$. 
As $\hat\nabla_\sigma$ lies in $\tS$, so does $\hat\nabla_\sigma\otimes (V_{\sigma})_\triv$.
So $V$ lies in $\tS$.
\end{proof}

\begin{remark}
    In  Theorem \ref{th:Pgeneration}
    and
    Theorem \ref{th:intervalgeneration} we deal with $\rep(\bp)$,
    even over $\Z$, while Ananyevskiy deals with 
     $R(\bp)$ in \cite{Ana}. Of course it is easier to work in $R(\bp)$, but basically the arguments are the same and we get the same understanding of Steinberg weights and a similar constructive decomposition of elements of $R(\bp)$ over $R(\bg)$ as in the proofs of Ananyevskiy.
\end{remark}

\begin{definition}
For $v\in W^\bp$ put $\hat Q(e_v)=\ind_\fb^\bp Q(e_v)$. Then, for $p\in W^\bp$
\

\begin{itemize}

\item $\hat{\sf Q}_{\succeq p}:=\hull(\{\nabla_\lambda\otimes \hat Q(e_v)\}_{v\succeq p,  v\in W^\bp,\lambda\in X(\bt)_+})$

\item $\hat{\sf Q}_{\succ p}:=\hull(\{\nabla_\lambda\otimes \hat Q(e_v)\}_{v\succ p,  v\in W^\bp,\lambda\in X(\bt)_+})$

\item $\hat{\sf P}_{\preceq p}:=\hull(\{\nabla_\lambda\otimes P(-e_v)^*\}_{v\preceq p,  v\in W^\bp,\lambda\in X(\bt)_+})$

\item $\hat{\sf P}_{\prec p}:=\hull(\{\nabla_\lambda\otimes P(-e_v)^*\}_{v\prec p,  v\in W^\bp,\lambda\in X(\bt)_+})$

\end{itemize}
 \end{definition}

We get from Theorem \ref{th:Pgeneration} the following replacement of Theorem \ref{th:derived_generation} 

\begin{theorem}\label{th:Pderived_generation}
Given a $p\in W^\bp$, the triangulated hull in $\Dd(\Rep(\bp))$ of the two categories $\hat{\sf P}_{\preceq p}$ and 
$\hat{\sf Q}_{\succ p}$ 
is $\Dd^b (\rep ( P ))$. 
\end{theorem}

 We also have the following replacement of Theorem \ref{th:indPQ}
 
 \begin{theorem}\label{th:PindPQ}
Let $v,w\in W^\bp$.
\begin{enumerate}
\item $\RR^i\ind_\bp^\bg (P(-e_v)\otimes \hat Q(e_v))=\begin{cases}\sk &\text{if $i=0$,}\\
0&\text{else.}\end{cases}$
\item If $w\not\leq v$ then $\RR^i\ind_\bp^\bg (P(-e_v)\otimes \hat Q(e_w))=0$ for all $i$.
\end{enumerate}
\end{theorem}

\begin{proof}By the Universal coefficient Theorem \ref{th:universal coefficients}, it suffices to treat the case where $\sk$ is a field.
By Proposition \ref{prop:Schubert module} the module $Q(e_w)$ is $\ind_\fb^\bp$ acyclic. Lemma \ref{lem:PisP}, together with the Generalized Tensor Identity \cite[I Proposition 4.8]{Jan}, gives
$$\RR^j\ind_\fb^\bp (P(-e_v)\otimes Q(e_w))=P(-e_v)\otimes\RR^j\ind_\fb^\bp ( Q(e_w))=\begin{cases}P(-e_v)\otimes\hat  Q(e_w)&\text{if $j=0$,}\\
0&\text{else.}\end{cases}$$
So
$\RR^i\ind_\fb^\bg (P(-e_v)\otimes  Q(e_w))=\RR^i\ind_\bp^\bg (P(-e_v)\otimes \hat Q(e_w))$.
But the left hand side is known from Theorem \ref{th:indPQ}.
\end{proof}
\begin{proof}[\rm \bf Proof of Theorems \ref{th:Z-X^P_p-G-exceptional-coll},  \ref{th:Psemi-orthogonal}, \ref{th:Pcoh-semi-orthogonal}]

With these replacements in hand, we now proceed as in the case $\bp=\fb$. For $p\in W^\bp$, let $\hat X_p$ be the image of $P(-e_p)^*$ under the left adjoint of the inclusion of $\hat{\sf Q}_{\succeq p}$ 
into $\D=\Dd ^b(\rep( \bp))$. 
And let $\hat Y_p$ be the image of $\hat Q(e_p)$ under the 
right adjoint of the inclusion of
$\hat{\sf P}_{\preceq p}$ into $\D$. 

Let $f_p$ be the natural map from $P(-e_p)^*$ to $Q(e_p)$. Recall that $f_p$ factors as a surjection $P(-e_p)^*\to\sk_{e_p}$, followed by an injection $\sk_{e_p}\to Q(e_p)$. Let $\hat f_p:P(-e_p)^*\to \hat Q(e_p)$ be induced by $f_p$. Then $\hat f_p$ induces an isomorphism of $\bt$-modules 
$(P(-e_p)^*)_{e_p}\to (\hat Q(e_p))_{e_p}$ and by Lemma \ref{lem:indpreserve} the weights $\mu$ of $\kernel (\hat f_p)$, $\coker(\hat f_p)$ satisfy $\mu<_a e_p$.
We have an exact triangle 
$$ \ker(\hat f_p)[1]\to \cone(\hat f_p)\to \coker(\hat f_p)$$ in $\Dd ^b(\rep( \bp))$.
Take $\hat M_v=\hat Q(e_v)$ for $p\prec v\in W^\bp$ and $\hat M_v=P(-e_v)^*$ for  $p\succ v\in W^\bp$. It does not matter what we choose for $\hat M_p$ itself.

Using Theorem \ref{th:intervalgeneration} instead of Theorem \ref{th:cone generation}, we find that 
$\cone(\hat f_p)$ lies in the triangulated hull of $\{\hat M_v\otimes \nabla_\lambda \mid v\in W^\bp,\ e_v <_a e_p,\ \lambda\in X(\bt)_+\}$. 
So $\cone(\hat f_p)$ lies in the triangulated hull of $\hat{\sf P}_{\prec p}\cup\hat{\sf Q}_{\succ p}$.
One uses this to show that $\hat X_p$ and $\hat Y_p$ are isomorphic.
\emph{Mutatis mutandis} the old constructions and proofs go through and 
Theorems \ref{th:Z-X^P_p-G-exceptional-coll}, \ref{th:Psemi-orthogonal}, \ref{th:Pcoh-semi-orthogonal} follow.
\end{proof}
It remains to prove Theorem \ref{th:P and B}.

\begin{lemma}\label{lem:inducing Q}
Let $s$ be a simple reflection and $\bp_s$ the corresponding minimal parabolic. 
\begin{itemize}
\item If $s\lambda<_d\lambda$ then there is an exact sequence
$0\to Q(s\lambda)\to{\sH_s}(Q(\lambda))\to Q(\lambda)\to 0$,
\item If $s\lambda=\lambda$ then ${\sH_s}(Q(\lambda))=Q(\lambda)$,
\item If $s\lambda>_d\lambda$ then ${\sH_s}(Q(\lambda))=0$.
\end{itemize}
\end{lemma}
\begin{proof}
If $\sk=\Z$ then the first part is \cite[Lemma 7.2.3]{WvdKTata}. So the first part follows from Lemma \ref{lem:HwQk}.
Now let $s\lambda=\lambda$. We want to show that the natural map 
${\sH_s}(Q(\lambda))\to Q(\lambda)$ is an isomorphism.
By Lemma \ref{lem:HwQk} we may assume $\sk$ is a field.
By definition $Q(\lambda)$ is the kernel of a surjective map
 $P(\lambda)\rightarrow {H}^0(\partial X_w,\Ll (\lambda ^+))$, and we get by  \cite[Proposition 2.24]{WvdK} and Proposition \ref{prop:Schubert module} a relative Schubert filtration on 
 ${\sH}_s(Q(\lambda))$, with layers described by the extremal weights. 
The only extremal weight of ${\sH}_s(Q(\lambda))$ is the extremal
weight  $\lambda$ of ${\sH}_s(\sk_\lambda)$, because the weights $\nu$ of $Q(\lambda)$ that differ from $\lambda$ are too short to contribute, cf. proof of Lemma \ref{lem:indpreserve}. 

The last part follows in the same way, as ${\sH}_s(\sk_\lambda)$ vanishes when $s\lambda>_d\lambda$.
\end{proof}

\begin{lemma}\label{lem:ind of Q(e_w)}
Let $s$ be a simple reflection and 
let $w\in W$.
\begin{itemize}
\item If $se_w<_de_w$ then $ws\succ w$, $se_w=e_{ws}$, and there is an exact sequence
$$0\to Q(e_{ws})\to{\sH_s}(Q(e_w))\to Q(e_w)\to 0,$$
\item If $se_w=e_w$ then ${\sH_s}(Q(e_w))=Q(e_w)$,
\item If $se_w>_de_w$ then ${\sH_s}(Q(e_w))=0$.
\end{itemize}
\end{lemma}
\begin{proof}We take $\lambda=e_w$ in Lemma \ref{lem:inducing Q}. Recall that $w^{-1}$ is a minimal coset representative of the stabilizer in $W$ of the dominant weight $we_w$.
    If $se_w<_de_w$, then $se_w>_ee_w$ by Lemma \ref{lem:aed}, so $\ell(sw^{-1})=\ell(w^{-1})+1$ and $sw^{-1}$ is also a minimal coset representative of the stabilizer in $W$ of the dominant weight $we_w$.
    So $sw^{-1}\alpha>0$ for simple roots $\alpha$ perpendicular to  $we_w$.  
Now 
    consider a simple root $\beta$ that is not perpendicular to $we_w$. From the definition of $e_w$ it follows that $w^{-1}\beta<0$. Then also $sw^{-1}\beta<0$, because otherwise $sw^{-1}$  would make fewer positive roots negative than $w^{-1}$ does and 
    thus $\ell(sw^{-1})$ would be less than $\ell(w^{-1})$ by \cite[1.6, 1.7]{Humphreys}.
    It follows that $se_w$
    is the Steinberg weight $e_{ws}$. The other points are clear from Lemma \ref{lem:inducing Q}.
\end{proof}

\begin{lemma}\label{lem:indP of Q(e_w)}
Let $w\in W^\bp$.
Then there is a surjective map $\hat Q(e_w)\to Q(e_w)$ whose kernel lies in ${\sf Q}_{\succ w}$.
\end{lemma}

\begin{proof}
First let $\sk$ be a field.
Recall that there is a reduced decomposition  $w_0^{\bp}=s_{1}\cdots s_{t}$ of $w_0^\bp$ and that $\res^\bp_\fb\ind_\fb^\bp=\sH_{s_{1}}\circ\cdots\circ\sH_{s_{t}}$. 
Put $M_i=\sH_{s_{i}}\circ\cdots\circ\sH_{s_{t}}(Q(e_w))$ for $1\leq i\leq t$.
We claim that the maps $M_i\to Q(e_w)$ are surjective, with a kernel that has a relative Schubert filtration with layers of type $Q(e_v)$ with $v\succ w$, $e_v<_ae_w$. 
We will show this by descending induction on $i$.
As $w\in W^\bp$, the claim holds for $M_t$ by Lemma 
\ref{lem:ind of Q(e_w)}.
By Proposition \ref{prop:Schubert module} all these modules with a relative Schubert filtration are 
$\sH_{s}$-acyclic for a simple reflection  $s$.
Now assume the claim for $M_{i+1}$. From the $\sH_{s_{i}}$-acyclicity and
Lemma \ref{lem:ind of Q(e_w)} it follows that
the map $M_i=\sH_{s_{i}}(M_{i+1})\to \sH_{s_{i}}(Q(e_w))$ is surjective, with a kernel that
has a relative Schubert filtration with layers of type $Q(e_v)$ with $v\succ w$, $e_v<_ae_w$. 
Similarly the map 
$\sH_{s_{i}}(Q(e_w))\to Q(e_w)$ is surjective, 
because $w\in W^\bp$.
Its kernel also
has a relative Schubert filtration with layers of type $Q(e_v)$ with $v\succ w$, $e_v<_ae_w$. 
The kernel of the composite map
$M_{i}\to \sH_{s_{i}}(Q(e_w))\to Q(e_w)$ is an extension of the two kernels.
The claim follows for $M_i$.
For the case that $\sk$ is  a field the Lemma follows from the claim for $M_1$.
When $\sk=\Z$ one may use Lemma \ref{lem:Qexercise} to conclude that 
$\hat Q(e_w)$ has a relative Schubert filtration with the appropriate layers.  
The canonical map $\hat Q(e_w)\to  Q(e_w)$ is surjective at weight $e_w$ by Lemma \ref{lem:indpreserve}. By equation (\ref{eq:ExtQQ}) it annihilates the  $Q(e_v)$ with $e_v<_a e_w$. Now use equation (\ref{eq:ExtQQ}) repeatedly.
\end{proof}

\begin{lemma}\label{lem:inclusion of hat}
For $w\in W^\bp$ one has $\hat{\sf Q}_{\succ w}\subset {\sf Q}_{\succ w}$.
\end{lemma}
\begin{proof}
This follows by descending induction along $\succ$, using Lemma \ref{lem:indP of Q(e_w)}.
\end{proof}

\begin{proof}[\rm\bf Proof of Theorem \ref{th:P and B}] 
Take $p\in W^\bp$. 
That $\tX_p=\pi^*\hat\tX_p$ will be clear once we have that ${\hat{X}} _p$ equals ${{X}} _p$.
By equation (\ref{eq:def_of_Y_p}) we have 
\begin{equation}
    Y_p=i^{\sf P{\ !}}_{\preceq p}(Q(e_p)).
\end{equation}
 By Lemma \ref{lem:indP of Q(e_w)} we have 
\begin{equation}
i^{\sf P{\ !}}_{\preceq p}(Q(e_p))=i^{\sf P{\ !}}_{\preceq p}(\hat Q(e_p)).\end{equation}
As ${\sf cone}(\hat Y_p\rightarrow \hat Q(e_p))\in \hat{\sf P}_{\preceq p}^{\perp}=
\hat{\sf Q}_{\succ p}\subset  {\sf Q}_{\succ p}$, we also have \begin{equation}
i^{\sf P{\ !}}_{\preceq p}(\hat Q(e_p))=i^{\sf P{\ !}}_{\preceq p}(\hat Y_p).    
\end{equation} 
Further $\hat Y_p=\hat X_p\in \hat{\sf P}_{\preceq p}\subset {\sf P}_{\preceq p}$, so that 
\begin{equation}
i^{\sf P{\ !}}_{\preceq p}(\hat Y_p)=\hat Y_p.    
\end{equation}
Taken together, 
$Y_p=i^{\sf P{\ !}}_{\preceq p}(Q(e_p))=i^{\sf P{\ !}}_{\preceq p}(\hat Q(e_p))=i^{\sf P{\ !}}_{\preceq p}(\hat Y_p)=\hat Y_p$.
\end{proof}

\section{\unboldmath Dual exceptional collections}\label{subsec:dualFEC}
Let $\sk$ be a field.
Let $X$ be a smooth projective variety, and assume given a full exceptional collection $(E_0,\dots,E_n)$ in $\Dd ^{b}(X)$. By \cite{Ef} there is {a right dual exceptional collection} $(G_n,\dots, G_0)$ to $(E_0,\dots,E_n)$ and {a left dual exceptional collection} $(F_n,\dots, F_0)$ to $(E_0,\dots,E_n)$.
They are both full exceptional collections, and they are characterized by

\begin{proposition}\cite[Proposition 2.15]{Ef}\label{prop:dual_exc_coll_characterization}
Let $(E_0,\dots,E_n)$ be a full exceptional collection  in a $\sk$-linear 
triangulated category $\D$. The left dual exceptional collection $( F_n, \dots, F_0)$ is uniquely determined by the following property:
\begin{equation}
\Hom _{\D}(E_i,F _j[l]) = \left\{
   \begin{array}{l}
    {\sf k}, \quad {\rm for} \quad l=0, \  i=j,\\
    0, \quad {\rm otherwise}. \\
   \end{array}
  \right.
\end{equation}
Similarly, the right dual exceptional collection $( G_n,\dots, G _0) $ is uniquely determined by the following property:
\begin{equation}
\Hom _{\D}(G_i,E_j[l]) = \left\{
   \begin{array}{l}
    {\sf k}, \quad {\rm for} \quad l=0, \  i=j,\\
    0, \quad {\rm otherwise}. \\
   \end{array}
  \right.
\end{equation}
\end{proposition}

\subsection{\unboldmath Borel-Weil-Bott theorem}\label{th:BWB_van}

Define $$\overline C _{\mathbb Z}=\{\lambda\in X(\bt) \ | \ 0\leq ( \lambda + \rho,\beta ^{\vee}) \ {\rm for \ all} \ \beta \in \Phi^+\}$$ if ${\rm char}(k)=0$ and $$\overline C _{\mathbb Z}=\{\lambda\in X(\bt) \ | \ 0\leq ( \lambda + \rho,\beta ^{\vee})\leq p \ {\rm for \ all} \ \beta \in \Phi^+\}$$ if ${\rm char}(k)=p>0$.

\begin{theorem}\cite[II Corollary 5.5]{Jan}\label{th:Bott-Demazure_th}

\begin{itemize}

\item[(a)] If $\lambda \in \overline C _{\mathbb Z}$ with $\lambda \notin X(\bt)_+$, then 
$$
{H}^i({ \bg}/{\fb},\Ll ({w\cdot\lambda}))=0
$$
for all $w\in W$. Here $\cdot$ is the dot--action of the Weyl group $W$ on 
$X$.

\

\item[(b)] If $\lambda \in  \overline C _{\mathbb Z}\cap  X(\bt)_+$, then 
for all $w\in W$ and $i\in \mathbb N$

$$
{H}^i({ \bg}/{\fb},\Ll ({w\cdot \lambda}))={H}^0({ \bg}/{\fb},\Ll ({\lambda})),
$$

\

if $i=\ell(w)$ and otherwise ${\rm H}^i({ \bg}/{\fb},\Ll ({w\cdot \lambda}))=0$. Here $\ell(w)$ is the length function.

\end{itemize}

\end{theorem}

\subsection{\unboldmath  The left dual exceptional collection to the \unboldmath$X_v$'s, conjecturally}\label{sec:dualFEC_conj}

Recall the notation $\tP_v=\Ll(P(-e_v))$ and $\tQ_v=\Ll(Q(e_v))$ for $v\in W$.
Some evidence suggests that the following could be true:
\begin{equation}\label{eq:indPQ-rho} R\Gamma (\bg/\fb,\tP_v\otimes \tQ_{w_0v}\otimes \Ll (-\rho))= {\sf k}[-\ell(w_0v)].
\end{equation}
We have in any case, provided $0\in\overline C _{\mathbb Z}$,
\begin{equation}\label{eq:indStSt-rho}
 R\Gamma (\bg/\fb,\Ll(-e_v)\otimes \Ll(e_{w_0v})\otimes \Ll (-\rho))= {\sf k}[-\ell(w_0v)] .
\end{equation}
This uses that 
$-e_v+e_{w_0v}-\rho=(w_0v)^{-1}\rho-\rho$ and Theorem \ref{th:Bott-Demazure_th}.

There is also evidence that 
\begin{equation}\label{eq:QQ-rho}R\Gamma ({\bg/\fb},\tQ_v^{\ast}\otimes \Ll (-\rho)\otimes \tQ_w) \mbox{ vanishes unless } w_0 v \geq w,\end{equation}
and 
\begin{equation}\label{eq:QQ-rho=}R\Gamma ({\bg/\fb},\tQ_v^{\ast}\otimes \Ll (-\rho)\otimes \tQ_{w_0 v})={\sf k}[-\ell(w_0 v)],
\end{equation}
and\begin{equation}\label{eq:PP-rho} R\Gamma ({\bg/\fb}, \tP_v\otimes \Ll (-\rho)\otimes \tP_w^{\ast}) \mbox{ vanishes unless } w_0 v \leq w,\end{equation}
and\begin{equation}\label{eq:PP-rho=} R\Gamma ({\bg/\fb}, \tP_v\otimes \Ll (-\rho)\otimes \tP_{w_0 v}^{\ast})={\sf k}[-\ell(w_0 v)].\end{equation}

Assuming (\ref{eq:QQ-rho}), (\ref{eq:QQ-rho=}), (\ref{eq:PP-rho}), (\ref{eq:PP-rho=}), it is not difficult to see with equation (\ref{eq:X in Q and P}) that 
$\RHom_{\bg/\fb}(\tX_v\otimes \Ll (\rho), \tX_w)$ vanishes unless $w_0 v = w$.
And is not difficult to see
that $\RHom_{\bg/\fb}(\tX_{w_0 v}\otimes \Ll (-\rho),\tX_{v})={\sf k}[-\ell(w_0v)]$.

We can now state:

\begin{conjecture}\
The left dual exceptional collection to the $\tX_v$'s consists of the $\tX_{w_0v}\otimes\Ll(\rho)[-\ell(w_0v)]$.
\end{conjecture}

\end{document}